\documentclass[11pt]{amsart}
\usepackage{amscd,amssymb,graphics,color,a4wide,hyperref,mathtools}
\usepackage{graphicx}
\usepackage{psfrag} 
\usepackage{marvosym}
\usepackage{tikz}
\usepackage{multicol}
\usetikzlibrary{matrix}
\usetikzlibrary{mindmap,trees,calc}
\usepackage{textcomp}
\usepackage{xparse}

\usepackage{color}
\usepackage{mathrsfs}
\usepackage[all]{xy}

\DeclareFontFamily{U}{rsf}{}
\DeclareFontShape{U}{rsf}{m}{n}{
  <5> <6> rsfs5 <7> <8> <9> rsfs7 <10->  rsfs10}{}
\DeclareMathAlphabet{\mathscr}{U}{rsf}{m}{n}

\newtheoremstyle{citing}
  {}
  {}
  {\itshape}
  {}
  {\bfseries}
  {\textbf{.}}
  {.5em}
  {\thmnote{#3}}

{\theoremstyle{citing}
}

\newtheorem{theorem}[subsection]{Theorem}
\newtheorem{lemma}[subsection]{Lemma}
\newtheorem{proposition}[subsection]{Proposition}
\newtheorem{corollary}[subsection]{Corollary}

\theoremstyle{definition}
\newtheorem{definition}[subsection]{Definition}
\newtheorem{construction}[subsection]{Construction}
\newtheorem{example}[subsection]{Example}

\theoremstyle{remark}
\newtheorem{remark}[subsection]{Remark}

\numberwithin{equation}{section}

\usepackage{color}
\definecolor{forestgreen}{rgb}{0.13, 0.55, 0.13}

\newcommand{\NN} {\mathbb{N}}
\newcommand{\ZZ} {\mathbb{Z}}
\newcommand{\QQ} {\mathbb{Q}}
\newcommand{\RR} {\mathbb{R}}

\newcommand{\CC} {\mathbb{C}}

\newcommand{\PP} {\mathbb{P}}

\newcommand {\shL} {\mathcal{L}}
\newcommand {\shM} {\mathcal{M}}

\newcommand {\shX} {\mathcal{X}}
\newcommand {\shY} {\mathcal{Y}}
\newcommand {\shZ} {\mathcal{Z}}

\newcommand {\an} {\mathrm{an}}

\newcommand {\Aut} {\operatorname{Aut}}
\newcommand{\Amp}{\operatorname{Amp}}

\newcommand {\Bir} {\operatorname{Bir}}

\newcommand {\Bs} {\operatorname{Bs}}

\newcommand {\C} {\mathbb{C}}

\newcommand{\contr}{\operatorname{contr}}

\newcommand{\YP}{Y_\mathscr{P}}
\newcommand{\YT}{Y_\mathscr{T}}

\newcommand{\YYP}{\shY_\mathscr{P}}
\newcommand{\YYT}{\shY_\mathscr{T}}

\newcommand {\eps} {\varepsilon}

\newcommand {\End} {\operatorname{End}}

\newcommand {\Eff} {\operatorname{Eff}}

\newcommand {\id} {\operatorname{id}}

\renewcommand {\Im} {\operatorname{Im}}

\renewcommand {\ker } {\operatorname{ker}}

\newcommand {\Mov} {\operatorname{Mov}}

\newcommand {\NS} {\operatorname{NS}}
\newcommand {\Nef}{\operatorname{Nef}}
\newcommand{\NE}{\operatorname{NE}}
\renewcommand{\O} {\mathcal{O}}

\newcommand {\Pic} {\operatorname{Pic}}

\newcommand {\Proj} {\operatorname{Proj}}

\newcommand {\sing} {\mathrm{sing}}
\newcommand {\Sing} {\operatorname{Sing}}

\newcommand {\Spec} {\operatorname{Spec}}
\newcommand {\Spf} {\operatorname{Spf}}

\newcommand {\supp} {\operatorname{supp}}

\newcommand {\X} {\shX}
\newcommand {\Y} {\shY}

\def\mydate{\ifcase\month \or January\or February\or March\or
April\or May\or June\or July\or August\or September\or October\or 
November\or December\fi \space\number\day,\space\number\year}

\newlength{\picwidth} 
\newlength{\miniwidth} 
\newlength{\nanowidth} 
\newlength{\melowidth} 
\newlength{\leftminiwidth} 
\newlength{\rightminiwidth} 
\newlength{\minipagewidth} 

\let\oldeqref\eqref
\makeatletter
\RenewDocumentCommand\eqref{s m}{%
  \IfBooleanTF#1%
  {\textup{\tagform@{\ref*{#2}}}}
  {\oldeqref{#2}}
}
\makeatother

\newcommand {\Morifan}{\operatorname{MF}}

\newcommand{\BB}{{\mathrm{BB}}}

\newcommand{\opp}{{\mathrm{opp}}}
\newcommand{\rc}{{\mathrm{rc}}}
\newcommand{\nef}{{\mathrm{nef}}}
\newcommand{\conv}{{\mathrm{conv}}}
\newcommand{\Fghks}[1][2d]{{\sF_{#1}^{\scriptscriptstyle \mathrm{GHKS}}}}
\newcommand{\Fbb}[1][2d]{{\sF_{#1}^\mathrm{BB}}}
\newcommand{\Ftwod}[1][2d]{{\sF_{#1}}}

\newcommand{\Fcox}[1][2d]{{\sF_{#1}^{\scriptscriptstyle \mathrm{Cox}}}}
\newcommand{\FAET}[1][2d]{{\sF_{#1}^{\scriptscriptstyle \mathrm{AET}}}}
\newcommand{\mghks}[1][2d]{{\Sigma_{#1}^{\scriptscriptstyle \mathrm{GHKS}}}}

\newcommand{\mghksnef}[1][2d]{{\Sigma_{#1}^{\scriptscriptstyle \mathrm{GHKS, nef}}}}
\newcommand{\vf}[1][2d]{{\Sigma_{#1}^{\scriptscriptstyle \mathrm{Cox}}}}

\newcommand{\VD}{{\mathrm{VD}}}

\newcommand{\Gammabar}[1][2d]{{\ol\Gamma_{#1}^+}}
\newcommand{\Pbar}{{\ol P_{2d}^+}}
\newcommand{\Pbartwo}{{\ol P_{2}^+}}

\newcommand{\Cusp}{{\mathrm{Cusp}}}

\newcommand{\sC}{{\mathcal C}}
\newcommand{\sD}{{\mathcal D}}

\newcommand{\sF}{{\mathcal F}}

\newcommand{\sH}{{\mathcal H}}

\newcommand{\sL}{{\mathcal L}}
\newcommand{\sM}{{\mathcal M}}
\newcommand{\sN}{{\mathcal N}}
\newcommand{\sO}{{\mathcal O}}

\newcommand{\sX}{{\mathcal X}}
\newcommand{\sY}{{\mathcal Y}}
\newcommand{\sZ}{{\mathcal Z}}

\newcommand{\scrG}{{\mathscr G}}

\newcommand{\scrL}{{\mathscr L}}

\newcommand{\scrP}{{\mathscr P}}

\newcommand{\scrS}{{\mathscr S}}
\newcommand{\scrT}{{\mathscr T}}

\newcommand{\scrX}{{\mathscr X}}

\newcommand{\IP}{{\mathbb P}}
\newcommand{\Q}{{\mathbb Q}}
\newcommand{\R}{{\mathbb R}}
\newcommand{\IS}{{\mathbb S}}
\newcommand{\Z}{{\mathbb Z}}

\newcommand{\gothp}{{\mathfrak p}}

\newcommand{\gothL}{{\mathfrak L}}

\newcommand{\gothX}{{\mathfrak X}}
\newcommand{\gothY}{{\mathfrak Y}}

\newcommand{\abs}[1]{{\left|#1\right|}}

\newcommand{\chern}{{\rm c}}

\newcommand{\Def}{{\operatorname{Def}}}

\newcommand{\del}{\partial}

\newcommand{\Exz}{{\operatorname{Exc}}}

\renewcommand{\Im}{\operatorname{Im}}
\newcommand{\img}{\operatorname{im}}
\newcommand{\into}{{\, \hookrightarrow\,}}
\newcommand{\isom}{{\ \cong\ }}

\newcommand{\Nefe}{{\operatorname{Nef}^{{\operatorname{e}}}}}

\renewcommand{\O}{{\rm O}}

\newcommand{\ohne}{{\ \setminus \ }}

\newcommand{\ol}[1]{{\overline{#1}}}

\newcommand{\rat}[1][]{{\stackrel{\ #1\ }{\ratl}}}

\newcommand{\ratl}{\dashrightarrow}

\newcommand{\rk}{{\rm rk}}

\newcommand{\Sch}[1]{{{\rm Sch} / #1}}
\newcommand{\Set}{{\rm Set}}

\newcommand{\SO}{{\rm SO}}

\renewcommand{\to}[1][]{\xrightarrow{\ #1\ }}

\newcommand{\tensor}{\otimes}

\newcommand{\veps}{\varepsilon}
\newcommand{\vphi}{\varphi}
\newcommand{\vrho}{{\varrho}}

\newcommand{\wh}[1]{{\widehat{#1}}}
\newcommand{\wt}[1]{{\widetilde{#1}}}

\makeatletter
\newcommand*{\da@rightarrow}{\mathchar"0\hexnumber@\symAMSa 4B }
\newcommand*{\da@leftarrow}{\mathchar"0\hexnumber@\symAMSa 4C }
\newcommand*{\xdashrightarrow}[2][]{%
  \mathrel{%
    \mathpalette{\da@xarrow{#1}{#2}{}\da@rightarrow{\,}{}}{}%
  }%
}
\newcommand{\xdashleftarrow}[2][]{%
  \mathrel{%
    \mathpalette{\da@xarrow{#1}{#2}\da@leftarrow{}{}{\,}}{}%
  }%
}
\newcommand*{\da@xarrow}[7]{%
  \sbox0{$\ifx#7\scriptstyle\scriptscriptstyle\else\scriptstyle\fi#5#1#6\m@th$}%
  \sbox2{$\ifx#7\scriptstyle\scriptscriptstyle\else\scriptstyle\fi#5#2#6\m@th$}%
  \sbox4{$#7\dabar@\m@th$}%
  \dimen@=\wd0 %
  \ifdim\wd2 >\dimen@
    \dimen@=\wd2 %
  \fi
  \count@=2 %
  \def\da@bars{\dabar@\dabar@}%
  \@whiledim\count@\wd4<\dimen@\do{%
    \advance\count@\@ne
    \expandafter\def\expandafter\da@bars\expandafter{%
      \da@bars
      \dabar@ 
    }%
  }%
  \mathrel{#3}%
  \mathrel{%
    \mathop{\da@bars}\limits
    \ifx\\#1\\%
    \else
      _{\copy0}%
    \fi
    \ifx\\#2\\%
    \else
      ^{\copy2}%
    \fi
  }%
  \mathrel{#4}%
}
\makeatother

\newtheoremstyle{citing}
  {}
  {}
  {\itshape}
  {}
  {\bfseries}
  {\textbf{.}}
  {.5em}
  {\thmnote{#3}}

\theoremstyle{plain}

\theoremstyle{remark}

\theoremstyle{definition}

\numberwithin{equation}{section}

\theoremstyle{remark}

{\theoremstyle{citing}
}

\makeatletter
\newsavebox\myboxA
\newsavebox\myboxB
\newlength\mylenA

\newcommand*\xtilde[2][0.8]{%
    \sbox{\myboxA}{$\m@th#2$}%
    \setbox\myboxB\null
    \ht\myboxB=\ht\myboxA%
    \dp\myboxB=\dp\myboxA%
    \wd\myboxB=#1\wd\myboxA
    \sbox\myboxB{$\m@th\widetilde{\copy\myboxB}$}
    \setlength\mylenA{\the\wd\myboxA}
    \addtolength\mylenA{-\the\wd\myboxB}%
    \ifdim\wd\myboxB<\wd\myboxA%
       \rlap{\hskip 0.5\mylenA\usebox\myboxB}{\usebox\myboxA}%
    \else
        \hskip -0.5\mylenA\rlap{\usebox\myboxA}{\hskip 0.5\mylenA\usebox\myboxB}%
    \fi}

\newbox\usefulbox

\def\getslant #1{\strip@pt\fontdimen1 #1}

\def\xxtilde #1{\mathchoice
 {{\setbox\usefulbox=\hbox{$\m@th\displaystyle #1$}%
    \dimen@ \getslant\the\textfont\symletters \ht\usefulbox
    \divide\dimen@ \tw@ 
    \kern\dimen@ 
    \xtilde{\kern-\dimen@ \box\usefulbox\kern\dimen@ }\kern-\dimen@ }}
 {{\setbox\usefulbox=\hbox{$\m@th\textstyle #1$}%
    \dimen@ \getslant\the\textfont\symletters \ht\usefulbox
    \divide\dimen@ \tw@ 
    \kern\dimen@ 
    \xtilde{\kern-\dimen@ \box\usefulbox\kern\dimen@ }\kern-\dimen@ }}
 {{\setbox\usefulbox=\hbox{$\m@th\scriptstyle #1$}%
    \dimen@ \getslant\the\scriptfont\symletters \ht\usefulbox
    \divide\dimen@ \tw@ 
    \kern\dimen@ 
    \xtilde{\kern-\dimen@ \box\usefulbox\kern\dimen@ }\kern-\dimen@ }}
 {{\setbox\usefulbox=\hbox{$\m@th\scriptscriptstyle #1$}%
    \dimen@ \getslant\the\scriptscriptfont\symletters \ht\usefulbox
    \divide\dimen@ \tw@ 
    \kern\dimen@ 
    \xtilde{\kern-\dimen@ \box\usefulbox\kern\dimen@ }\kern-\dimen@ }}%
 {}}

\newcommand*\xoverline[2][0.75]{%
    \sbox{\myboxA}{$\m@th#2$}%
    \setbox\myboxB\null
    \ht\myboxB=\ht\myboxA%
    \dp\myboxB=\dp\myboxA%
    \wd\myboxB=#1\wd\myboxA
    \sbox\myboxB{$\m@th\overline{\copy\myboxB}$}
    \setlength\mylenA{\the\wd\myboxA}
    \addtolength\mylenA{-\the\wd\myboxB}%
    \ifdim\wd\myboxB<\wd\myboxA%
       \rlap{\hskip 0.5\mylenA\usebox\myboxB}{\usebox\myboxA}%
    \else
        \hskip -0.5\mylenA\rlap{\usebox\myboxA}{\hskip 0.5\mylenA\usebox\myboxB}%
    \fi}

\def\xxoverline #1{\mathchoice
 {{\setbox\usefulbox=\hbox{$\m@th\displaystyle #1$}%
    \dimen@ \getslant\the\textfont\symletters \ht\usefulbox
    \divide\dimen@ \tw@ 
    \kern\dimen@ 
    \overline{\kern-\dimen@ \box\usefulbox\kern\dimen@ }\kern-\dimen@ }}
 {{\setbox\usefulbox=\hbox{$\m@th\textstyle #1$}%
    \dimen@ \getslant\the\textfont\symletters \ht\usefulbox
    \divide\dimen@ \tw@ 
    \kern\dimen@ 
    \xoverline{\kern-\dimen@ \box\usefulbox\kern\dimen@ }\kern-\dimen@ }}
 {{\setbox\usefulbox=\hbox{$\m@th\scriptstyle #1$}%
    \dimen@ \getslant\the\scriptfont\symletters \ht\usefulbox
    \divide\dimen@ \tw@ 
    \kern\dimen@ 
    \xoverline{\kern-\dimen@ \box\usefulbox\kern\dimen@ }\kern-\dimen@ }}
 {{\setbox\usefulbox=\hbox{$\m@th\scriptscriptstyle #1$}%
    \dimen@ \getslant\the\scriptscriptfont\symletters \ht\usefulbox
    \divide\dimen@ \tw@ 
    \kern\dimen@ 
    \xoverline{\kern-\dimen@ \box\usefulbox\kern\dimen@ }\kern-\dimen@ }}%
 {}}
\makeatother

\makeatletter
\@namedef{subjclassname@2020}{
	\textup{2020} Mathematics Subject Classification}
\makeatother


\begin{document}

\title
[Compactification of degree two K3 surfaces]
{On the GHKS compactification of the moduli space of K3 surfaces of degree two}
\author{Klaus Hulek}
\author{Christian Lehn}
\author{Carsten Liese}
\address{Klaus Hulek\\Institut f\"ur Algebraische Geometrie, Leibniz Universit\"at Hannover, Welfengarten 1, 30167 Hannover,
Germany}
\email{hulek@math.uni-hannover.de}
\address{Christian Lehn\\ Fakult\"at f\"ur Mathematik\\ Technische Universit\"at Chemnitz\\
Reichenhainer Stra\ss e 39, 09126 Chemnitz, Germany}
\email{christian.lehn@mathematik.tu-chemnitz.de}
\address{Carsten Liese\\Institut f\"ur Algebraische Geometrie, Leibniz Universit\"at Hannover, Welfengarten 1, 30167 Hannover,
Germany}
\email{liese@math.uni-hannover.de}
\begin{abstract}
We investigate a toroidal compactification of the moduli space of K3 surfaces of degree $2$ originating from the program 
formulated by Gross-Hacking-Keel-Siebert. This construction uses Dolgachev's formulation of mirror symmetry and the birational geometry of the mirror family. Our main result in an analysis 
of the toric fan. For this we use the methods developed by two of us in a previous paper.
\end{abstract}
\maketitle

\tableofcontents
\section{Introduction}\label{section introduction}

We investigate a new toroidal compactification of the moduli space of K3 surfaces of degree $2$ inspired by mirror symmetry and examine its properties. This construction uses Dolgachev's formulation of mirror symmetry \cite{Dol96} and the birational geometry of the mirror family. It originates from the construction proposed by Gross-Hacking-Keel-Siebert \cite{GHKS} and our analysis relies on previous work by two of us \cite{HL}.

Moduli spaces of polarized K3 surfaces have been a center of interest ever since. In \cite{PSS71}, Pjatecki\v{\i}-\v{S}apiro and \v{S}afarevi\v{c} have shown that the moduli functor of degree $2d$ polarized K3 surfaces is coarsely representable. Thanks to their global Torelli theorem, the corresponding moduli space can be described as
\begin{equation}\label{eq f2d intro}
\sF_{2d}:= \Gamma_{2d} \backslash \sD_{2d}
\end{equation}
where $\sD_{2d}$ is the period domain and $\Gamma_{2d}$ is an arithmetic group acting on it; we refer to Section~\ref{section moduli spaces} for more details. A cusp of the Baily--Borel compactification $\sF_{2d}^\BB$ determines a certain cone and toroidal or semitoric compactifications of $\sF_{2d}$ are determined by (toric or semitoric) fans supported on these cones and equivariant for the action of $\Gamma_{2d}$. 

Let us assume that $2d$ is square free. In this case there is a unique zero dimensional cusp in $\sF_{2d}^\BB$ and there is a certain lattice $M_{2d}$ of signature $(1,18)$ such that the cone in question is a connected component $C_{2d}$ of the cone of positive vectors in $M_{2d,\R}:=M_{2d}\tensor \R$ or rather its rational closure $C_{2d}^\rc:=\conv\left(\ol{C_{2d}} \cap M_{2d}\right)$. Now a semitoric (respectively toroidal) compactification of $\sF_{2d}$ is determined by a semitoric (respectively toric) fan in $M_{2d,\R}$ whose support is $C_{2d}^\rc$ and which is equivariant for the action of a certain subgroup $\Gammabar \subset O(M_{2d})$. 

There is a canonical such semitoric fan, the Coxeter fan $\vf$, whose maximal dimensional cones are the fundamental domains for the Weyl group action on $C_{2d}^\rc$, see \cite{Vin75,Vin85,AET19}. The fan we consider in this paper is a refinement of the Coxeter fan obtained by the birational geometry of the mirror family to the moduli space $\sF_{2d}$, see Definition~\ref{definition cusp model GHKS fan}. We refer to it as the \emph{Gross--Siebert--Hacking--Keel fan} or \emph{GHKS fan}. Our first result is

\begin{theorem}[See Theorem~\ref{theorem fan and compactification}]\label{theorem semitoric}
Let $2d$ be an even square-free positive integer, let $\mghks$ be the GHKS fan in degree $2d$, and let $\Fghks$ be the associated semitoric compactification of $\sF_{2d}$. Then $\mghks$ is a refinement of the Coxeter fan $\vf$ and there is a bimeromorphic morphism
\[
\Fghks \to \Fcox.
\]
\end{theorem}

We build on and extend the investigation of the Dolgachev--Nikulin--Voisin family associated to a polarized K3 surface of degree $2$ over the complex numbers that has been initiated in \cite{HL}. 
It turns out that for  $2d=2$ the GHKS fan $\mghks[2]$ is indeed a fan. By construction, it is a refinement of the Coxeter fan $\vf[2]$ which in this case is also an honest toric fan. Alexeev, Engel, and Thompson constructed in \cite{AET19} a semitoric coarsening of $\vf[2]$. Altogether, this leads to the

\begin{corollary}
There is a sequence of morphisms 
\[
\Fghks[2] \to \Fcox[2] \to \FAET[2]
\]  
where $\Fghks[2]$ and $\Fcox[2]$ are toroidal and $\FAET[2]$ is  a semitoric compactification.
\end{corollary}

Our main result is

\begin{theorem}[See Corollary~\ref{corollary fan and compactification degree two}, Corollary~\ref{corollary morifan cusp degree 2}, Corollary~\ref{corollary morifan cusp degree 2 orbits}]\label{theorem main}
The toric fan $\mghks[2]$ has $31$ maximal cones inside the fundamental domain of the Coxeter fan.
The $\Gammabar[2]$-action induces a residual $S_3$-action on the set of these cones with $17$ orbits. 
\end{theorem}

The ultimate goal is to examine modularity of the compactification $\Fghks$, in particular for degree $2$ polarized K3 surfaces. 

\subsection{Strategy of the proof and outline of the paper}\label{section outline}
Let us outline the construction of the fan $\mghks$. The first observation is that the lattice $M_{2d}$ is precisely the Picard group of a very general  member in Dolgachev's mirror moduli space $\check{\sF}_{2d}$. 
This moduli space is one dimensional whereas $\sF_{2d}$ is $19$-dimensional.  Observe that the Picard rank of the very general element of $\check{\sF}_{2d}$ has rank $19$ as opposed to $1$ for the very general element of ${\sF}_{2d}$. 
To be independent of any choices we consider the mirror family over a neighborhood of the unique cusp of $\check{\sF}_{2d}$ --- this is a uniquely determined projective K3 surface $\scrS$ over $\C((t))$ which, following GHKS, we call the \emph{Dolgachev--Nikulin--Voisin} family or DNV family for short, see also \cite[Remark 1.17 and Definition 1.18]{HL}.

We can then realize $M_{2d}$ as the Picard group of $\scrS$ together with its intersection pairing. Next we consider certain degenerations $\sY \to S:=\Spec\C[[t]]$ with generic fiber $\sY_\eta=\scrS$, so-called \emph{models} of the DNV family. Recall that in \cite{HL}, the so-called Mori fan of $\sY$ was studied. Its cones are of the form $f^*\Nefe(\sY')$ where $f:\sY \ratl \sY'$ are rational contractions of $\sY$ over $S$, see Definition~\ref{definition rational contraction}, and $\Nefe$ denotes the effective nef cone, see \eqref{eq definition nefe}. Let us denote by $\iota:\sY_\eta \to \sY$ the inclusion of the generic fiber. The fan $\mghks$ is now roughly constructed in two steps:
\begin{itemize}
	\item The Mori fan with support on the movable cone $\Mov(\sY)$ is pulled back to $\Nef(\sY_\eta)$ via certain sections of the restriction $\iota^*$ constructed via birational geometry. This is done in Section \ref{section ghks fan}.
	\item The thus obtained fan $\Sigma'$ with support on the nef cone of $\sY_\eta$ cone is shown to be equivariant under the action of the subgroup of $\Gammabar$ preserving the nef cone. It will be extended to a fan on the whole positive cone using the Weyl group action, see Definition~\ref{definition cusp model fan}.
\end{itemize}

In Section~\ref{section moduli spaces}, we recall the basics on moduli spaces of K3 surfaces and their compactifications. Possibly only Section~\ref{section semitoric compactifications} is non-standard here and explains Looijenga's construction of semitoric compactifications. In Section~\ref{section degenerations}, we recall the theory of Friedman--Kulikov--Pinkham--Persson of degenerations of K3 surfaces and adapt it to our setting. In particular, we provide the necessary algebraization statements and analysis of the Picard groups. Section~\ref{section dnv family} is devoted to the study of the DNV family. We recall the Mori fan in Section~\ref{section morifan}. The central notion is that of a cusp model, introduced in Section~\ref{section cusp models}. These are certain birational morphisms $\Y \to \sY'$ and are used to define the GHKS refinement of the Coxeter semitoric fan and hence a semitoric compactification of $\Ftwod$ in Section~\ref{section ghks fan}. Sections~\ref{sec:cuspidal_cones} to~\ref{section counting cones} are devoted to the classification of cusp models. The idea is that the existence of a cusp model $\sY\to \sY'$ imposes strong restrictions on the Picard group of the components of the central fiber of $\sY$. While in Section~\ref{sec:cuspidal_cones} we recall the notion of a curve structure from \cite{HL} and obtain necessary conditions for the existence of cusp models, Sections~\ref{section class p} and~\ref{section class t} give a more detailed analysis depending on the dual intersection complex of the central fiber of $\sY$. 
The classification of cusp models and the counting of cones in degree $2$ is carried out in Section~\ref{section counting cones} where our main result, Theorem~\ref{theorem main}, is proven.

\subsection{Previous work}\label{section previous work}

This article clearly relies on the Gross--Hacking--Keel--Siebert construction suggested in  \cite{GHKS}.
Even though the fan we consider is not exactly the same as theirs, see Remark~\ref{remark not exactly ghks}, the construction is, of course, basically due to them. In the attempt to be self-contained, we prove some of the statements that one can also find in \cite{GHKS} and give a bit more details here and there, but no originality is claimed in the construction of the fan. We also decided to take a slightly different perspective: while GHKS use the construction to refine a given toric fan, we see it as a method to construct a semitoric fan in the sense of Looijenga \cite{Loo03,Loo03a}. This possibility has also been pointed out in \cite[Remark~0.11]{GHKS} but was not further pursued there. 
The new contribution of the present paper is  the detailed analysis of the degree $2$ case where the GHKS-construction gives an actual toric fan and thus a honest toroidal compactification.

\subsection*{Notation and terminology}

We will denote $S:= \Spec\C[[t]]$ and write $\eta$ for its generic point and $c$ for its closed point. A \emph{lattice} will be a torsion-free $\Z$-module of finite rank together with an integral non-degenerate symmetric bilinear form.

\subsection*{Acknowledgements}

We are grateful to Mark Gross, Paul Hacking, Sean Keel, and Bernd Siebert for sharing their unpublished manuscript \cite{GHKS}. 
The second named author would like to thank Ben Bakker, Philip Engel, and Luca Giovenzana for helpful discussions. We are grateful to Valery Gritsenko, Slava Nikulin, and Alessandra Sarti for answering our questions by email.

Klaus Hulek was partially supported by DFG grant Hu 337/7-1. 
Christian Lehn was supported by the DFG through the research grants Le 3093/2-2 and  Le 3093/3-1.

\section{Moduli spaces of K3 surfaces and their toroidal compactifications}\label{section moduli spaces}

In this section, we recall some basics about K3 surfaces, their moduli spaces, and compactifications. For K3 surfaces and their moduli spaces, we refer to Huybrecht's textbook \cite{Huy16}, in particular its Sections 5 and 6, and references therein. For compactifications we refer to \cite{AMRT} and \cite{Sca87}.

\subsection{K3 surfaces and their moduli spaces} \label{section k3 moduli}

The results stated in this section are well-known and we refer to \cite{BHPV} or \cite{Huy16} for proofs and additional references. Recall that a \emph{K3 surface} is a smooth compact complex surface with trivial canonical bundle $\sO_X \isom \omega_X$ and vanishing irregularity $h^1(X,\sO_X)$. A \emph{polarization} on a K3 surface $X$ is an ample divisor $H$ on $X$, a polarized K3 surface is a pair $(X,H)$ consisting of a K3 surface together with a polarization. If $H$ is only assumed big and nef, we refer to it as a \emph{quasi-polarization} and to the pair $(X,H)$ as a quasi-polarized K3 surface. The \emph{degree} of a polarization $H$ is the integer $H^2\in 2\Z$. Sometimes we relax the smoothness hypothesis and ask $X$ to have at most ADE-singularities. We speak of \emph{K3 surfaces with ADE-singularities} in this case. The notions of a polarization and a quasi-polarization extend to this setup.

Let $\sM_{2d}: (\Sch \C)^\opp \to \Set$ be the functor of degree $2d$ polarized K3 surfaces with ADE-singularities. Thanks to \cite{PSS71}, the moduli functor is coarsely representable. Pjatecki\v{\i}-\v{S}apiro and \v{S}afarevi\v{c} use the Global Torelli Theorem to show this, see also \cite{Vie90} for a different approach in the spirit of GIT. We refer to \cite{Huy16}, especially Sections~5 and 6, for more details and references.

\subsection{Periods of K3 surfaces} \label{section k3 periods}
For any lattice $L$ we define 
\begin{equation}\label{eq period domain}
\sD_L:= \{ x \in \IP(L \tensor \C) \mid (x,x) = 0, (x,\bar x) > 0 \}.
\end{equation}
Then $\sD_L$ is the period domain for weight $2$ Hodge structures on the lattice $L$ with $h^{2,0}=1$ such that the restriction of the pairing to the real space underlying $H^{2,0}\oplus H^{0,2}$ is positive definite and orthogonal to $H^{1,1}$. The geometry of the period domain, and in particular of the action of the orthogonal group $\O(L)$ on it, is very sensitive to the signature of the lattice. If $L$ has signature $(3,n)$, which is the case relevant for K3 surfaces, $\sD_L$ is connected and the group action has dense orbits and one does not have a reasonable quotient. It is worthwhile noting that precisely this ill-behaved group action can also be exploited to give a different proof of Torelli's theorem, see e.g. \cite[Theorem 1.1]{BL18}.

If $L$ has signature $(2,n)$, which is the case relevant for \emph{polarized} K3 surfaces, then $\sD_L$ is a hermitian symmetric domain of type IV and has two connected components. The quotient $O(L)\backslash \sD_L$ is a quasiprojective variety, and we will briefly discuss the vast theory of compactifications of this quotient beginning with Section~\ref{section bb compactification}. 

\subsection{Marked K3 surfaces and the Torelli theorem} \label{section torelli}
Let $X$ be a smooth K3 surface. Together with the intersection pairing, the group $H^2(X,\Z)$ is known to be isomorphic to the so-called \emph{K3 lattice}
\begin{equation}\label{eq k3 lattice}
\Lambda:= E_8(-1)^{\oplus 2} \oplus U^{\oplus 3}
\end{equation}
where $E_8(-1)$ stands for the negative definite root lattice of type $E_8$ and $U$ is the hyperbolic plane. The properties of the intersection pairing show that the weight $2$ Hodge structure of $X$ lies in $\sD_{H^2(X,\Z)}\isom \sD_\Lambda$. As the identification of the cohomology of the K3 surface with the K3 lattice $\Lambda$ is not canonical, one has to choose a \emph{marking}, i.e. an isometry $\mu:H^2(X,\Z)\to \Lambda$. A pair $(X,\mu)$ consisting of a K3 surface and a marking is called a \emph{marked K3 surface}.  Let $\sM_\Lambda$ be the \emph{marked moduli space}, i.e. the space of all isomorphism classes of marked K3 surfaces $(X,\mu)$, where isomorphisms have to be compatible with the markings. 

The Global Torelli Theorem was proven in \cite{PSS71} for algebraic K3 surfaces and for K\"ahler K3 surfaces in \cite{BR75}. After Verbitsky's proof \cite{Ver13} of the Global Torelli Theorem for irreducible symplectic manifolds, the following formulation has become popular:
\begin{theorem}\label{theorem torelli k3}
Consider the period map $\wp:\sM_\Lambda \to \sD_\Lambda$ and let $\omega \in \sD_\Lambda$. Then all $(X,\mu), (X',\mu') \in \wp^{-1}(\omega)$ satisfy $X \isom X'$. Moreover, the marked moduli space $\sM_\Lambda$ has two connected components and the restriction of the period map $\wp$ to such a component $\sN$ is surjective and injective over the complement of a countable union of hyperplanes.
\end{theorem}
We refer to \cite[Proposition 7.5.5]{Huy16} for the statement about the number of connected components.

\subsection{Moduli spaces as quotients of period domains} \label{section type four}

For all hermitian symmetric domains $\sD$ together with an arithmetic group $\Gamma \subset \Aut(\sD)$, the quotient $\Gamma\backslash \sD$ is a quasiprojective variety \cite{BB66}. We continue to denote by $\Lambda$ the K3 lattice from \eqref{eq k3 lattice}. As explained in Section~\ref{section k3 periods}, the period domain $\sD_\Lambda$ for K3 surfaces is \emph{not} hermitian symmetric, but the period domains for \emph{polarized} K3 surfaces are. Let $(X,H,\mu)$ be a polarized marked $K3$ surface, i.e. a marked K3 surface $(X,\mu)$ together with a polarization $H$. Then its period lies in the hyperplane 
$$\sD_{v^\perp} = \sD_\Lambda \cap \IP(v^{\perp}) \quad \textrm{where } v =\mu(H)\in \Lambda.$$

By \cite[Theorem]{Jam68} there is a unique $\O(\Lambda)$-orbit of primitive vectors $v$ with $v^2=2d$. We can therefore choose $v=e+df$ where $e,f$ is a basis of a $U$ summand in \eqref{eq k3 lattice} so that $v^\perp$ can be identified with 
\begin{equation}\label{eq lattice l2d}
L_{2d} :=  E_8(-1)^{\oplus 2} \oplus U^{\oplus 2} \oplus \left\langle -2d \right\rangle.
\end{equation}
Here $\left\langle m\right\rangle$ stands for a rank one lattice with generator of square $m$. As mentioned in Section~\ref{section k3 periods}, the period domain 
\[
\sD_{v^\perp}\isom \sD_{L_{2d}} = \sD_{2d} \amalg \sD_{2d}'
\]
has two connected components and we choose one of them. Accordingly, we denote by 
\[
\O^+(L_{2d}\tensor \R) \subset \O(L_{2d}\tensor \R) \quad \textrm{ and } \O^+(L_{2d}) \subset \O(L_{2d}) 
\]
the subgroups with real spinor norm $1$, which are the subgroups preserving $\sD_{2d}$. Using the Torelli theorem one can show that the moduli functor $\sM_{2d}$ is represented by 
\begin{equation}\label{eq f2d}
\Ftwod := \Gamma_{2d} \backslash \sD_{L_{2d}} = \Gamma_{2d}^+ \backslash \sD_{2d}
\end{equation}
where $\Gamma_{2d}\subset \O(\Lambda)$ is the subgroup fixing $v=e+df$, see \cite[Corollary~6.4.3 and Remark~6.4.5]{Huy16}, and $\Gamma_{2d}^+ = \O^+(\Lambda) \cap \Gamma_{2d}$. Here the action of $\Gamma_{2d}$ on $\sD_{L_{2d}}$ is through the canonical restriction homomorphism $\Gamma_{2d}\to \O(L_{2d})$. It can be shown that the restriction yields an isomorphism $\Gamma_{2d}\to \wt\O(L_{2d})$ with the finite index subgroup $\wt\O(L_{2d}) \subset \O(L_{2d})$ of elements acting trivially on the discriminant $L_{2d}^\vee/L_{2d}$. The group $\wt\O(L_{2d})$ is called the \emph{stable orthogonal group}. We will henceforth always refer to $\Ftwod$ as the \emph{moduli space of degree $2d$ K3 surfaces}. Keep in mind that it also parametrizes ADE-singular K3 surfaces as explained in Section~\ref{section k3 moduli}. 

\subsection{Baily--Borel compactification}\label{section bb compactification}

As all arithmetic quotients of hermitian symmetric domains, the moduli space $\Ftwod$ of polarized K3 surfaces has a canonical compactification, the Baily--Borel compactification $\sF_{2d}^\BB$, see \cite{BB66}. We also refer to Scattone's book \cite{Sca87} for a general reference concerning $\Fbb$. We will give a rough idea of how it is constructed. Let us consider the quadric $Q_{2d} \subset \IP(L_{2d}\tensor \C)$ cut out by the quadratic form and consider the closure of $\sD_{2d}$ inside $Q_{2d}$. Boundary components are the maximal subsets of the topological boundary $\del \sD_{2d}$ that are connected by analytic arcs. A boundary component $F$ is called rational, if its stabilizer subgroup $N(F)\subset \O^+(L_{2d}\tensor \R)$ is defined over $\Q$. The rational closure $\sD_{2d}^\rc$ is defined to be the union of $\sD_{2d}$ with all rational boundary components, endowed with the horocyclic topology. Then
\[
\Fbb:= \Gamma_{2d}^+\backslash \sD_{2d}^\rc.
\]
There is an alternative description as the $\Proj$ of a ring of modular forms, from which one immediately sees that it is projective. Moreover, $\Ftwod\subset\Fbb$ is a Zariski open, dense subset.

The (rational) boundary components for type IV domains turn out to be either of dimension one respectively zero, and they correspond to (rational) isotropic planes respectively lines in $L_{2d}$. Accordingly, also the complement $\sF_{2d}^\BB \ohne \sF_{2d}$ has a canonical stratification into so-called cusps which for type IV domains are either of dimension $1$ or $0$. 

\subsection{Toroidal compactifications}\label{section toroidal compactifications}

Apart from the Baily--Borel compactification, the moduli space $\Ftwod$ of polarized K3 surfaces has plenty of toroidal compactifications, see \cite{AMRT}. These depend on certain collections $\Sigma$ of fans similar to the fans in toric geometry and by construction map to the Baily--Borel compactification
\[
\beta:\sF_{2d}^\Sigma \to \Fbb.
\]

Let us be a bit more precise. For notational simplicity, we will abbreviate by $\Gamma:=\Gamma^+_{2d} \subset \O^+(L_{2d})$. 
For a (rational) boundary component $F$, recall that $N(F) \subset \O^+(L_{2d}\tensor \R)$ was the stabilizer. By $U(F) \subset N(F)$ we denote the center of its unipotent radical. Then $U(F) \isom \R^k$ for some $k$. Let us abbreviate $N(F)_\Gamma:=N(F)\cap \Gamma$, $U(F)_\Gamma:=U(F)\cap \Gamma$, and
\[
\ol\Gamma_F:=\Im\left(N(F)\cap \Gamma \to \Aut(U(F))\right).
\]

For every boundary component $F$, we consider $F/N(F)_\Gamma \subset \Fbb$. Then the preimage $\beta^{-1}(F)$ is the quotient of partial compactification of $\Gamma\backslash \sD$ inside a relative toric variety over a fiber bundle over $F$. This toric variety is determined by a fan whose support is the rational closure $C(F)^\rc$ of a certain cone $C(F)$ inside $U(F)$. A toroidal compactification is then determined by what we call a \emph{$\Gamma$-admissible collection of fans}\footnote{See \cite[III, Definition~5.1]{AMRT} where this is called a $\Gamma$-admissible collection of polyhedra.}
$$\Sigma:=\{\Sigma_F \mid F \textrm{ rational boundary component}\}$$
 where $\Sigma_F$ is a $\ol\Gamma_F$-fan\footnote{See \cite[II, Definition~4.10]{AMRT} where this is called a $\ol\Gamma_F$-admissible polyhedral decomposition. Sometimes we will also speak of a \emph{toric fan} instead of a fan, in particular if we want to emphasize the contrast to a semitoric fan, see Section~\ref{section semitoric compactifications}.} in $U(F)$ with support the rational closure $C(F)^\rc$. This means that $\Sigma$ has a certain equivariance for the group $\Gamma$ and that the group
$\ol\Gamma_F$ acts on each $\Sigma_F$ so that there are only finitely many cones up to the action of $\ol\Gamma_F$. 

As mentioned before, the boundary components for type IV domains have dimension $1$ or $0$ and correspond either to an isotropic plane $E$ or an isotropic line $\ell$ in $L_{2d}$. It turns out that $U(E)$ is one dimensional, see e.g. \cite[Lemma 2.25]{GHS07}, and there is the $\Gamma$-equivariance condition is automatic. On the other hand, $U(\ell)\isom \ell^\perp/\ell$ which is a hyperbolic lattice of signature $(1,18)$.

\subsection{The square free case}\label{section square free}
Let us assume from now on that $2d$ is square free. Then up to the action of $\Gamma_{2d}$, there is a unique isotropic line $\ell \subset L_{2d}$, see~\cite[Theorem~4.0.1]{Sca87}. We can therefore choose $\ell$ to be spanned by one of the basis vectors in one of the $U$ summands of $L_{2d}$ and identify $U(\ell)=\ell^\perp/\ell$ with the lattice
\begin{equation}\label{eq lattice m2d}
M_{2d}=E_8(-1)^{\oplus 2} \oplus U \oplus \left\langle -2d \right\rangle.
\end{equation}
The cone $C(\ell)$ is then realized as a connected component $C_{2d}$ of the cone of positive vectors in $M_{2d,\R}:=M_{2d}\tensor \R$ and its rational closure is
\begin{equation}\label{eq rational closure}
C_{2d}^\rc=\conv\left(\ol{C_{2d}} \cap M_{2d}\right).
\end{equation}
Let us denote $\Gammabar:=\ol{(\Gamma_{2d}^+)}_{\ell} \subset \O^+(M_{2d})$ the image of $N(\ell) \cap \Gamma_{2d}^+\to\O^+(M_{2d})$. Note that $\ol\Gamma_{2d}^+$ preserves the cone $C_{2d}$. As explained in Section~\ref{section toroidal compactifications}, the $\Gamma_{2d}^+$-compatibility with the fans $\Sigma_E$ corresponding to isotropic planes $E$ is automatic. Moreover, since there is only one orbit of isotropic lines, we may summarize our discussion as follows.
\begin{proposition}\label{proposition toroidal compactification}
A toroidal compactification of $\sF_{2d}$ for square free $2d$ is determined by a $\Gammabar$-fan in $M_{2d,\R}$ whose support is $C_{2d}^\rc$.
\end{proposition}

\subsection{Semitoric compactifications}\label{section semitoric compactifications}

 In \cite{Loo03}, \cite{Loo03a} Looijenga has introduced the notion of semitoric compactifications which generalizes both toroidal compactifications and the Baily--Borel compactification. Similarly to toroidal compactifications, these depend on an admissible collection of \emph{semitoric fans} and are also constructed relatively over the Baily--Borel compactification. Here we shall recall the basic notions necessary for this construction. Since we shall only apply this in the case of square-free $d$ we will restrict to this case, as
this allows us to keep the discussion considerably simpler.   

We start with a lattice $L$ of signature $(2,n)$ (which will later become $L_{2d}$) and an arithmetic group $\Gamma^+ \subset \O^+(L)$ (which will later become $\Gamma^+_{2d}$). We will make the assumption that 
$\Gamma^+$ operates transitively on the set of rational isotropic lines. Let $\ell$ be such an isotropic line and set $M:= U(\ell)=\ell^\perp/\ell$.  
Then $M$ (which will later be $M_{2d}$) is a hyperbolic lattice of signature $(1,n-1)$ and we denote by $M_{\RR}$ the associated real vector space. 
Let $C:=C(\ell)$ with rational closure $C^\rc$. The following definitions will be crucial. As described above the group $\Gammabar[]:=\ol{(\Gamma^+)}_{\ell} \subset \O^+(M)$ acts on  $C^\rc$.

\begin{definition}
By a {\em rational cone} we mean a rational polyhedral cone which is nondegenerate (i.e. does not contain an affine line).
A {\em rational cone system} in $C^\rc$ is a finite collection $\Sigma$  of  rational cones in $C^\rc$ such that
\begin{enumerate}
\item If $\sigma, \tau \in \Sigma$, then $\sigma \cap \tau \in \Sigma$
\item If $\sigma \in \Sigma$ and $\tau$ is a face of $\sigma$, then $\tau \in \Sigma$.  
\end{enumerate}
\end{definition}

\begin{definition}
A {\em locally rational cone} in $C^\rc$ is a cone $\sigma \subset C^\rc$  such that the restriction of $\sigma$ to any rational polyhedral subcone of $C^\rc$ is a 
rational cone.
\end{definition}

\begin{definition}\label{def:locallyrationalcollection}
A {\em locally rational decomposition of $C^\rc$  } is a collection $\Sigma$ of convex cones with the following properties: 
\begin{enumerate}
	\item $\cup_{\sigma \in \Sigma} \sigma = C^\rc$.
	\item The restriction of $\Sigma$ to any rational subcone of $C^\rc$ is a rational cone system.
\end{enumerate}
\end{definition}

If $h \supset \ell$ is a rational isotropic plane, then $J(h):=h/\ell$ is a rational isotropic line in $M_{\RR}$ (and every rational isotropic line in $M_{\RR}$ arises this way). We denote the half line of $J(h)$ which 
belongs to $C^\rc$ by $J^+(h)$.   

\begin{definition}\label{definition semitoric fan}
A {\em semitoric fan}  (for the group $\Gamma^+$) is a collection $\Sigma$ of cones which decompose the rational closure $C^\rc$  with the following properties:
\begin{enumerate}
\item $\Sigma$ is $\ol\Gamma^+$-equivariant.
\item The origin $\{0\}$ and the half lines $J^+(h)$, where $h \supset \ell$ is a rational isotropic plane are cones belonging to $\Sigma$. 
\item The cones of $\Sigma$ define a locally rational decomposition of  $C^\rc$.  
\end{enumerate} 
\end{definition}

This is Looijenga's definition adapted to the special case where we have only one isotropic rational line up to the action of the group $\Gamma^+$. If there are several such lines, then one must choose a decomposition
for all $C^\rc(\ell)$, or more intrinsically of the conical locus in the sense of Looijenga, and these must fulfill a (fairly complicated to formulate) compatibility condition. As we will not need this, we have omitted the details, which 
can be found in \cite[Section~6]{Loo03a}.

By Looijenga's theory every semitoric fan as above gives rise to a compactification of the quotient
\[
\sF = \Gamma^+ \backslash \sD^+. 
\]
We shall call the resulting compactification, which we will also denote by $\sF^{\Sigma}$ the {\em semitoric compactification} given by the semitoric fan $\Sigma$. This is a normal  compact complex space (possibly projective). 
Refinements of semitoric fans correspond to blow-ups of semitoric compactifications and all semitoric  compactifications have a natural map to the Baily--Borel compactification $\sF^{\mathrm{BB}}$ (as will become clear from 
the second example below). The following are the most important examples for us.
 
\begin{example}
Every fan for the group $\Gamma^+$ is also a semitoric fan for this group, in particular, every toroidal compactification is also a semitoric compactification.
\end{example} 

The Baily--Borel compactification itself is also a semitoric compactification:

\begin{example} 
The coarsest semitoric fan $\Sigma$ consists simply of  the cones $\{0\}$, the half lines $J^+(h)$ and $C^\rc$. This leads to the 
{\em Baily--Borel compactification} $\sF^{\mathrm{BB}}$.
\end{example} 

An important class of semitoric compactifications is defined via {\em hyperplane arrangements}, indeed, these were the motivating examples for Looijenga's work.

\begin{example}\label{example:Coxeter}
We consider a set $\sH$ of rational hyperplanes in $L_{\RR}$ of signature $(2,n-1)$ which has the property that it consists of only 
finitely many $\Gamma$-orbits (and thus is locally finite). If $H \in \mathcal H$ is a hyperplane containing our chosen isotropic line $\ell$, then $H/\ell$ defines a hyperplane in $M_{\RR}$. 
Intersecting the cone $C^\rc$  with these
hyperplanes and taking the closures of the  connected components defines a locally rational decomposition $\Sigma=\Sigma(\sH)$ of $C^\rc$.   
Special cases are:
\begin{enumerate}
\item Again, we note that for $\mathcal H=\emptyset$ we once again recover the {\em Baily--Borel compactification}. 
\item We can take the set of all roots $r\in L$ and the collection of hyperplanes $\mathcal H_{\operatorname{root}}:= \cup_rH_r$ where $H_r=r^{\perp}$. We shall call the resulting semitoric fan the {\em Coxeter semitoric fan}. 
\end{enumerate}
\end{example}

At this point we would like to make a comment on terminology. In the literature one finds both names Coxeter chambers and Vinberg chambers. Indeed, both authors have contributed decisively to this theory. In order to avoid 
confusion in the literature we have decided to stay with the name Coxeter fan as in \cite{AET19}. 

Applying the above discussion to the lattice $\Gamma_{2d}$ we obtain the semitoric fan analog of Proposition~\ref{proposition toroidal compactification}.

\begin{proposition}\label{proposition semitoric compactification}
A semitoric compactification of $\sF_{2d}$ for square free $2d$ is determined by a $\Gammabar$ semitoric fan in $M_{2d,\R}$ whose support is $C_{2d}^\rc$.
\end{proposition}

\section{Birational geometry of degenerations}\label{section degenerations}

In this section we provide basic facts about the birational geometry of models. Much of this is standard and we claim no originality for this.

As before we denote $S=\Spec\C[[t]]=\{c,\eta\}$. We will study degenerations of projective K3 surfaces over $S$. Before we get more specific, let us observe that such a degeneration always comes from an algebraic one.

\begin{proposition}\label{proposition general algebraization}
Let $\sX \to S$ be a flat, projective morphism of schemes. Then there exists a quasi-projective variety $C$, a projective morphism $\gothX \to C$, and a cartesian diagram
\[
\xymatrix{
\wh\sX \ar[r]\ar[d] & \gothX \ar[d]\\
\wh S \ar[r]^\vphi & C
}
\]
in the category of formal schemes where $\wh\sX\to\wh S:=\Spf \C[[t]]$ denotes the formal completion of $\sX \to S$ along the central fiber.
\end{proposition}
\begin{proof}
We fix an ample line bundle $L$ on the central fiber $X:=\sX_c$. By Grothendieck's existence theorem \cite{Gro95}, see also \cite[Theorem~2.5.13]{Ser06}, there is a deformation $\scrX \to \scrS$ of $X$ over a local noetherian scheme $(\scrS,c)$ and a line bundle $\scrL$ on $\scrX$ such that the pair $(\scrX\to \scrS,\scrL)$ is versal for $(X,L)$ at $c$. By Artin's result \cite[Theorem~1.6]{Art69}, the formal completion of $\scrX \to \scrS$ at $c$ is induced from a projective scheme $\gothX \to C$ as claimed.
\end{proof}

Observe that in the statement of the proposition, the scheme theoretic image of $\vphi$ need not be a curve in general.

\subsection{Kulikov models}\label{section kulikov}

A flat, projective morphism $\sY\to S$ from a normal scheme $\sY$ such that the generic fiber $\sY_{\eta}$ is a $K3$ surface will be called a \emph{degeneration of K3 surfaces}. Such a degeneration is called $K$-trivial if the canonical sheaf is the trivial line bundle: $\omega_{\sY}\isom \sO_X$. 

\begin{definition}\label{definition kulikov}
A \emph{Kulikov model} is a $K$-trivial degeneration of K3 surfaces $\sY\to S$ such that $\sY$ is a regular $3$-fold and the central fiber $\sY_c$ is a reduced normal crossing divisor.
\end{definition}  

The central fibers of such degenerations are classified by \cite{Per77,Kul77,PP81}\footnote{To be precise, these references treat degenerations over a disk $\Delta \subset \C$ in the analytic category. However, given $\sY\to S$ as above, one can always reduce to this situation by Proposition~\ref{proposition general algebraization}.} into type I, II, and III, see also \cite[\S 5]{Fri83} for more details and further references. Conversely, one may ask which normal crossing surfaces admit a smoothing to a K3 surface. For this purpose, Friedman introduced the notion of \emph{$d$-semistability}, which is defined as the triviality of the (intrinsically defined) infinitesimal normal bundle, see \cite[Definition~(1.13)]{Fri83}. It is easily seen to be satisfied by the central fiber of a Kulikov model and for the converse Friedman proved in \cite[Theorem~(5.10)]{Fri83} that $d$-semistable $K3$ surfaces always admit a smoothing. As we are most interested in Kulikov models whose central fiber is of type III, we recall the following definition.

\begin{definition}\label{definition typeiiik3}
A \emph{$d$-semistable $K3$ surface of type III} is a reduced, projective normal crossing surface $Y$ such that
\begin{enumerate}
	\item $Y$ is $d$-semistable,
	\item the dualizing sheaf $\omega_Y$ is trivial,
	\item the irreducible components are rational surfaces such that the preimage of the double curves form cycles of rational curves on the normalization, and
	\item the dual intersection complex of $Y$ is a triangulation of the $2$-sphere $\IS^2$.
\end{enumerate} 
\end{definition}

As we will not consider type I or II degenerations in this article, we will usually use the term \emph{Kulikov models} as a synonym to Kulikov model of type III.

\subsection{Algebraization}\label{section gaga}

One parameter smoothings $\sY\to S$ of a given $d$-semistable K3 surface are highly non-unique. One can obtain uniqueness (up to pullbacks along finite covers) by enforcing $\sY$ to have a high rank Picard group which leads to a strong algebraicity property of so-called maximal Kulikov models, see Proposition~\ref{proposition algebraic model}. 

\begin{definition}\label{definition maximal degeneration}
Recall that $\sY \to S$ is called \emph{maximal}, if the restriction homomorphism
\begin{equation}\label{eq restriction}
\Pic(\sY) \to \Pic(\sY_c)
\end{equation}
to the central fiber is an isomorphism. 
\end{definition}
Recall that a Kulikov model $\sY \to S$ is maximal if and only if $\sY_c$ is \emph{maximal} in the sense that its Carlson invariant is trivial. We refer to the discussion in \cite[p. 11]{HL} for more details and further references. We will see in Proposition~\ref{proposition picard group central fiber} that maximality is preserved under birational transformations. One can show that for a maximal Kulikov model one always has
\begin{equation}\label{eq picard rank maximal}
\vrho(\sY)=\vrho(\sY_c)= 18+n,
\end{equation}
where $n$ is the number of irreducible components of $\sY_c$, see e.g. \cite[Section~3.1]{Laz08}. 

By deformation theory, we can always realize a \emph{maximal} Kulikov model as a base change from a finite type scheme, more precisely:

\begin{proposition}\label{proposition algebraic model}
Let $\sY\to S$ be a maximal Kulikov model of type III. Then there is a diagram
\begin{equation}\label{eq diagram algebraization}
\xymatrix{
\sY \ar[r]\ar[d] & \gothY \ar[d]\\
 S \ar[r]^\vphi & C
}
\end{equation}
such that the following holds.
\begin{enumerate}
\item\label{proposition algebraic model item one} On the right, $\gothY$ is a smooth $3$-fold, and $C$ is a smooth, affine curve.
	\item\label{proposition algebraic model item two} The morphism $\gothY \to C$ is flat and projective, smooth over $C^*:=C\ohne\{\vphi(c)\}$, and for $t\neq \vphi(c)$ the fiber $\gothY_t$ is a K3 surface over $k(t)$.
\item\label{proposition algebraic model item three} The morphism $\vphi$ is \'etale at $c\in S$ and \eqref{eq diagram algebraization} is cartesian.
\item\label{proposition algebraic model item four} If $\sY_c = Y_1\cup \ldots \cup Y_n = \gothY_{\vphi(c)}$ are the irreducible components of the central fiber and $\eta_C\in C$ denotes the generic point, then the diagram 
\begin{equation}\label{eq picard groups}
\xymatrix{
0 \ar[r] & \sum_{i=1}^m \Z Y_i \ar[r]\ar@{=}[d] & \Pic(\gothY)/\Pic(C) \ar[r]\ar[d] & \Pic(\gothY_{\eta_C}) \ar[d]^\alpha\ar[r] & 0\\
0 \ar[r] & \sum_{i=1}^m \Z Y_i \ar[r] & \Pic(\sY) \ar[r] & \Pic(\sY_{\eta}) \ar[r] & 0\\
}
\end{equation}
has exact rows and the vertical morphisms (obtained by restriction) are isomorphisms.
\item\label{proposition algebraic model item five} For $t\in C$ the restriction $\Pic(\gothY) \to[\rho_t] \Pic(\gothY_t)$ induces a canonical injection 
\begin{equation}\label{eq canonical iso picard}
\bar\rho_t: \Pic(\sY_\eta) \to[\alpha^{-1}] \Pic(\gothY_{\eta_C}) \to[\isom] \Pic(\gothY)/(\Pic(C) + \sum_{i=1}^m \Z Y_i) \to  \Pic(\gothY_t)
\end{equation}
which for very general $t\in C$ is an isomorphism.
\end{enumerate}
\end{proposition}
\begin{proof}
This is a deformation theoretic argument using Grothendieck's existence theorem and work of Friedman-Scattone \cite{FS86}. The proof is easily obtained from the stronger result \cite[Proposition 1.13]{HL} where also more precise references can be found. Explicitly, the proof of the aforementioned result gives a cartesian diagram as in \eqref{eq diagram algebraization} which (after possibly shrinking $C$) satisfies \eqref*{proposition algebraic model item one} and \eqref*{proposition algebraic model item two} together such that for a choice of line bundles $L_1, \ldots, L_{19}$ on $\sY$ whose restrictions generate the Picard group of $\sY_\eta$, there are line bundles $\sL_1, \ldots, \sL_{19}$ on $\gothY$ such that $\left(\sL_i\right)\vert_{\gothY_c}$ is isomorphic to $(L_i)\vert_{\sY_c}$ on $\sY_c=\gothY_c$. This is where maximality is needed. Claim \eqref*{proposition algebraic model item three} follows by replacing $C$ with a finite cover and possibly shrinking it further, so let us prove the statements about the Picard group.

The horizontal morphisms in \eqref{eq picard groups} form short exact sequences if there are no nontrivial divisors on $\sY$ respectively $\gothY$ whose support is contained in a fiber different from the special fiber. This however holds, as the morphisms are smooth outside the special fibers and have connected fibers. By construction, $\alpha$ is surjective and by \cite[\href{https://stacks.math.columbia.edu/tag/0CC5}{Lemma 0CC5}]{SP20} it is also injective, hence \eqref*{proposition algebraic model item four} follows. The existence of the morphisms in \eqref{eq canonical iso picard} is clear. Let us denote by $\gothY^* \subset \gothY$ the complement of the special fiber. By the localization sequence, $\Pic(\gothY^*)/\Pic(C^*) \isom \Pic(\gothY_{\eta_C})$ and thus it follows from the sequence (4.11) in \cite{FS86} that $\bar\rho_t$ is injective and its image is saturated. As the $\chern_1(\sL_1),\ldots,\chern_1(\sL_{19})$ are linearly independent in $H^2((\gothY^*)^\an,\Z)$, the Picard rank of $\gothY_t$ is $\geq 19$ for every $t\in C^*$. As $\gothY \to C$ is a type III degeneration, for every $t\in C^*$ the classifying map $(C,t) \to \Def(\gothY_t)$ is not constant, hence maps onto the Hodge locus. By the geometry of the period domain, a very general period in this Hodge locus has Picard number 19 so that by saturatedness of the image, $\rho_t$ is also surjective for very general $t\in C^*$.
\end{proof}

Note that this statement is much stronger than the standard algebraicity statement coming from Grothendieck's existence theorem and Artin algebraicity result, cf. Proposition~\ref{proposition general algebraization}.

\begin{remark}\label{remark pic ns}
One could formulate the above statement equivalently in terms of the N\'eron--Severi group. In fact, the only non-trivial $\Pic^0$ in \eqref{eq picard groups} is the one of $C$. This is because for a variety $X$ the tangent space to $\Pic^0(X)$ is $H^1(X,\sO_X)$ which vanishes for all varieties above except $C$ above by Lemma~\ref{lemma properties cusp model}.
\end{remark}

The exact sequence \eqref{eq picard groups} crucially needs the regularity assumption (or actually factoriality). For a singular model we can however still compare the Picard group to the one of an algebraization. We use the following notation. A pointed scheme $(X,x)$ is a scheme $X$ together with a distinguished closed point $x \in X$. We write $f:(X,x)\to (Y,y)$ for a morphism of schemes with $f(x)=y$.

\begin{proposition}\label{proposition picard group algebraic model}
Let $(C,c)$ be a smooth algebraic curve and let $\gothp:\gothY \to \gothY'$ be a birational $C$-morphism of normal schemes projective over $C$ that is an isomorphism over $C\ohne\{c\}$. Given a flat morphism $\vphi: (S,c) \to (C,c)$, we consider the diagram of pullback squares
\[
\xymatrix{
\sY \ar[r]^\pi\ar[d]_j& \sY' \ar[r]\ar[d]_{j'}& (S,c) \ar[d]\\
\gothY \ar[r]^\gothp & \gothY' \ar[r]& (C,c).\\
}
\]
Then the canonical morphism
\begin{equation}\label{eq picard fiber product}
\Pic(\gothY') \to \Pic(\sY') \times_{\Pic(\sY)} \Pic(\gothY)
\end{equation}
is an isomorphism. In particular, the canonical morphism $\Pic(\gothY')/\Pic(C) \to \Pic(\sY')$ is an isomorphism in the situation of Proposition~\ref{proposition algebraic model}.
\end{proposition}
\begin{proof}
By normality of $\gothY$, the morphism \eqref{eq picard fiber product} is injective because its composition with the projection to $\Pic(\gothY)$ is. To prove surjectivity, let $L'$ be a line bundle on $\sY'$ and $\gothL$ be a line bundle on $\gothY$ such that $j^*\gothL = \pi^*L'$. As normality can be read off from the completion of a local ring, also $\sY$ and $\sY'$ are normal. Then ${j'}^*\gothp_*\gothL = \pi_*{j}^*\gothL = \pi_*\pi^*L' = L'$ by flat base change and normality. In particular, $\gothL':=\gothp_*\gothL$ is locally free of rank one along the central fiber $\gothY'_c$ and thus a line bundle satisfying ${j'}^*\gothL' = L'$. Moreover, the canonical morphism $\gothp^*\gothL' \to \gothL$ is an isomorphism over $C\ohne\{c\}$ and pulls back to an isomorphism under $j$. Thus, $\gothp^*\gothL' = \gothL$ and \eqref{eq picard fiber product} is surjective.
The last claim is now immediate from Proposition~\ref{proposition algebraic model}.
\end{proof}

We continue with the following simple observation.

\begin{lemma}\label{lemma algebraization of contractions}
Let $\sY \to S$ be a maximal Kulikov model and let $\gothY \to C$ be an algebraization of $\sY$ as in Proposition~\ref{proposition algebraic model}. Then the following hold.
\begin{enumerate}
	\item If $\gothp:\gothY \to \gothY'$ is a birational $C$-morphism of relative Picard rank $k$, then the base change of $\gothp$ to $S$ is a birational $S$-morphism of relative Picard rank $k$.
	\item Conversely, if $\pi:\sY \to \sY'$ is a birational $S$-morphism and $\sY' \to S$ is projective, then after possibly shrinking $C$ there is a birational $C$-morphism $\gothp:\gothY\to\gothY'$ where $\gothY'\to C$ is projective and $\pi$ is the base change of $\gothp$ to $S$. Again, the relative Picard ranks of $\pi$ and $\gothp$ coincide.
\end{enumerate}
\end{lemma}
\begin{proof}
In both cases, the statement about the Picard rank follows from Proposition~\ref{proposition picard group algebraic model}.
\begin{enumerate}
	\item It is clear that a $C$-morphism $\gothp:\gothY \to \gothY'$ induces an $S$-morphism $\pi:\sY \to \sY'$ by base change and that $\pi$ is birational if $\gothp$ is. 
\item Given $\pi:\sY \to \sY'$, we chose a very ample line bundle on $\sY'$ and denote $L$ its pullback to $\sY$. By \eqref{eq picard groups} again, there is a line bundle $\gothL$ on $\gothY$ whose pull back to $\sY$ is $L$. The set of points $t\in C$ where the restriction $\gothL_t$ to the fiber $\gothY_t$ is base point free is Zariski open and certainly contains the origin as $\gothL_0 = L_c$ on $\gothY_0=\sY_c$. Thus, after possibly shrinking $C$ we get a $C$-morphism $\gothp:\gothY \to \gothY'$ whose base change is $\pi$ and which therefore must be birational.
\end{enumerate}
\end{proof}

\begin{remark}\label{remark mmp} 
Proposition \ref{proposition algebraic model} and Lemma~\ref{lemma algebraization of contractions} allow us to treat maximal Kulikov models $\sY \to S$ as if they were finite type schemes. For example, we will make use of the MMP at various places without justifying it every time. This is possible as we may lift the Kulikov model to a morphism $\gothY \to C$ as in the proposition, perform MMP there, and then base change to obtain a Kulikov model again.
\end{remark}

We now want to study the birational geometry of models. We recall the definition of a birational contraction. Later we will also use the notion of a rational contraction due to Hu--Keel \cite[Definition~1.1]{HK} which generalizes it.
\begin{definition}\label{definition rational contraction}
Let $T$ be a scheme and $X, Y$ be normal projective $T$-schemes. A rational $T$-map $f:X \ratl Y$ is called a \emph{birational contraction} if its inverse does not contract any divisor. This can be reformulated as follows. If $X \xleftarrow{\ p\ } W \to[q] Y$ is a resolution of indeterminacy, i.e. $p,q$ are proper and $p$ is birational, then $f$ is a birational contraction if and only if every $p$-exceptional divisor is $q$-exceptional. A rational $T$-map $f$ is called a \emph{rational contraction} if, with the same notation every $p$ exceptional divisor is $q$-fixed, i.e. no effective Cartier divisor $D$ whose support is $p$-exceptional is $q$-moving. This means that the natural map $\sO_Y \to q_*\sO_W(D)$ is an isomorphism.
\end{definition}

\begin{lemma}\label{lemma flop is regular}
Let $\sY \to S$ be a maximal Kulikov model. If $\vphi:\sY \ratl \sY'$ is a birational contraction, then there is a maximal Kulikov model $\sX\to S$ and a diagram
\[
\xymatrix{
\sY \ar@{-->}[dr]_{\vphi} \ar@{-->}[rr]^f&& \sX\ar[dl]^{\psi} \\
&\sY' &\\
}
\]
over $S$ such that $f$ is a composition of flops and $\psi$ is a regular contraction.
\end{lemma}
\begin{proof}
For an ample divisor $A$ on $\sY'$, the pull back to $\sY$ is contained in the effective movable cone $\Mov(\sY)$, see \eqref{eq definition move}.
Therefore, there is a marked minimal model $f:\sY \ratl \sX$ over $S$ such that the pullback of $A$ to $\sX$ is nef, see \cite[Theorem 2.3]{Kaw97}. This implies that $\sX \ratl \sY'$ is actually regular by Kawamata's base point free theorem. Moreover, it follows from the proof of loc. cit. that $f$ is a composition of flops. 

A priori $\sX$ only has $\Q$-factorial terminal singularities, but by \cite[Theorem 2.4]{Kol89} (see also \cite[Theorem 6.15]{KM98}) flops preserve regularity so that $\sX$ is smooth as well. Clearly, maximality is preserved under flops so that $\sX\to S$ is also a maximal Kulikov model.
\end{proof}

\begin{lemma}\label{lemma properties cusp model}
Let $\sY \to S$ be a maximal Kulikov model and let $\Y \ratl \sY'$ be a birational contraction. Then the following holds:
\begin{enumerate}
	\item\label{lemma properties cusp model item one} $K_{\sY'}=0$ and $\sY'$ has canonical, hence rational singularities.
	\item\label{lemma properties cusp model item two} $\sY'$ and the central fiber $\sY'_c$ are Cohen Macaulay.
	\item\label{lemma properties cusp model item three} We have $h^i(\sY'_c,\sO_{\sY'_c})=\begin{cases} 1 & \textrm{ for } i=0,2,\\ 0& \textrm{ for } i=1.\\ \end{cases}$
	\item\label{lemma properties cusp model item four} Base change holds, i.e. $H^i(\sY',\sO_{\sY'}) \tensor_{\C[[t]]} \C \to[\isom] H^i(\sY'_c,\sO_{\sY_c'})$, 	and moreover $$H^i(\sY',\sO_{\sY'})=\begin{cases} \C[[t]] & \textrm{ for } i=0,2,\\ 0& \textrm{ for } i=1.\\ \end{cases}$$
\end{enumerate}
\end{lemma}
\begin{proof} By Lemma \ref{lemma flop is regular} we may assume $\sY \to \sY'$ to be regular. The morphism $\sY \to \sY'$ is an isomorphism over the regular locus of $\sY'$, so $K_{\sY'}=0$. In particular, $K_{\sY'}$ is Cartier. Consequently, $\sY \to \sY'$ is crepant and, in particular, $\sY'$ has canonical singularities. Canonical singularities are rational by \cite{Elk81} and rational singularities are Cohen Macaulay, see \cite[Theorem 5.10]{KM98}. Being an effective Cartier divisor in $\sY'$ also $\sY'_c$ is Cohen Macaulay. As it is connected and reduced, we have $H^0(\sO_{\sY'_c})=\C$ and from Serre duality we infer $H^2(\sO_{\sY'_c})=\C$. Finally, $H^1(\sO_{\sY'_c})=0$ follows as the generic fiber is a K3 surface and the Euler characteristic is constant in families. The last statement follows now from \eqref{lemma properties cusp model item three} and base change, see e.g. \cite[Theorem III.12.11]{Har77}.
\end{proof}

Let $\sY \to S$ be a maximal Kulikov model and let $f:\sY\ratl \sY'$ be a birational contraction. Lemma~\ref{lemma algebraization of contractions} and Lemma~\ref{lemma flop is regular} imply that there is an algebraization $\gothY \ratl \gothY' \to (C,c)$. We denote by $\Delta \subset C^\an$ a small disk centered at $c$ and by ${\gothY'}^\an_\Delta$ the base change to $\Delta$ of the analytification.

\begin{proposition}\label{proposition picard group central fiber}
Let $\sY \to S$ be a maximal Kulikov model and let $f:\sY\ratl \sY'$ be a birational contraction. Then the following holds:
\begin{enumerate}
\item\label{item one proposition picard group central fiber} If $\gothY' \to (C,c)$ is an algebraization, the morphism 
$$\Pic(\gothY')/\Pic(C) \to \Pic\left({\gothY'}^\an_\Delta\right), \quad L\mapsto L^\an \tensor_{\sO_C}\sO_\Delta$$ is an isomorphism. In particular, there is a canonical isomorphism $\Pic(\sY')\to \Pic\left({\gothY'}^\an_\Delta\right)$.
	\item\label{item two proposition picard group central fiber} The degeneration $\sY'\to S$ is maximal, i.e. the restriction 
$$\Pic(\sY')\to \Pic(\sY'_c)$$
is an isomorphism.
\item\label{item three proposition picard group central fiber} If $f:\sY\to \sY'$ is a morphism, the pullbacks
\[
f^*: \Pic(\sY')\to \Pic(\sY) \qquad \textrm{ and } \qquad  f^*: \Pic(\sY_c')\to \Pic(\sY_c) 
\]
are injective of the same corank.
\end{enumerate}
\end{proposition}

\begin{proof}
We will prove all statements at the same time. First we bound the size of $\Pic\left({\gothY'}^\an_\Delta\right)$. The fibers $\gothY'_t$ for $t\neq c$ are contractions of K3 surfaces and therefore satisfy $H^1(\gothY'_t,\sO_{\gothY'_t})=0$. Moreover, $H^1(\gothY'_c,\sO_{\gothY'_c})=0$ by Lemma~\ref{lemma properties cusp model}. By GAGA, the same vanishing holds true after analytification and therefore $H^1({\gothY'}^\an_\Delta,\sO_{{\gothY'}^\an_\Delta})=H^0({\gothY'}^\an_\Delta,R^1 h_*\sO_{{\gothY'}^\an_\Delta})=0$ by base change where $h:{\gothY'}^\an_\Delta \to \Delta$ is the projection. From the exponential sequence we thus obtain the diagram	

\begin{equation}\label{eq small exponential diagram}
\xymatrix{
H^1({\gothY'}^\an_\Delta,\sO^\times_{{\gothY'}^\an_\Delta}) \ar@{^(->}[r] \ar[d]  & H^2({\gothY'}^\an_\Delta,\Z_{{\gothY'}^\an_\Delta}) \ar[r] \ar[d]^{{\iota'}^*}  & H^2({\gothY'}^\an_\Delta,\sO_{{\gothY'}^\an_\Delta}) \ar[d] \\
H^1({\gothY'}^\an_c,\sO^\times_{{\gothY'}^\an_c}) \ar@{^(->}[r] & H^2({\gothY'}^\an_c,\Z_{{\gothY'}^\an_c}) \ar[r] & H^2({\gothY'}^\an_c,\sO_{{\gothY'}^\an_c})\\
}
\end{equation}
where the left horizontal maps are injections. The morphism ${\iota'}^*$ is an isomorphism thanks to Lojasiewicz's theorem \cite{Loj64}, see also \cite[Theorem I.8.8]{BHPV}. This implies that the restriction 
$$\Pic({\gothY_\Delta'}^\an) = H^1({\gothY'}^\an_\Delta,\sO^\times_{{\gothY'}^\an_\Delta}) \to H^1({\gothY'}^\an_c,\sO^\times_{{\gothY'}^\an_c}) = \Pic({\gothY'}^\an_c)$$ is injective. The analogous statements hold for $\gothY$.

Thanks to Lemma~\ref{lemma flop is regular} we may assume that $f:\sY\to\sY'$ is a morphism. Because of Proposition~\ref{proposition algebraic model}, we may replace $\Pic(\sY)$ by $\Pic(\gothY)/\Pic(C)$ and $\Pic(\sY_c)$ by $\Pic(\gothY_c)$ everywhere and the same for $\sY'$ by Proposition~\ref{proposition picard group algebraic model}. Let us consider the diagram
\begin{equation}\label{eq picard gaga diagram}
\xymatrix{
\Pic(\gothY')/\Pic(C) \ar[rr]^\eta \ar[dd]_\zeta\ar[dr]^\delta & & \Pic({\gothY_\Delta'}^\an) \ar@{^(->}[dd]|(0.49)\hole_(0.3)\vartheta\ar[dr]^\beta&\\
& \Pic(\gothY)/\Pic(C)\ar[rr]\ar[dd]^(0.3)\alpha && \Pic({\gothY_\Delta}^\an)\ar@{^(->}[dd]\\
\Pic(\gothY_c')\ar[rr]|(0.533)\hole^(0.3)\isom \ar[dr]^\eps&& \Pic({\gothY_c'}^\an) \ar[dr]^\gamma&\\
& \Pic(\gothY_c)\ar[rr]^\isom && \Pic({\gothY_c}^\an)\\
}
\end{equation}
where all maps are given by pullback along the canonical morphisms of ringed spaces. The two lower horizontal maps are isomorphisms by GAGA, the two right vertical maps are injective by the observations made so far. By maximality, $\alpha$ is an isomorphism and thus the front square consists of isomorphisms. We will see below that $\beta$ and $\gamma$ are injective. The map $\delta$ is injective by normality of $\gothY'$ and injectivity of $\eps$ follows from the bottom square admitting that $\gamma$ is injective. It follows from the diagram that $\zeta$ and hence $\eta$ are injective. 

We will show next the surjectivity of $\vartheta$ and injectivity of $\beta$ and $\gamma$. For this we consider the following diagram of cohomology groups induced by the exponential sequence. 
\begin{equation}\label{eq exponential diagram}
\xymatrix{
\Pic({\gothY'}_\Delta^\an) \ar@{^(->}[r] \ar@<-8mm>@/_2pc/[ddd]_\beta  \ar[d]_\vartheta  & H^2({\gothY'}_\Delta^\an,\Z_{{\gothY'}_\Delta^\an}) \ar@<-8mm>@/_2pc/[ddd]_{s} \ar[r] \ar[d]^{{\iota'}^*}  & H^2({\gothY'}_\Delta^\an,\sO_{{\gothY'}_\Delta^\an}) \ar[d] \ar@<8mm>@/^2pc/[ddd]^t\\
\Pic({\gothY'}^\an_c) \ar@{^(->}[r] \ar[d]_\gamma  & H^2({\gothY'}^\an_c,\Z_{{\gothY'}^\an_c}) \ar[r] \ar[d]  & H^2({\gothY'}^\an_c,\sO_{{\gothY'}^\an_c}) \ar[d] \\
\Pic({\gothY}_c^\an) \ar@{^(->}[r] & H^2({\gothY}_c^\an,\Z_{{\gothY}_c^\an}) \ar[r]^u & H^2({\gothY}_c^\an,\sO_{{\gothY}_c^\an}) \\
\Pic({\gothY}_\Delta^\an) \ar[u]_\isom \ar@{^(->}[r] & H^2({\gothY}_\Delta^\an,\Z_{{\gothY}_\Delta^\an}) \ar[r]^v \ar[u]_{{\iota}^*}   & H^2({\gothY}_\Delta^\an,\sO_{{\gothY}_\Delta^\an}) \ar[u]\\
}
\end{equation}
As above, injectivity of the first column of horizontal morphisms follows from Lemma \ref{lemma properties cusp model} and bijectivity of $\iota^*, {\iota'}^*$ from Lojasiewicz's theorem. Invoking Lemma \ref{lemma properties cusp model} once more, we conclude that $\sY'$ (and hence also $\gothY'$ and ${\gothY'}^\an_\Delta$) has rational singularities. Now a standard argument, see e.g. \cite[Lemma 2.1]{BL16}, shows that we have $R^1{f_\Delta^\an}_*\Z_{{\gothY_\Delta'}^\an}=0$ and hence $s$ is injective. This immediately implies injectivity of $\beta$ and $\gamma$. The kernel of $u$ is identified with the kernel of $v$ under $\iota^*$. Rationality of singularities implies that $t$ is an isomorphism and we obtain surjectivity of $\vartheta$ by a diagram chase. 

Now one shows literally as in Proposition~\ref{proposition picard group algebraic model} that the top square is cartesian, in particular, $\eta$ is an isomorphism. This implies \eqref{item one proposition picard group central fiber}. From \eqref{eq exponential diagram}, we deduce that also $\zeta$ is an isomorphism, so \eqref{item two proposition picard group central fiber} follows.
Putting everything together, we also obtain \eqref{item three proposition picard group central fiber}.
\end{proof}

\section{The DNV family and the Mori fan}\label{section dnv family}

In this section we introduce the DNV family in Definition~\ref{definition dnv} and review some of its basic properties. Afterwards in Section~\ref{section morifan} we recall the construction of the Mori fan from \cite{HL}, which will be crucial for the rest of the paper.

Let us briefly recall the notion of primitivity. Given a Kulikov model $\sY \to S$, we can always consider an analytic family $\gothY \to \Delta$ with the same central fiber, e.g. by analytifying the algebraic deformation from Proposition~\ref{proposition general algebraization}. By \cite[Theorem (0.5)]{FS86}, the monodromy action $T$ on the cohomology $H^2(\gothY_t,\Z)$ of a nearby fiber of $\gothY \to \Delta$ of this degeneration only depends on the central fiber and the logarithm $N:=\log T$ is integral, that is, $N$ is an endomorphism of $H^2(\gothY_t,\Z)$. Following Friedman-Scattone, we call $\sY\to S$ \emph{primitive}, if $N$ is primitive as a vector in $\End_\Z\left(H^2(\gothY_t,\Z)\right)$. 

It follows from \cite[Proposition~1.15]{HL} that the generic fiber of a primitive maximal Kulikov model $\sY \to S$ of a $d$-semistable K3 surface is uniquely defined up to isomorphism. The existence statement of \cite[Proposition~1.16]{HL} ensures that the following definition makes sense.

\begin{definition}\label{definition dnv} 
Fix a square free integer $d>0$. The generic fiber $\sY_\eta$ of a primitive maximal Kulikov model $\sY \to S$ of a $d$-semistable K3 surface is called the \emph{Dolgachev--Nikulin--Voisin family (DNV family) of degree $2d$} if $\Pic(\sY_\eta)$ is isomorphic to the lattice $M_{2d}$ from \eqref{eq lattice m2d}. We refer to $\sY\to S$ as a \emph{model of the DNV family} in this case. 
\end{definition}

The DNV family is thus a $K3$ surface $\sY_{\eta}$ over the field $\CC((t))$. If we simply speak of the {Dolgachev--Nikulin--Voisin family}, it is understood that we mean the Dolgachev--Nikulin--Voisin family of degree $2d$ for some square free $d$. 

\begin{corollary}\label{corollary flop is regular}
If in Lemma~\ref{lemma flop is regular} the morphism $\sY \to S$ is a semistable model of the DNV family, then $\sX\to S$ is so as well.
\end{corollary}
\begin{proof}
Clearly, maximality and primitivity are preserved under flops so that $\sX\to S$ is also a model of the DNV family by \cite[Proposition~1.15]{HL}.
\end{proof}

We recall the construction of these models, see \cite[\S 1.2]{HL} for more details.

\begin{example}\label{example dnv family}
Recall that there is a bijection between triangulations of the sphere $\IS^2$ such that no vertex has valency greater than 6 and maximal $d$-semistable K3 surfaces of type III in so-called $(-1)$-form, see \cite[\S 5.1]{Laz08} and \cite[\S 3.9]{FS86}. 
Under this bijection, vertices correspond to irreducible components of the surface whose normalizations are weak del Pezzo surfaces of degree $s$ equal to the valency of the vertex, and the double locus of the central fiber gives rise to an anticanonical cycle on the normalizations of the components. For $s \leq 6$ there is a unique such anti-canonical pair $(\gothY_s, D_s)$ which is a component of a $d$-semistable 
$K3$ surface with trivial Carlson map, see \cite[Construction 1.23]{HL}, and also  
\cite[Proposition~5.2 and Lemma~5.14]{Laz08}.

There are precisely two different triangulations of $\IS^2$ with two triangles, and they correspond to dual intersection 
 complexes of degenerations with central fibers having three components. We denote by $\scrP$ the triangulation given by two triangles glued along the boundary and by $\scrT$ the one given by two triangles glued along one side to each other, see Figure \ref{ZETA}. The corresponding surfaces $\YP$ and $\YT$ look like this: $\YP$ is obtained by gluing three copies of $\mathfrak{Y}_2$ by identifying  components of the boundary cycle, and $\YT$ is a copy of $\mathfrak{Y}_4$, with two opposite components, say ${D}_1$ and ${D}_3$, of the anticanonical curve ${D}$ identified, and two copies of $\mathfrak{Y}_1$ glued to the (images) of ${D}_2$ and  $D_4$. We refer to \cite[Construction~1.23]{HL} for details.

The associated maximal smoothings $\YYP\to S$ and $\YYT\to S$ are models of the Dolgachev family in degree $2$. 

\begin{figure}\centering
\begin{tikzpicture}[scale=0.75]
\draw[]    (0,1) -- ++(0:1.5) node(C1)[below]{f}--++(0:1.5)  
-- ++(120:1.5) node(A1){}--++ (120:1.5) node(top1){} --++ (240:1.5)node(B1)[left]{e} --cycle; 
\draw[] (top1.center)--++(0:1.5) node(C2)[above]{e}--++(0:1.5) --++(-120:1.5) node[right]{f}--++(-120:1.5);
\draw node at (1.75,0) {$\mathscr{P}$};

\draw[]    (6,1) -- ++(0:1.5) node(C3)[below]{f}--++(0:1.5)  
-- ++(120:1.5) node(A3){}--++ (120:1.5) node(top2){} --++ (240:1.5)node(B3)[left]{f} --cycle; 
\draw[] (top2.center)--++(0:1.5) node(C4)[above]{e}--++(0:1.5) --++(-120:1.5) node[right]{e}--++(-120:1.5);
\draw node at (7.75,0) {$\mathscr{T}$};
\end{tikzpicture}\caption{The triangulations $\mathscr{P}$ and $\mathscr{T}$.}
\label{ZETA}
\end{figure}
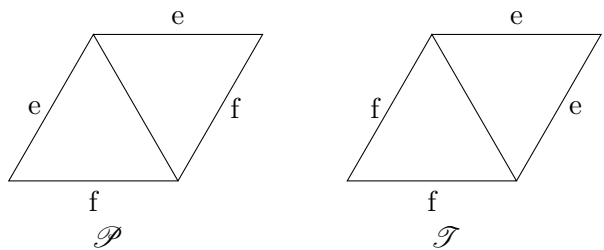
\end{example}

For later use we record

\begin{lemma}\label{lemma irreducible components}
Let $\sY \to S$ be a regular model of the DNV family of degree $2d$. Then the number of irreducible components of the central fiber $\sY_c$ is $d+2$.
\end{lemma}
\begin{proof}
This is proven e.g. in \cite[Proposition 6.3]{theta15}. We will give an independent argument. If $n$ is the number of components of $\Y_c$, then the number of triple points is $t=2n-4$. This follows from $t- e + n = 2$ where $e$ is the number of edges in the dual complex and the obvious relation $3t=2e$. Comparing the Picard lattices, it follows from \cite[Proposition 1.13]{HL} that $-2n+4=-2d$ so that $n=d+2$.
\end{proof}

\subsection{The Mori fan}\label{section morifan}

Let $\sX\to S=\Spec \C[[t]]$ be a projective morphism of a normal threefold $\sX$. Let us first discuss the various cones of interest inside $\NS(\sX)_\R$. As the Picard group of $S$ is trivial, the notion of relative and absolute ampleness for line bundles on $\sX$ coincide and similarly for nef line bundles. We denote by 
$$\Nef(\sX) \subset \NS(\sX)_\R$$ 
the nef cone of $\sX$, that is, the closure of the ample cone. We further denote by $\Eff(\sX)$ the convex cone generated by all effective line bundles and by
\begin{equation}\label{eq definition nefe}
\Nefe(\sX):=\Nef(\sX)\cap \Eff(\sX) \subset \NS(\sX)_\R
\end{equation}
the \emph{effective nef cone}. A line bundle $L$ is \emph{movable} if its base locus $\Bs(L) \subset \sX$ has codimenison at least $2$. We denote by $\ol{\sM}(\sX)$ the \emph{closed movable cone} defined as the closure of the  convex cone generated by movable line bundles and by $\Eff(\sX)$ the convex cone generated by all effective line bundles. The cone of interest to us will be
\begin{equation}\label{eq definition move}
\Mov (\shX)=\ol{\shM}(\shX)\cap\Eff(\shX).
\end{equation}
to which we refer as the \emph{effective movable cone}, sometimes also somewhat imprecisely just as the \emph{movable cone}.

We can now define the Mori fan.

\begin{definition}\label{definition morifan}

Let $\sY\to S$ be flat projective morphism from a $\Q$-factorial normal scheme $\sY$ with $\dim\sY=3$ whose generic fiber is the DNV family and whose relative canonical sheaf is the trivial line bundle. For a rational map $f:\sY \ratl \sY'$, we put $C(f):=f^*\Nefe(\sY')\subset \NS(\sY)$. Then the \emph{Mori fan} of $\sY$, denoted by $\Morifan(\sY)$, is defined as
\begin{equation}\label{eq definition morifan}
\Morifan(\sY):=\{C(f)  \mid f:\sY \ratl \sY' \textrm{ is a rational contraction}\}.
\end{equation}
Note that if $f$ is the map to a point, then $C(f)=\{0\}$.
\end{definition}

This fan was first considered by Hu and Keel for Mori dream spaces, see \cite[1.11~Proposition]{HK}. In our context, the definition is due to Gross--Hacking--Keel--Siebert, see \cite[Section 6]{GHKS}. 

\begin{remark}\label{remark maximal cones}
It follows readily from the definition of the Mori fan that its maximal cones are of the form $C(f)$ where $f:\sY \ratl \sY'$ is a small $\Q$-factorial modification. Recall that $f\colon \sY\ratl \sY'$ is called a \emph{small $\Q$-factorial modification} (over $S$) if $\sY'$ is $\Q$-factorial and $f$ is a birational $S$-map and an isomorphism in codimension one. 
\end{remark}

\begin{remark}\label{remark kovacs}
We emphasize that the nef cone is in general strictly larger than its subcone generated by effective nef divisors. However, for K3 surfaces $\sX$ containing curves of negative self intersection, the cones $\Nef(\sX)$ and $\Nefe(\sX)$ coincide by \cite[Corollary~1]{Kov94}.
\end{remark}

From \cite[Theorem~6.5]{GHKS}, we infer:

\begin{proposition}\label{proposition morifan}
The support of $\Morifan(\sY)$ is $\Mov(\sY)$ and $\Mov(\sY)$ is the rational closure of its interior.\qed
\end{proposition}

\begin{remark}\label{remark morifan}\

\begin{enumerate}
\item The fact that one can run an MMP for $\sY \to S$ implies that a small modification  $f \colon \sY\dashrightarrow \sY'$ factors into flops. Hence, $\sY'\to S$ is also a model of the DNV family.
	\item Similarly to \cite[Lemma 1.5]{Kaw97} one shows that if $f:\sY \ratl \sY'$ and $g:\sY \ratl \sY''$ are birational contractions such that $C(f)\cap C(g)$ has a point that is interior to both $C(f)$ and $C(g)$, then there is an isomorphism $\beta\colon\sY''\to \sY'$ with $f=\beta \circ g$,
and hence $C(f)=C(g)$.
\item Note that if $\sY'\to S$ is another regular model of the DNV family, then every birational map $\sY \ratl \sY'$ yields an identification of $\Morifan(\sY')$ with $\Morifan(\sY)$. We refer to the discussion of Section~2.2 of \cite{HL} for the fact that $\Morifan(\sY)$ is indeed a fan. This follows  from the fact that the lattice $M_{2}$ is a reflexive lattice, i.e. the quotient $\O^+(M_2)/W_2$ of the orthogonal group of $M_2$ preserving the positive cone by the Weyl group is 
finite. This was shown by Nikulin in \cite{Nik2el}. 
\item One can show that nef line bundles on $\Y$ are semi-ample, see \cite[Remark~2.6]{HL}.
\item The cone $\Nef(\sY_\eta)$ is finitely polyhedral if and only if $d=1$, see e.g. \cite[Remark~2.7]{HL}. 
\end{enumerate}
\end{remark}

\section{Cusp models}\label{section cusp models}

In this short section, we introduce cusp models, prove their existence, and introduce cuspidal cones which will be central in the definition of the GHKS compactification in the next section.

\begin{definition}\label{definition cusp model}
Let $f:\sY \to S$ be a model of the DNV family. A \emph{rational cusp model} for $f$ is a flat projective morphism $f':\sY'\to S$ together with a birational contraction $\vphi:\sY \ratl \sY'$ over $S$ such that the following holds:
\begin{enumerate}
	\item The scheme $\sY'$ is normal with $\Q$-factorial singularities, 
	\item the map $\vphi$ is an isomorphism over the generic point of $S$, and 
	\item the central fiber $\sY'_c$ of $f'$ is irreducible.
\end{enumerate}
A \emph{cusp model} is a rational cusp model such that $\vphi:\sY \to \sY'$ is a morphism. In both cases, if $Y \subset \sY_c$ is the unique component which is not contracted, we refer to $\sY'$ as a \emph{(rational) cusp model} for $Y$ or a \emph{$Y$-cusp model}. 
\end{definition}

Note that a cusp model is not a model of the DNV family in the terminology of Definition~\ref{definition dnv} since the total space of a model needs to be regular. Sometimes we will refer somewhat imprecisely to $\sY'$ itself as a cusp model. 
Lemma~\ref{lemma flop is regular} in particular applies to rational cusp models $\sY \ratl \sY'$. 
The following is a direct consequence of Proposition~\ref{proposition picard group central fiber}.

\begin{lemma}\label{lemma picard rank cusp model}
Let $f:\sY \to S$ be a model of the DNV family. If $\vphi:\sY \ratl \sY'$ is a rational cusp model, then $\sY'$ has Picard number $19$ and the restriction $\Pic(\sY') \to \Pic(\sY_\eta)$ to the generic fiber is an isomorphism. If $\vphi$ is a morphism, also the restriction $\Pic(\sY') \to \Pic(\sY'_c)$ to the special fiber is an isomorphism.
\end{lemma}
\begin{proof}
As $\vphi$ is a contraction, the pullback $\vphi^*$ on (rational) Picard groups is injective. The Picard rank of $\sY$ is $19 +n$ if $n+1$ is the number of irreducible components of the central fiber $\sY_c$. Let $D$ be a divisor on $\sY$. As $\sY'$ is $\Q$-factorial, $\vphi^*\vphi_*D$ coincides with $D$ up to exceptional divisors. The claim follows because by definition of a cusp model there are exactly $n$ distinct exceptional divisors.
\end{proof}

\begin{remark}
Notice that unlike in \cite{theta15}, we ask cusp models to be $\QQ$-factorial. This turned out to be convenient; it is e.g. crucial for Lemma \ref{lemma picard rank cusp model}. 
\end{remark}

Recall from Remark \ref{remark mmp} that we have the MMP for Kulikov models at our disposal -- just as for ordinary threefolds over a curve.

\begin{proposition}
Let $\sY\to S$ be a model of the DNV family of degree $2d$ and let $Y$ be an irreducible component of $\sY_c$. Then there is a rational $Y$-cusp model $\sY\dashrightarrow \sY'$.
\end{proposition}
\begin{proof}
We consider the $\R$-divisor $\Delta:=\veps (\sY_c-Y)$ for a sufficiently small $\veps >0$ such that the pair $(\sY,\Delta)$ is klt. Note that $K_{\sY/S}+\Delta=K_{\sY}+\Delta=\Delta$. We run a $(K_{\sY/S}+\Delta)$-log MMP over $S$ and obtain a sequence of rational $S$-maps
\[
\sY =\sX_0 \rat[\phi_1] \sX_1 \rat[\phi_1] \ldots \rat[\phi_{N}] \sX_N=:\sY',
\]
with the following properties. All schemes $\sX_i$ are $\Q$-factorial, and if we define $\Delta_0:=\Delta$ and inductively $\Delta_i:={\phi_{i-1}}_*\Delta_{i-1}$ for all $i$, then the pairs $(\sX_i,\Delta_i)$ are klt. In addition, the divisor $\Delta_N$ is nef. 

We claim that $K_{\X_i/S}=0$ and that $\phi_{i+1}$ is either a $\Delta_i$-flip or a divisorial contraction whose exceptional locus is a divisor contained in the support of $\Delta_i$ for all $i=1,\ldots,N$. 
Inductively, we may assume this to be the case for all $j<i$ where $i$ is fixed. By MMP, there is a $\Delta_{i-1}$-negative extremal ray $R$ and a contraction $c_R:\sX_{i-1} \to Z$ which 
contracts exactly those curves $C \subset \sX_{i-1}$ with $[C]\in R$ such that the following holds: either $c_R$ is a divisorial contraction or a Mori fiber space, $\sX_i=Z$, and $\phi_i=c_R$ or $c_R$ is 
small and $\phi_i$ is the $\Delta_{i-1}$-flip. As the Kodaira dimension $\kappa(\sX_{i-1})$ is not $-\infty$, the contraction $c_R$ cannot be a Mori fiber space. In the other two cases, we still have $K_{\sX_i/S}=0$. 
Thus, the exceptional locus of $c_R$ is contained in $\supp \Delta$ and the claim about $\phi_i$ follows. We deduce that the composition $\phi:=\phi_N \circ \ldots \circ \phi_1$ is birational when restricted to $Y$. 
We claim that $\phi:\sY \ratl \sY'$ is a cusp model for $Y$.

Being the outcome of an MMP, $\sY'$ is clearly $\Q$-factorial. It remains to show that $\Delta_N=0$. Suppose this is not the case and write the central fiber as $\sY'_c= Y' + Y''$ where $Y'\subset \sY'_c$ is the 
irreducible component such that $\phi: Y \ratl Y'$ is birational. Then we have $\Delta_N=\veps Y''$, i.e. $Y''$ is the support of $\Delta_N$. Choose a curve $C\subset Y''$ which is not contained in $Y'$ but has positive intersection with it: $C.Y' > 0$. Then
\[
0=\veps C.\sY_c'= C.\Delta_N + \veps C.Y' > C.\Delta_N
\]
and therefore $\Delta_N$ is not nef. We obtain a contradiction so that $\Delta_N=0$ and $\phi:\sY \ratl \sY'$ is a cusp model for $Y$.
\end{proof}

Recall the definition of $\Morifan(\sY)$ from Section~\ref{section morifan}.

\begin{definition}\label{definition cuspidal cones}
A cone $\sigma$ of $\Morifan(\sY)$ is called \emph{cuspidal} if there is a component $Y\subset \sY_c$  and a rational $Y$-cusp model $f\colon \sY\dashrightarrow\sY'$ such that $\sigma= C(f)$.
\end{definition}

\begin{remark}\label{remark cuspidal cones}\ 
\begin{enumerate}
	\item As a consequence of Lemma~\ref{lemma picard rank cusp model}, cuspidal cones are $19$ dimensional. 
	\item Lemma~\ref{lemma flop is regular} says that for every cuspidal cone $\sigma$ there is a marked minimal model $f:\sY \ratl \sX$ and a cusp model $\pi:\sX \to \sY'$ such that $\sigma=f^*\pi^*\Nefe(\sY')$. Cusp models are divisorial contractions of extremal rational faces\footnote{That is, intersections of $\Nefe(\sX)$ with a rational linear subspace not meeting the ample cone with non-empty relative interior.} of the nef cone of $\sX$ that lie on the boundary of the movable cone. Therefore, cuspidal cones have to lie on the boundary of the movable cone.
 \item The just mentioned fact that cuspidal cones correspond (not necessarily one-to-one) to marked minimal models, i.e. regular models of the DNV family, together with a divisorial contraction is what will allow us to classify cuspidal cones in Section~\ref{section counting cones}. The task there will be to classify marked minimal models allowing for a divisorial contraction of relative Picard rank two that is an isomorphism on the generic fiber.
\end{enumerate}
\end{remark}

\section{Construction of the fan}\label{section ghks fan}

Let $\sY \to S$ be a regular model of the DNV family of degree $2d$. We fix this model once and for all. Recall that the real N\'eron--Severi group $\NS(\sY_\eta)_\R=\NS(\sY_\eta)\tensor\R$ of the generic fiber of our model is isomorphic to $M_{2d,\R}$ where $M_{2d}$ is the lattice from \eqref{eq lattice m2d}. We  now fix an identification $\NS(\sY_\eta)\cong M_{2d,\R}$ once and for all.

The goal of this section is twofold. 
First, we introduce the \emph{GHKS fan} (see Definition~\ref{definition cusp model GHKS fan}), whose support is the rational closure $C_{2d}^\rc \subset M_{2d,\R} \isom \NS(\sY_\eta)_\R$ of the positive cone. 
In Theorem~\ref{theorem fan and compactification} we prove that the GHKS fan is a semitoric fan in the sense of Looijenga and thus gives rise to a semitoric compactification. 
This is implicit in \cite{GHKS}, but there the emphasis lies more in using the GHKS fan to refine a given toric fan which then gives rise to another toroidal compactification. 
The techniques we use here are clearly based on  \cite{GHKS}.

The second goal of this section is the analysis of the GHKS fan in degree $2d=2$. This case has two important features: the construction is somewhat simpler due to extra symmetries and the resulting fan is an actual toric fan, see Propositions~\ref{proposition degree two semitoric fan is fan} and~\ref{proposition degree two all fans equal}. 

This GHKS fan will be the common refinement of certain fans $\Sigma_{Y}$ coming from irreducible components $Y \subset \sY_c$ of the central fiber. We will first define their restriction to the nef cone.

\begin{definition}\label{definition cusp model fan}
Let $Y \subset \sY$ be an irreducible component of the central fiber. We choose a rational cusp model $\sY \ratl \sY'$ for $Y$ and denote by $\iota:\sY_\eta \to \sY'$ the inclusion. Then we define $\Sigma_{Y}^\nef$ on $\NS(\sY_\eta)_\R$ to be the pull back of the Mori fan of the cusp model along $\iota$, that is, the following collection of cones:
\begin{equation}\label{eq definition cusp model fan}
\Sigma_{Y}^\nef:=\{ \iota^* (\sigma) \mid \sigma \in \Morifan(\Y')\}.
\end{equation}
\end{definition}

\begin{lemma}\label{lemma cusp model fan}
The collection of cones $\Sigma_Y^\nef$ from Definition~\ref{definition cusp model fan} has support equal to $\Nef(\Y_\eta)$. Moreover, the restriction of $\Sigma_Y^\nef$ to any rational subcone of $C_{2d}^\rc$ is a rational cone system.
\end{lemma}
\begin{proof}
The second statement follows from \cite[Theorem 6.5]{GHKS}, see also \cite[Theorem~2.4]{HL}. In both references, the total space $\sY'$ is supposed to be regular, but this is not necessary. The main technical tools are Theorems~3 and~4 of \cite{Kaw11} which are proven for $\Q$-factorial klt pairs. What is used is that an effective divisor $B$ on $\sY'$ can always be scaled so as to make $(\sY, B)$ klt. Alternatively, we can deduce the first claim from the embedding of $\Morifan(\sY')$ into $\Morifan(\sY)$ as in Section~\ref{section embed into morifan} and argue for the Mori fan of the regular scheme $\sY$.

Moreover, the isomorphism $\iota^*: \NS(\sY') \to \NS(\sY_\eta)$ from Lemma \ref{lemma picard rank cusp model} defines an identification $\Mov(\sY')\cong \Nef(\Y_\eta)$, see also \cite[Lemma 8.1]{GHKS}. As the support of $\Morifan(\sY')$ is $\Mov(\sY')$, the first claim follows.
\end{proof}

The notation $\Sigma_Y^\nef$ is justified by the following lemma.

\begin{lemma}\label{lemma auxilliary fan independent of cusp model}
Let $Y \subset \sY$ be an irreducible component of the central fiber. Then the collection $\Sigma_Y^\nef$ is independent of the choice of a cusp model for $Y$.
\end{lemma}
\begin{proof}
Let $p_1:\sY \ratl \sY_1$, $p_2:\sY \ratl \sY_2$ be cusp models for $Y$. Then the induced map $p=p_2 \circ p_1^{-1} : \sY_1 \ratl \sY_2$ is an isomorphism in codimension one, thus a small $\Q$-factorial modification. The isomorphism $p^*:\NS(\sY_2)\to \NS(\sY_1)$ maps $\Morifan(\sY_2)$ isomorphically to $\Morifan(\sY_1)$ and is compatible with restriction to the central fiber so the claim follows.
\end{proof}

Our goal is  to construct a semitoric fan covering the rational closure of (a component of) the positive cone in $M_{2d,\R}$. For this we need to understand the action of the group $\Gammabar$ on the cones of $\Sigma_{Y}^\nef$.
We denote  the set of roots in $M_{2d}$ by $R_{2d}$, i.e.
\[
R_{2d}:=\{ v\in M_{2d} \mid v^2=-2\}.
\]
We choose a set of simple roots $\Delta \subset R$. It is well-known from the theory of reflection groups that the set 
\[
V_{2d}:=\{ v\in C^\rc_{2d} \mid v^2 \geq 0, \, \alpha.v \geq 0 \,\textrm{ for all } \alpha \in \Delta\}
\]
is a fundamental domain for the action of $W_{2d}$ on $C^\rc_{2d}$ and that the orthogonal group $\O(M_{2d})$ decomposes as a semidirect product $W_{2d} \rtimes P_{2d}$ where $W_{2d}$ is the Weyl group (generated by reflections in roots in $R_{2d}$) and $P_{2d}$ is the subgroup that fixes the fundamental domain $V_{2d}$ (\cite[Theorem 12.2]{Hum}). 
Note that reflections on $M_{2d}$ always extend to the K3 lattice so that $W_{2d}\subset \Gammabar$. Hence, the decomposition of the orthogonal group induces a decomposition
\begin{equation}\label{eq semidirect product}
\Gammabar = W_{2d} \rtimes \Pbar, \qquad \textrm{ where }\ \Pbar:=P_{2d} \cap \Gammabar.
\end{equation}

\begin{lemma}\label{lemma action on nef}
Under the identification $M_{2d,\R} \isom \Pic(\sY_\eta)$, a distinguished set of simple roots is given by the effective $(-2)$ classes. With this choice, the fundamental domain $V_{2d}$ is isomorphic to the cone $\Nef(\sY_\eta)$.
\end{lemma}
\begin{proof}
From Riemann-Roch and Serre duality one deduces that effective $(-2)$ classes constitute a set of simple roots. It follows easily from the Nakai-Moishezon-Kleiman criterion, see e.g. \cite[8, Theorem~1.2]{Huy16} that $\Nef(\sY_\eta)$ cut out by their orthogonals inside the positive cone. 
\end{proof}

Recall that for complex K3 surfaces $S$, the morphism $\Aut(S)\to \O(H^2(S,\Z))$ is injective. With this in mind, we prove

\begin{proposition}\label{proposition automorphisms}
Under the identification $\Pic(\sY_\eta)\isom M_{2d}$, pullback via an automorphism induces a homomorphism $G:=\Aut(\sY_\eta) \to \O^+(M_{2d})$, whose image contains $\Pbar$. 
For $d=1$ we have a short exact sequence
$$
0\to \Z/2\Z \to G \to \Pbartwo \to 0
$$
where the kernel is given by a non-symplectic involution on $\sY_\eta$. 
\end{proposition}
\begin{proof}
We take an algebraization $\gothY \to C$ over a quasi-projective curve $C$ such that $\sY \to S$ is the base change under a morphism $S \to C$, see Proposition~\ref{proposition algebraic model}. For $t\in C$, there is a canonical morphism $\Pic(\sY_\eta) \to\Pic(\gothY_t)$ which is an isomorphism for the very general point $t$, see \eqref{eq canonical iso picard}. For the fibers over those points, also the nef cone is isomorphic to the fundamental domain $V_{2d}$ by the same reasoning as in 
Lemma~\ref{lemma action on nef}. We will show first that $\Pbar$ is contained in the image of $G \to \O^+(M_{2d})$. Let us observe that an $s\in \Pbar$ lifts to an isometry $\tilde s\in \O(\Lambda)$ of the K3 lattice \eqref{eq k3 lattice}. By the effective Torelli theorem for complex K3 
surfaces there is an automorphism $\vphi_t$ of $\gothY_t$ which induces $\tilde s$. Let $\Gamma_t \subset \gothY_t\times \gothY_t$ be the graph of $\vphi_t$. By countability of components of 
the relative Hilbert scheme of $\gothY \times_C\gothY$ over $C$, there is an irreducible cycle $\Gamma \subset \gothY \times_C\gothY$ whose fiber over uncountably many $t\in C$ equals $\Gamma_t$. 
Therefore, $\Gamma$ is the graph of a birational automorphism of $\gothY$, thus an automorphism of $\sY_\eta$. By construction, it acts as $\tilde s$ on $\Nef(\sY_\eta)$.

The following argument shows that the kernel of $G\to \O^+(M_{2d})$ is at most $\Z/2\Z$. For a very general $t\in C$ the only Hodge isometries of the transcendental lattice are $\pm\id$ which can be verified using a Mumford--Tate group argument. The Mumford--Tate group of the transcendental lattice $T \subset H^2(\gothY_t,\Z)$ for very general $t\in C$ is $\SO(T\tensor\Q)$, a Hodge isometry is invariant under the Mumford--Tate group in $\End(T)$, and by elementary representation theory of the special orthogonal group the only such invariants are multiples of the identity. We refer to \cite[Chapter~15.2]{CMP17} for generalities on Mumford--Tate groups.

Let us now specialize to $d=1$ and show that there is a non-trivial element in the kernel. For the very general $\gothY_t$ as above, the transcendental lattice is given by $T=U \oplus \left\langle 2d\right\rangle$. We claim that there is an isometry of $\Lambda$ that restricts to $\id$ on $M_{2d}$ and to $-\id$ on $T$. Let us choose generators $U=\left\langle e,f \right\rangle$ and define an isometry $\alpha:U\to U$ by $e\mapsto -f$, $f\mapsto -e$. Putting $h:=e+df$ and $h'=h-df$ we may write $T=U\oplus \left\langle h\right\rangle$ and $M_{2d}=U\oplus 2E_8(-1) \oplus \left\langle h'\right\rangle$ and the isometry
\[
\id \oplus \alpha \oplus \id: U\oplus U \oplus \left(U \oplus 2E_8(-1)\right) \to U\oplus U \oplus \left(U \oplus 2E_8(-1)\right)
\]
restricts to $(-\id_T)\oplus\id_{M_{2d}}$ on the sublattice $T\oplus M_{2d} = U\oplus \left\langle h, h'\right\rangle \oplus \left(U \oplus 2E_8(-1)\right) \subset \Lambda$. This isometry is clearly Hodge and invoking Torelli once more we see as above that it is induced by an automorphism of $\sY_\eta$, which is a non-symplectic involution. 
\end{proof}

\begin{remark}\label{remark antisymplectic involution}
We recall that $\Pbartwo \isom \wt\O^+(M_{2})/W_2 =  \O^+(M_{2})/W_2 \isom S_3$. In this way we have recovered the well known result that the automorphism group  
of a $K3$ surface with $\NS(X)\cong U\oplus 2E_8(-1) \oplus \langle -2 \rangle$ is given by $\Aut(X)\cong S_3\times \mathbb Z/2\mathbb Z$, see \cite{Kon89}. Here $S_3$ is the group of symplectic automorphisms and the 
non-trivial element of $\mathbb Z/2\mathbb Z$ is the anti-symplectic isomorphism constructed above.
\end{remark}

\subsection{Embedding into the Mori fan}\label{section embed into morifan}
Our next goal is to embed the collections $\Sigma^\nef_Y$ into the Mori fan of $\sY$. This is a helpful feature as it allows to use the methods from \cite{HL} to analyze them. Most of this material can be found in \cite{GHKS}, only here and there we give some more details or choose a different presentation.

Given a rational cusp model $p:\sY \ratl \sY_1$, we embed the Morifan $\Morifan(\sY_1)$ into $\Morifan(\sY)$. Note however, that such an embedding is not given by pullback along $p$ as $p^*$ does not necessarily send cones to cones due to the lack of functoriality of pullbacks along rational maps. Indeed, given a cone $C(f) \in \Morifan(\sY_1)$, say coming from a contraction $f:\sY_1 \ratl \sX$, we have $C(f\circ p)\in\Morifan(\sY)$ and this cone is in general not equal to the set $p^*(C(f))$. While the map $p^*$ is linear, but does not preserve cones, the map we aim to construct respects cones, but will only be piecewise linear. The embedding $C(f)\mapsto C(f\circ p)$ leads to the following definition.

\begin{definition}\label{definition component section}
Let $\sY \to S$ be a model of the DNV family, let $Y$ be an irreducible component of the central fiber of $\sY\to S$, and let $p:\sY\ratl \sY_1$ be a rational cusp model for $Y$. Let $q:\sY_1 \ratl \sY_2$ be a small $\Q$-factorial modification. Then on $q^*\Nefe(\sY_2) \subset \Mov(\sY_1)$ we define
\[
s_{\sY_1}:\Mov(\sY_1) \to \Mov(\sY), \qquad s_{\sY_1}\vert_{q^*\Nefe(\sY_2)}:= (q \circ p)^* \circ q_*.
\]
\end{definition}

As the following lemma shows, this does not depend on the choices made. More than that, the composition with the restriction to the generic fiber only depends on the component $Y$.

\begin{lemma}\label{lemma section independent}
The map $s_{\sY_1}:\Mov(\Y_1) \to \Mov(\sY)$ from Definition \ref{definition component section} is well defined, piecewise linear, and continuous. If $\iota_1: \sY_\eta \into \sY_1$ is the inclusion, the composition $s_{\sY_1} \circ (\iota_1^*)^{-1}$ is a section of the restriction $\Mov(\sY) \to \Mov(\sY_\eta)=\Nef(\sY_\eta)$ and depends only on $Y$, but not on the rational cusp model. 
\end{lemma}
\begin{proof}
By \cite[Theorem~2.3]{Kaw97}, the nef cones of small $\Q$-factorial modifications cover $\Mov(\sY_1)$ so that it is sufficient to define $s_{\sY_1}$ on these cones. Let us show that $s_{\sY_1}$ does not depend on the choice of small $\Q$-factorial modification. Let $q_2:\sY_1 \ratl \sY_2$ and $q_3:\sY_1 \ratl \sY_3$ be small $\Q$-factorial modifications such that $q_2^*\Nefe(\sY_2) \cap q_3^*\Nefe(\sY_3) \neq \emptyset$. If the interiors of both cones intersect, then there is an isomorphism $h:Y_2 \to[\isom] Y_3$ such that $h \circ q_2 = q_3$, see Lemma~\ref{lemma cones determine contractions} below. In this case obviously $(q_2 \circ p)^*\circ {q_2}_* = (q_3 \circ p)^*\circ {q_3}_*$.

Otherwise, the intersection has to be contained in the boundary of both cones. Suppose first that $q_2^*\Nefe(\sY_2) \cap q_3^*\Nefe(\sY_3)$ intersects the interior of the big cone. By \cite[Theorem~5.7]{Kaw88} (see also \cite[Theorem~1.9]{Kaw97}), the nef cone is rational polyhedral inside the big cone so that there is a unique face $F$ of $q_2^*\Nefe(\sY_2)$, $q_3^*\Nefe(\sY_3)$ such that $q_2^*\Nefe(\sY_2) \cap q_3^*\Nefe(\sY_3)=F$ and a diagram
\begin{equation}\label{eq diagram contractions for well defined section}
\xymatrix{
\sY_2 \ar[dr]_{f_2} \ar@{-->}[rr] &&\sY_3 \ar[dl]^{f_3}\\
&\sZ&\\
}
\end{equation}
where $f_2, f_3$ are regular contractions associated to $F$. In particular, $$F=q_2^*\left(f_2^*\Pic(\sZ) \cap \Nefe(\sY_2)\right)=q_3^*\left(f_3^*\Pic(\sZ) \cap \Nefe(\sY_3)\right).$$

Take $\alpha \in F \subset \Mov(\sY_1)$ and  write $\alpha = q_2^* f_2^* \beta = q_3^*f_3^*\beta$ for some class $\beta$ on $\sZ$. Taking a resolution $W$ of indeterminacies of all the rational maps, and letting $q:=q_3 \circ q_2^{-1}$, we obtain a diagram 
\[
\xymatrix{
&& W \ar[dll]_{\pi_0}\ar[ddll]^{\pi_1}\ar[dd]^{\pi_2} \ar[ddr]^{\pi_3}\\
\sY \ar@{-->}[d]_p &&\\
\sY_1 \ar@{-->}[rr]^{q_2} \ar@/^1.5pc/@{-->}[rrr]^{q_3}&& \Y_2 \ar[dr]_{f_2}\ar@{-->}[r]_{q}& \sY_3\ar[d]^{f_3} \\
& & & \sZ &\\
}
\]
and we have to show that $(q_2 \circ p)^* {q_2}_* \alpha = (q_3 \circ p)^* {q_3}_* \alpha$. Let us observe that $\pi_3^*\circ f_3^* = \pi_2^*\circ f_2^*$ by commutativity of the diagram and therefore
\[
\begin{aligned}
(q_2 \circ p)^* {q_2}_* \alpha & = (q_2 \circ p)^* {f_2}^* \beta = {\pi_0}_*\pi_2^*f_2^* \beta = {\pi_0}_*\pi_3^*f_3^* \beta \\
& = (q_3 \circ p)^*f_3^*\beta = (q_3 \circ p)^*{q_3}_*\alpha.
\end{aligned}
\]
Now suppose that $q_2^*\Nefe(\sY_2) \cap q_3^*\Nefe(\sY_3)$ lies on the boundary of the big cone and let $\alpha$ be contained in this intersection. As $\Mov(\sY_1)$ is the rational closure of its interior, it is sufficient treat the case where $\alpha$ is rational, or even integral. Then $\alpha$ gives rise to a fibration $f_2:\sY_2\to B$ which is an elliptic fibration on the generic fiber over $S$. 
Let $A$ be the pullback along $q$ of an ample prime divisor. As $A$ is $f_2$-nef if and only if $A+f_2^*L$ is nef for a sufficiently ample divisor $L$ on $B$ and clearly $A+\pi^*L \in q^*\Amp(\sY_3)$, we deduce that for $A$ to be $f_2$-nef, we must have  that $q$ is an isomorphism by Lemma~\ref{lemma cones determine contractions} below. Let us assume this is not the case. Then we run a log MMP for the pair $(\sY_2,\veps A)$, where $\veps$ is small enough in order to make the pair klt, and obtain a sequence of flops over $B$ connecting $\sY_2$ to $\sY_3$. Thus, we may reduce to the case where $q$ is a flop over $B$:
\[
\xymatrix{
\sY_2\ar[dr]_{f_2}\ar@{-->}[rr]^q&& \sY_3 \ar[ld]^{f_3}\\
&B&\\
}
\]
In this case, $q_2^*\Nefe(\sY_2) \cap q_3^*\Nefe(\sY_3)$ is a facet (in both cones) and we are back in the previous case.
We thus conclude that $s_{\sY_1}$ is well-defined and continuous. It is piecewise linear by definition.

Next we show that $s_{\sY_1} \circ (\iota_1^*)^{-1}$ is independent of the cusp model. Let $p:\sY \ratl \sY_1$, $p':\sY \ratl \sY_1'$ be rational $Y$-cusp models and consider the commutative diagram
\[
\xymatrix@=0.5cm{
\sY_\eta \ar@{^(->}[dr]^\iota \ar@{^(->}@/^1pc/[drrr]^{\iota_1'} \ar@{^(->}@/_1pc/[dddr]_{\iota_1}&&\\
&\sY \ar@{-->}[rr]^{p'} \ar@{-->}[dd]_{p}&& \sY_1' \ar@{-->}[dd]^{f'}\\
&&&\\
&\sY_1 \ar@{-->}[rruu]^r \ar@{-->}[rr]^f&& \sY_2\\
}
\] 
where $f, f'$ are small $\Q$-factorial modifications and $r:=p' \circ p^{-1}$.  By $\Q$-factoriality of cusp models, also $r$ is a small $\Q$-factorial modification. Showing that $s_{\sY_1}\circ  (\iota_1^*)^{-1}=s_{\sY'_1} \circ (\iota_1'^*)^{-1}$ is tantamount to showing $s_{\sY_1'} = s_{\sY_1}\circ r^*$. Let us consider these maps on the cone $f'^*\Nefe(\sY_2)$. We have $f_*'=f_*\circ r^*$ as all these maps are small $\Q$-factorial modifications and thus 
$$
s_{\sY_1'} = (f'\circ p')^*\circ f'_* =  (f\circ p)^*\circ f_*\circ r^* =s_{\sY_1}\circ r^*
$$
on $f'^*\Nefe(\sY_2)$. It is straight forward to see that $s_{\sY_1}\circ  (\iota_1^*)^{-1}$ is a section of the restriction and so we conclude the proof.
\end{proof}

In the proof of the preceding lemma, we used the following well-known lemma. We record a proof for convenience.

\begin{lemma}\label{lemma cones determine contractions}
Let $\sY, \sY', \sY''$ be projective $S$-schemes. Let $f:\sY \ratl \sY'$ and $h:\sY \ratl \sY''$ be birational contractions such that the relative interiors of $f^*\Nefe(\sY')$ and $h^*\Nefe(\sY'')$ have a nonempty intersection. Then there exists an isomorphism $g :\sY'' \to \sY'$ such that $g \circ h = f$.
\end{lemma}
Note that we do not ask that $\sY'$ and $\sY''$ be isomorphic in codimension one. This is a consequence.
\begin{proof}
Let us take a resolution of indeterminacy as in the following diagram
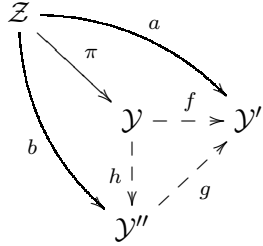
\[
\xymatrix{
\sZ \ar[dr]^\pi \ar@/_1pc/[ddr]_b \ar@/^1pc/[drr]^a &&\\
& \sY \ar@{-->}[r]^f \ar@{-->}[d]_h& \sY' \\
&\sY'' \ar@{-->}[ru]_g&\\
}
\]
We put $g:=f\circ h^{-1}$ and want to show that this is an isomorphism, in particular, a regular map. By assumption there are very ample Cartier divisors $A'$ on $\sY'$ and $A''$ on $\sY''$ such that $h^*A''=f^*A'$, i.e. $\pi_*b^*A''=\pi_*a^*A'.$ We deduce that there are $\pi$-exceptional divisors $E', E'' \geq 0$ such that $a^*A'+E' = b^*A'' +E''$. But all $\pi$-exceptional divisors are also exceptional for $a$ and $b$. From the equality
\[
H^0(\sY',\sO_{\sY'}(A'))  = H^0(\sZ,\sO_{\sZ}(a^*A'+E')) = H^0(\sZ,\sO_{\sZ}(b^*A''+E'')) = H^0(\sY'',\sO_{\sY''}(A'')) 
\]
we see that the morphisms associated to the linear systems of $A'$ and $A''$ fit into a commuting diagram
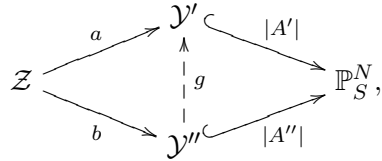
\[
\xymatrix@R=2mm@C=1.5cm{
&\sY' \ar@{^(->}[rd]^{\abs	{A'}}&\\
\sZ \ar[ru]^a \ar[rd]_b&& \IP^N_S,\\
&\sY'' \ar@{^(->}[ru]_{\abs{A''}} \ar@{-->}[uu]_g&\\
}
\]
so that $g$ has to be an isomorphism.
\end{proof}

We can now finally construct embeddings of the collections $\Sigma_Y^\nef$ into the Mori fan $\Morifan(\Y)$. This is obtained via the following   
\begin{definition}\label{definition section}
Let $Y \subset \sY_c$ be an irreducible component and let $\sY \ratl \sY_1$ by a $Y$-cusp model. We denote by $\iota_1:\sY_\eta \to \sY_1$ the inclusion and define the section 
\begin{equation}\label{eq section component}
s_Y:= s_{\sY_1} \circ (\iota_1^*)^{-1}: \Nef(\sY_\eta) \to \Mov(\sY).
\end{equation}
By Lemma \ref{lemma section independent}, the section $s$ does not depend on the $Y$-cusp model but only on $Y$, which justifies the notation. 
\end{definition}

\subsection{The action of the birational automorphism group} \label{section automorphism}

A semitoric fan on $C_{2d}^\rc$ is by definition endowed with an action of $\Gammabar$. Our collection $\Sigma_Y^\nef$ has support on $\Nef(\sY_\eta)$, so in view of \eqref{eq semidirect product} we will first describe the action of $\Pbar$ on $\Sigma_Y^\nef$. In fact, by Proposition~\ref{proposition automorphisms} we know that $\Pbar$ is contained in the image of the map
\[
G=\Aut(\sY_\eta) = \Bir(\sY) \to \O^+(M_{2d}),
\]
so it will be sufficient to describe an action of $G$. This is given by
\begin{equation}\label{eq action aut on mori}
G \times \Morifan(\sY) \to \Morifan(\sY), \quad (g,C(f)) \mapsto g^*C(f).
\end{equation}
It is worthwhile to note that, unlike in our discussion at the beginning of Section~\ref{section embed into morifan}, here we do have $g^*C(f) = C(f\circ g)$ because $g^*g_*=\id$. 

The automorphism group of $\sY_\eta$ also acts on the irreducible components $Y_1,\ldots,Y_{d+2}$ of $\sY_c$ by permutations. Let us denote by $\Pi_{2d} \subset S_{d+2}$ the image of $G \to S_{d+2}$. We infer from \cite[Corollary~5.39]{HL} that $\Pi_2=S_3$. Given a component $Y \subset \sY_c$, we denote by $G_Y$ the stabilizer group of $Y$ under the action $G \to S_{d+2}$.

We record the following elementary lemma.

\begin{lemma}\label{lemma group action}
Let $Y \subset \sY_c$ be an irreducible component and let $\sY \ratl \sY'$ be a rational $Y$-cusp model. Then there is an action
\begin{equation}\label{eq action stabilizer on mori cusp}
G_Y \times \Morifan(\sY') \to \Morifan(\sY'), \quad (g,C(f)) \mapsto g^*C(f),
\end{equation}
which is compatible with the action \eqref{eq action aut on mori}.\qed
\end{lemma}

We deduce the following basic properties of cuspidal cones.

\begin{corollary}\label{corollary cuspidal cones}
Let $Y \subset \sY_c$ be an irreducible component of the central fiber and denote by $\Cusp_Y \subset \Morifan(\sY)$ the set of cuspidal cones associated to rational $Y$-cusp models. Then every automorphism $g\in G$ induces a bijection
\begin{equation}\label{eq automorphism cuspidal cones}
g^*:\Cusp_Y \to \Cusp_{g.Y}, \quad C \mapsto g^*C
\end{equation}
and similarly for the Mori fan. Here, $g.Y$ denotes the permutation action of $g$ on the set of irreducible components of $\sY_c$. In particular, the collection of cones $\Sigma_Y$ only depends on the $\Pi_{2d}$-orbit of $Y$.
\end{corollary}
\begin{proof}
The map is well-defined, because for every rational $Y$-cusp model $f:\sY \ratl \sY'$ the composition $f\circ g:\sY \ratl \sY \ratl \sY'$ is a rational $g.Y$-cusp model.
\end{proof}

\subsection{The GHKS semitoric fan} \label{section ghks semitoric}

We will now define the GHKS semitoric fans. For this let $Y_1, \ldots, Y_{d+2}$ be the irreducible components of the central fiber of $\sY\to S$, cf. Lemma \ref{lemma irreducible components}. 
\begin{definition}\label{definition cusp model GHKS fan nef}
We define the collection $\mghksnef$ to be the coarsest common refinement of the collections 
$\Sigma_{Y_1}^\nef, \ldots, \Sigma_{Y_{d+2}}^{\nef}$. 
\end{definition}

\begin{remark}
The coarsest common refinement is characterized by the property  that every cone of $\mghksnef$ is a subcone of some cone of $\Sigma_{Y_i}^{\nef}$ for all $i$ and if $\Sigma'$ is another collection with this property, then every cone in $\Sigma'$ is a subcone of some cone in $\mghksnef$. The usual proof, cf. \cite[Corollary I.4.9]{AMRT}, that a finite set of fans has a coarsest common refinement, which is achieved by taking intersections of cones, carries over to this situation without problems.
\end{remark}

\begin{lemma}\label{lemma:equivariance}
The group $\Pbar$ acts on each of the collections $\Sigma_{Y_1}^\nef, \ldots, \Sigma_{Y_{d+2}}^{\nef}$ and therefore also on their coarsest common refinement $\mghksnef$. 
\end{lemma}
\begin{proof}
By Proposition~\ref{proposition automorphisms}, every element $h\in\Pbar$ is induced by an automorphism $g\in G$. Thus, $h$ clearly acts on $\NS(\sY_\eta)$. Given $\sigma \in\Sigma_{Y_i}$ for some $i$ we see that $h.\sigma \in \Sigma_{g.Y_i}$.  But $\Sigma_Y = g.\Sigma_Y$ by Corollary~\ref{corollary cuspidal cones}. 
\end{proof}

We are now ready to define the semitoric fans we are aiming at.
\begin{definition}\label{definition cusp model GHKS fan}
The \emph{Gross--Hacking--Keel--Siebert (GHKS) semitoric fan} is defined by
\begin{equation} \label{equ: cusp model GHKS fan}
\mghks:=  \{ s(\sigma) \mid \sigma \in \mghksnef, \, s \in W_{2d}\}.
\end{equation}
Similarly, given we put $\mghks[Y]:= \{ s(\sigma) \mid \sigma \in \Sigma_Y^\nef, \, s \in W_{2d}\}$ for an irreducible component $Y \subset \sY_c$.
\end{definition}
We will usually refer to this semitoric fan simply as the \emph{GHKS fan}, i.e. drop the word \emph{semitoric} for simplicity. 

In order to justify the definition retrospectively we state the 
 \begin{proposition}\label{proposition:justification}
 The collections $\mghks$, $\mghks[Y]$ define semitoric fans for the group $\Gammabar$ and hence give rise to semitoric compactifications $\Fghks$and $\Fghks[Y]$ of the moduli space of $2d$-polarized $K3$ surfaces.
 \end{proposition}
 \begin{proof}
We must check that the conditions of Definition \ref{definition semitoric fan} are fulfilled. Condition (1), namely the equivariance with respect to $W_{2d}$ is true by construction, the equivariance by $\Pbar$ is true by Lemma~\ref{lemma:equivariance}, and thus the collections are $\Gammabar$ equivariant. The same is true for condition (3)
as the support of $\mghksnef$ is $\Nef(\sY_{\eta})$ and the orbits of $\Nef(\sY_{\eta})$ under $W_{2d}$ cover $C^\rc_{2d}$. For the latter we refer to the geometric argument given in the proof of  \cite[Theorem 6.5]{GHKS}:
a rational ray in $C^\rc_{2d}$ is spanned by a nef class after a finite number of reflections in $(-2)$-curves. The same argument also proves condition (3).

For (2) we need that every isotropic ray $R$ in $\Nef(\sY_{\eta})$ is a cone given by some cusp model. Note that such a ray corresponds to an isotropic nef line bundle on the K3 surface $\sY_\eta$ which is therefore isomorphic to the pull back of a very ample bundle along an elliptic fibration of $\sY_\eta$. We can thus take an arbitrary rational cusp model $\sY \ratl \sY_1$ and obtain an effective prime divisor $D$ on $\sY_1$ such that the restriction $D_\eta$ lies in the ray $R$. If $D$ is not nef, we can run an MMP for the pair $(\sY_1,\veps D)$ for $0< \veps \ll 1$ which will terminate in a pair $(\sY_2, \veps D')$ such that $D'$ is nef on the cusp model $\sY_2$. Hence, it defines a cone in $\mghks$. We have already remarked after Definition \ref{definition morifan} that $\{0\}$ is also a cone in $\mghks$. This proves (2).
\end{proof}

We have already discussed the Coxeter semitoric fan $\vf$ in Example \ref{example:Coxeter}. By construction $\mghks$ is a refinement of the Coxeter fan. Hence, we obtain the main result of this section:

\begin{theorem}\label{theorem fan and compactification}
The GHKS semitoric fan $\mghks$ defines a semitoric compactification $\Fghks$  of the moduli $\sF_{2d}$ of $2d$-polarized $K3$ surfaces. This admits is a semitoric morphism 
\[
\Fghks \to \Fcox
\]
to the semitoric compactification given by the Coxeter semitoric fan.
\end{theorem}

\begin{remark}\label{remark not exactly ghks}
For their construction, Gross--Hacking--Keel--Siebert use an additional refinement of what we called the GHKS fan. They use an averaging process over the sections from Definition~\ref{definition section} resulting in a piecewise linear section $\Nef(\sY_\eta)_\R \to \Mov(\sY)$ which is linear on the cones of $\mghks$. Then they pull back the Mori fan via these sections. But as seen in this section, already $\mghks$ is a semitoric fan which we believe deserves further study.
\end{remark}

\subsection{The GHKS fan in degree two} \label{section ghks degree two}

The case $2d=2$ is the most important for us because the GHKS construction yields a \emph{toroidal} compactification in this case. Indeed, we first observe the following

\begin{proposition}\label{proposition degree two semitoric fan is fan}
Let $\sY \to S$ be a model of the DNV family of degree $2$ and let $Y \subset \sY$ be an irreducible component of the central fiber. Then the 
Coxeter semitoric fan $\vf[2]$ and the GHKS semitoric fan $\mghks[2]$ are $\Gammabar[2]$-fans.
\end{proposition}
\begin{proof}
This is an immediate consequence of the fact that the nef cone $\Nef(\sY_\eta)$ is finitely polyhedral, which in  turn follows from the fact that the lattice $M_2$ is
reflexive, see also \cite[Remark 2.7]{HL}.
\end{proof}

The following corollary is an immediate consequence of Proposition \ref{proposition degree two semitoric fan is fan}. 

\begin{corollary}\label{corollary fan and compactification degree two}
Let $\sY \to S$ be a model of the DNV family of degree $2$. Then $\Fghks[2]$ and $\Fcox[2]$ are toroidal compactifications.\qed
\end{corollary}

Finally, we would like to point out that the refinement process described in Definition  \ref{definition cusp model GHKS fan nef} is unnecessary in the case $d=1$,
Indeed, thanks to the section $s_Y: \Nef(\sY_\eta) \to \Mov(\sY)$ from Definition~\ref{definition section}, we 
can now interpret the collections $\Sigma_Y^\nef$ from Definition~\ref{definition cusp model fan} as living in the Mori fan of $\sY$. Thus we can use the symmetry of the $2d=2$ situation to show:

\begin{proposition}\label{proposition degree two all fans equal}
Suppose $d=1$ and denote by $\sY_c=Y_1\cup Y_2 \cup Y_3$ the decomposition of the central fiber into irreducible components. Then the collections 
from Definition~\ref{definition cusp model fan}
associated to these components coincide (as subsets of $\Nef(\sY_\eta)$):
\[
\Sigma_{Y_1}^\nef=\Sigma_{Y_2}^\nef=\Sigma_{Y_3}^\nef(=\mghksnef[2]).
\]
\end{proposition}
\begin{proof}
We can, without loss of generality, work with the model $\sY=\sY_{\scrP}$. 
By \cite[Corollary 5.39]{HL} there is a subgroup $S_3$ of $\Aut(\sY_{\scrP/S})$ which acts as group of permutations of the components of $\sY_c$. 
If $\sigma$ is an element of this group and $p:\sY \ratl \sY'$ is a rational $Y_i$-cusp model for some $1\leq i \leq 3$, it is clear that $p\circ \sigma$ is a $Y_{\sigma(i)}$-cusp model. The claim now follows from the independence of 
$\Sigma_{Y_i}^\nef$ of the rational ${Y_i}$-cusp model, see Lemma~\ref{lemma auxilliary fan independent of cusp model}.
\end{proof}

\section{Cuspidal Cones}\label{sec:cuspidal_cones}

In order to prove our main result, Theorem~\ref{theorem main}, we need to better understand cusp models. The purpose of this section is to show that the existence of a cusp model $\sY\to\sY'$ puts severe restrictions on the central fiber $\sY_c$. The idea is that as maximality of a degeneration is preserved by a birational contraction, the non-contracted component of $\sY_c$ has to account for most of its Picard group. The results of this section are important in the proof of Theorem~\ref{theorem main} which is achieved in Section~\ref{section counting cones}.

\begin{proposition}\label{proposition picard noncontracted}
Let $\sY\to S$ be a model of the DNV family of degree $2$, and let $f:\sY \to \sY'$ be a cusp model with $Y \subset \sY_c$ the non-contracted component of the central fiber. Let $Y^\nu \to Y$ be the normalization. 
Then the pullback map satisfies
\[
\rk\left(\Pic(\sY')\to\Pic(Y^\nu)\right) \geq 19.
\]
In particular, $\vrho(Y^\nu) \geq 19$.
\end{proposition}
\begin{proof}
Let us denote by $f_c:\sY_c \to \sY'_c$ the restriction of $f$ to the central fiber. Observe that by Proposition~\ref{proposition picard group central fiber} restriction gives an isomorphism $\Pic(\sY')\isom \Pic(\sY'_c)$ and $f_c^*:\Pic(\sY'_c) \to \Pic(\sY_c)$ is injective. By Lemma~\ref{lemma picard rank cusp model} we have $\vrho(\sY'_c) =19$ and from \cite[Lemma 1.10]{HL} we infer that the pullback $\Pic(\sY_c) \to \Pic(\sY_c^\nu)$ is injective where $\nu:\sY_c^\nu \to \sY_c$ is the normalization. Let $K:=\ker\left(\Pic(\sY_c) \to \Pic(Y^\nu)\right)$. It suffices to show that $K\cap \img f_c^*=0$.

Let $C \subset Y$ be an irreducible component of the double locus and denote by $C' \subset \Y'_c$ its image under $f_c$. Suppose that $L$ is a line bundle on $\sY_c'$ such that $f_c^* L \in K$. We have to show that the pullback of $L$ to every component of $\sY_c^\nu$ is trivial. 
By assumption, its pullback to $Y^\nu$ is trivial and hence so is the pullback to $C$. We conclude that the restriction $L\vert_{C'}$ is numerically trivial. For any component $X \subset \sY_c^\nu$ different from $Y^\nu$, the induced map $h:X \to \sY_c'$ factors through $C'$ for some double curve $C$ as above. Thus, $h^*L$ is numerically trivial as well and therefore trivial as $X$ is a rational surface. This concludes the proof.
\end{proof}

We will improve upon this result in Proposition~\ref{proposition picard noncontracted improved} below, this will however need the finer analysis. For this purpose we need so-called curve structures as developed in \cite{HL}.

\subsection{Curve structures}\label{section curve structures}

As in Example~\ref{example dnv family}, we denote by $\scrP$ and $\scrT$ the two possible triangulations of the $2$-sphere with two triangles, see Figure~\ref{ZETA}, and by $\YP$ and $\YT$ the corresponding $d$-semistable $K3$ surface of type III in $(-1)$-form with trivial Carlson map. 
Whenever $\sY\to S$ is a model of the DNV family in degree $2$, the central fiber $\sY_c$ fits into a sequence  of elementary modifications of type~I  
\begin{equation}\label{eq type one sequence}
Y_\scrG \dashrightarrow Y_1 \dashrightarrow \dots\dashrightarrow Y_i \dashrightarrow \dots \dashrightarrow Y_n=\sY_c
\end{equation}
where $\scrG\in \{\scrP,\scrT\}$. Recall that the components of $\YP$ or $\YT$ are isomorphic to $\gothY_1$, $\gothY_2$, or $\gothY_4$,  respectively, as introduced in Example~\ref{example dnv family}. For every irreducible component $Y \subset \sY_c$, its curve structure $\Gamma_Y$ is a combinatorial object that remembers how $Y$ the sequence of type I flops was built from some $\gothY_i$ for $i\in\{1,2,4\}$. The definition is by induction over the length of \eqref{eq type one sequence}. 

Denoting by $D$ the anticanonical divisor of $\mathfrak{Y}_i$ we set 
\[ C(\mathfrak{Y}_i):=\{ C\subset \mathfrak{Y}_i \mid C \text{ is an integral curve with } C^2<0, C\not\subset D \}. \] 
For $Y$ as above, \eqref{eq type one sequence} induces a sequence
\[
\gothY_i=:W_0 \ratl \ldots \ratl W_{n-1} \ratl W_n=:Y^\nu
\]
of birational maps and we suppose that $C(W_{n-1})$ has already been defined. Then the set of curves $C(W_n)$ on $W_n=Y^\nu$ is given by images of non-contracted curves in $C(W_{n-1})$ if $W_{n-1} \to W_n$ is a blow up or by either strict transforms of curves in $C(W_{n-1})$ or the exceptional curve if $W_{n-1} \ratl W_n$ is the inverse of a blow up.

\begin{definition}\label{definition curve structure} 
The \emph{curve structure} $\Gamma_Y$ is the dual graph of $C(Y)$ whose vertices are labelled with the self-intersection numbers. We say that $\Gamma_Y$ has \emph{type} $d_i$ if $Y$ maps to $\mathfrak{Y}_i$ under a sequence of type I flops. 

Let $Y^\nu\to Y$ be the normalization and $D=\sum D_j$ the anticanonical divisor of $Y^\nu$. The \emph{augmented} curve structure $\Gamma_Y^a$ is obtained by appending a vertex $v_{D_j}$ to $\Gamma_{Y}$ for each $j$ and an edge between $v_{D_j}$ and $v\in \Gamma_Y$ if and only if $D_j.C_v  \neq 0$ for the curve $C_v$ corresponding to $v$. 
\end{definition}

Let us review some basic properties of curve structures.

\begin{enumerate}
\item The curve structure is well defined, i.e. independent of the chosen sequence \eqref{eq type one sequence} by \cite[Lemma~3.8]{HL}.
\item Given $\sY_c$, the type of $\Gamma_Y$ is well-defined as type~I flops do not change the number of components of the anticanonical divisor.
	\item If $\Gamma_Y$ has type $d_i$ with $i\in \{1,2,4\}$, then  $\{C_v\mid v \in \Gamma_Y\}$ is a $\QQ$-basis of $\Pic(Y^\nu)$, see \cite[Proposition~3.16]{HL}.
	\item If $C_v \subset Y^\nu$ denotes the curve corresponding to a vertex $v\in\Gamma_Y$, the intersection number $C_v.C_w$ for $v, w\in \Gamma_Y$ is either $1$ or $0$, see the paragraph after \cite[Remark~3.11]{HL}.
\end{enumerate} 

Next we recall the definition of an exceptional vertex, see \cite[Definition~3.12]{HL}.

\begin{definition}\label{definition exceptional vertex}
A vertex $v\in \Gamma_Y$ is called  \emph{exceptional} if
\begin{enumerate}
	\item $v^2=-1$,  
	\item  there is a unique vertex in $\Gamma_Y \ohne\{ v\}$ intersecting $v$ nontrivially, and 
	\item if $D_0$ denotes the unique (by \cite[Proposition~3.2]{HL}) component of the anticanonical cycle met by $v$, no other vertex in $\Gamma_Y \ohne\{ v\}$ meets $D_0$.
\end{enumerate} 
\end{definition}

The \emph{leg} of an exceptional vertex $v$ is defined as the maximal connected (full) subgraph $L:=L(v)$ of (the graph underlying) $\Gamma_Y$ containing $v$ such that
\begin{enumerate}
	\item all vertices are adjacent to at most two vertices of $L$,
	\item all vertices except possibly the end $e$ of $L$ (i.e. the unique vertex different from $v$ that is adjacent to exactly one other vertex in $L$) are adjacent to at most two vertices of $\Gamma_Y$, and
	\item no vertex in $L\ohne \{v, e\}$ meets the anticanonical divisor.
\end{enumerate} 

Informally speaking, the leg of an exceptional vertex is obtained by joining adjacent vertices until one reaches a fork or a component of the anticanonical cycle.

\begin{definition} A curve structure $\Gamma_Y$ is called \emph{degenerate} if  there is no exceptional vertex, or if for some  exceptional vertex $v_e$ the leg $L(v_e)$ ends in a vertex $v$ with $v.D_1\geq 1$ for some smooth irreducible double curve $D_1\subset Y^\nu$ which maps to a smooth curve $Y$. 
Curve structures that are not degenerate will be called \emph{non-degenerate}.
\end{definition}

The following notion refines the division into degenerate and non-degenerate curve structures. It has been defined in \cite[Definition~3.17]{HL} for curve structures of type $d_2$ only. We drop this assumption here because it is not needed even though we mainly use it for type $d_2$.

\begin{definition}\label{definition regular curve structure}
We say that a curve structure $\Gamma_Y$ is \emph{regular} if  $|\Gamma_Y|>1$ and no leg $L(e)$ of an exceptional vertex $e$ ends in a vertex $v$ with $v^2=0$. 
A curve structure which is not regular  is called {\em very degenerate} and a curve structure which is  regular and degenerate is called \em{tamely degenerate}.
\end{definition}

We have the implications  ``very degenerate'' $\Rightarrow$``degenerate'' $\Rightarrow$ ``tamely degenerate''. 
Indeed, the second implication holds by definition. If a curve structure is very degenerate then either $\abs{\Gamma_Y}=1$ in which case there is no exceptional vertex and there is an exceptional vertex whose leg ends in a vertex $v$ with $v^2=0$. As $C_v$ is a rational curve, it must intersect $D$ with $C_v.D=2$ by \cite[Proposition~3.2]{HL}, hence the curve structure is degenerate. 
In conclusion, the set of degenerate curve structures is divided into very degenerate and tamely degenerate ones.

\begin{proposition}\label{proposition picard noncontracted improved}
Let $\sY\to S$ be a model of the DNV family of degree $2$ and $f:\sY \to \sY'$ be a cusp model. Let $Y \subset \sY_c$ be the non-contracted component of the central fiber with normalization $Y^\nu \to Y$. 
If $\Gamma_Y$ is of type $d_2$ or $d_4$, then $\vrho(Y^\nu) \geq 20$.
\end{proposition}
\begin{proof}
In view of Proposition~\ref{proposition picard noncontracted}, to prove the claim it is sufficient to exhibit a line bundle on $Y^\nu$ that is not a pullback from $\sY_c'$. 

Let us first treat the case where the type of $\Gamma_Y$ is $d_2$. In this case, $Y^\nu=Y$ is smooth. Let $D_1 + D_2$ be the anticanonical divisor on $Y$. If a component $D_i$ is contracted to a point under $f$, the corresponding line bundle $\sO_{Y}(D_i)$ is not a pullback from $\sY'_c$. If no $D_i$ is contracted, then $D_1$ and $D_2$ are glued under $f$. In case $\Gamma_Y$ has an exceptional vertex, by definition there is a curve $C$ in $\Gamma_Y$ that intersects $D_1$ but not $D_2$, so $\sO_{Y}(C)$ cannot be a pullback. If $\Gamma_Y$ has no exceptional vertex, it must be degenerate. The degenerate curve structures of type $d_2$ with no exceptional vertex are classified, see \cite[Figure~15]{HL}, and necessitate $\vrho(Y) \leq 2$, which is impossible.

If $\Gamma_Y$ is of type $d_4$, the claim follows similarly from \cite[Proposition~4.1]{HL}.
\end{proof}

\begin{corollary}\label{corollary degenerate}
Let $\shY$ be a model of the DNV family of degree $2$ and let $f\colon\shY\to \sY'$ be a  cusp model. Let  $Y\subset \shY_c $ be a component that is contracted. Then  $|\Gamma_Y|\leq 4$ and $|\Gamma_Y|= 4$ only if the non-contracted component is of type $d_1$ and $\Gamma_Y$ is of type $d_4$.  If $\Gamma_Y$ is of type $d_2$,  then $\Gamma_Y$ is  degenerate.
\end{corollary}
\begin{proof}Write $\shY_c=Y_1\cup Y_2\cup Y_3$ with $Y=Y_1$ and suppose $Y_3$ is not contracted.
We have $\sum |\Gamma_{Y_i}|=24$ as $\Gamma_{Y_i}$ is a basis for $\Pic(Y_i)_\QQ$.
So from Proposition~\ref{proposition picard noncontracted improved}, we have  $1\leq|\Gamma_Y|\leq 4$ and even $|\Gamma_Y|\leq 3$ if $Y_3$ is of type $d_2$ or $d_4$. So if $|\Gamma_Y|=4$, then $Y_3$ must be of type $d_1$ and $Y_2$ must have Picard number one, so it cannot be the type $d_4$ component. If $\Gamma_Y$ is of type $d_2$ non-degenerate, we necessarily have $|\Gamma_Y|\geq 4$ as there are $2$ legs that end on a fork by \cite[Example~3.15]{HL}. 
 \end{proof}

For the enumeration of cusp models, we have to provide a deeper analysis of the possible models. For this purpose, we examine which curves \emph{cannot} be contracted in a cusp model.

\begin{proposition}\label{proposition minus two not contracted}
Let $\shY$ be a model of the DNV family of degree $2$ and write $\shY_c=Y_1\cup Y_2 \cup Y_3$. Let $f\colon \shY\to \shY'$ be a $Y_3$-cusp model. 
Let $C \subset Y_1^\nu$ be a $(-2)$-curve on the normalization $Y_1^\nu \to Y_1$. Then $C$ is not contracted by $f$.
\end{proposition}
\begin{proof}
Recall that $C$ does not intersect the double locus by \cite[Proposition~3.2]{HL}. Thus, the line bundle $L:=\sO_{Y^\nu_1}(C)$ has degree zero on the double locus by \cite[Proposition~3.2]{HL}. Hence, by \cite[Lemma~1.10]{HL} there is a line bundle $\shL_c$ on $\shY_c$ that pulls back to $L$ under the canonical map. By maximality of the DNV family, one obtains a line bundle $\shL$ on $\shY$ whose restriction to $\Y_c$ is $\sL_c$. For the restriction $\shL_\eta$, we have $\shL_\eta^2=-2$, so being effective on the central fiber $\shL_\eta$ is effective and hence $\sL$ is effective. Let us write $\sL=\sO_{\sY}(\sC)$ for an effective divisor $\sC \subset \sY$. Let us write $f=\phi_{|B|}$ for some divisor $B$ on $\sY$. Then $f$ contracts $C$ if and only if $B.C=0$. The latter is equivalent to $B.\sC_\eta =0$ which would say that $f$ contracts a divisor on the generic fiber which is a contradiction.
\end{proof}

The following construction is needed to rule out certain constellations.

\begin{construction}\label{changemodel}
Let $\shY$ be a model of the DNV family of degree $2$, let $Y \subset \sY_c$ be an irreducible component of the central fiber, and let $f\colon \shY\to \shY'$ be a  $Y$-cusp model. 
Suppose $C \subset \sY_c$ is a curve contracted by $f$ and that one of the following holds:
\begin{enumerate}
\item $C$ is an interior\footnote{Here, \emph{interior} means it is not part of the anticanonical cycle.} $(-1)$-curve, and if $\shY_c\dashrightarrow Y_c^+$ denotes the elementary modification of type $I$ in $C$, then $Y_c^+$ is projective.
\item $C$ is a component of the double locus and $\nu^{-1}(C)=D_{ij}\cup D_{ji}$  and $D^2_{ij}=D^2_{ji}=-1$, where $\nu\colon \shY^\nu_c\to \sY_c $ denotes the normalization.
\end{enumerate}
By \cite[Corollaries 5.7, 5.8]{HL}, there is a flopping contraction $\phi\colon\shY\to \shZ$ contracting exactly $C$.  Hence, $f$ factors as $\shY \to \shZ \to \shY'$, and there is  model of the DNV family $\shY^+$ and a commutative diagram
\[
\xymatrix{ \shY\ar@{-->}[rr]\ar[dr]^\phi \ar[rdd]_f  & & \shY^+\ar[dl]\\
&\shZ\ar[d] & \\
& \shY'.&
}
\]
Moreover, in the first case the central fiber of $\sY^+$ is exactly $\sY_c^+=Y_c^+$. By construction $\shY^+\to \sY'$ is also a cusp model. 
\end{construction}

We record some consequences of Construction \ref{changemodel} in the next proposition.

\begin{proposition}\label{proposition not contracted}
Let $\shY$ be a model of the DNV family of degree $2$.  
 Write $\shY_c=Y_1\cup Y_2 \cup Y_3$. Let $f\colon \shY\to \shY'$ be a $Y_3$-cusp model. 
Let $C \subset \sY_c$ be a curve such that one of the following assertions holds:
\begin{enumerate}
\item\label{proposition not contracted one} There is $i=1$ or $2$ such that $C \subset Y_i$ is an interior $(-1)$-curve meeting $D_{i3}$, and if $\shY_c\dashrightarrow Y_c^+$ denotes the elementary modification of type $I$ in $C$, then $Y_c^+$ is projective. 
\item\label{proposition not contracted two}  $\nu^{-1}(C)=D_{21}\cup D_{12}$  and $D^2_{12}=-1$, where $\nu\colon \shY^\nu_c\to \shY_c$ denotes the normalization.
\end{enumerate}
Then $C$ is not contracted by $f$.
\end{proposition}
\begin{proof}Consider the situation in Construction \ref{changemodel}.
In case \eqref{proposition not contracted one} the flopped curve $C^+$ of $C$ is an interior  curve on $Y_3^+$. It is contracted by $\shY^+\to \sY'$, a contradiction to the exceptional locus being purely of codimension one, see e.g. \cite[Theorem~VI.1.5]{Kol96}. Case \eqref{proposition not contracted two} is proven similarly, using the description of type II flops, see e.g. \cite{FM83}. 
\end{proof}

\section{Models with dual intersection complex \texorpdfstring{$\mathscr{P}$}{P}}\label{section class p}

In this section, we characterize those models of the DNV family with dual intersection complex $\mathscr{P}$ that are total spaces of cusp models. 

\begin{proposition}\label{proposition leq2}
Let $\shY$ be a model of the DNV family of degree $2$ with dual intersection complex  $\mathscr{P}$ and let $f\colon\shY\to \sY'$ be a cusp model. If $Y\subset \sY_c $ is a contracted component, then $|\Gamma_Y|\leq 2$.
\end{proposition}
\begin{proof}
We know that $|\Gamma_{Y}|\leq 3$ and that $\Gamma_Y$ is degenerate by Corollary~\ref{corollary degenerate}, so it is enough to exclude $|\Gamma_{Y}|=3$. Let us write $\sY_c=Y_1\cup Y_2 \cup Y_3$ and assume $Y_1=Y$. The restriction of $f$ to $Y$ can be written as the map $\pi:=\pi_{|L|}\colon Y\to K$ associated to a linear system for some base point free $L\in \Pic(Y)$. We will distinguish the cases $\Gamma_Y$ very degenerate or tamely degenerate.

Suppose $\Gamma_Y$ is very degenerate and $|\Gamma_{Y}|=3$. Then $\Gamma_{Y}$ has an exceptional vertex $v_1$ and is equal to its leg $L(v_1)=(v_1,v_2,v_3)$. By Proposition~\ref{proposition minus two not contracted}, the $(-2)$-curve $v_2$ is not contracted so that $K\isom \IP^1$. We will produce a contradiction by showing that $L^2\neq 0$.
Let us write $D_{12}+D_{13}$ for the double curve on $Y$ and assume that $D_{12}.{v_1}=1$. By \cite[Proposition~6.3]{HL}, the curve structure determines the surface uniquely, so using $|\Gamma_{Y_1}|=3$ and that $\Gamma_{Y_1}$ is very degenerate, one deduces $D_{13}^2=4$ and $D_{12}^2=-1$. Thus, \cite[Corollary~3.7]{HL} implies
\[
\NE(Y_1)=\langle v_2,v_1,D_{12} \rangle.
\]
Let us write $L= a v_1 + b v_2 + c D_{12}$ with $a,b,c\in \NN_0$. Then 
\[
0=L^2= a(L.v_1) + b (L.v_2) + c (L.D_{12}).
\]
For each generator of $\NE(Y_1)$ that is not contracted by $\pi$, the corresponding coefficient must vanish. It follows that $b=0$. If $Y_2$ is not contracted by $f$, the curve $v_1$ is not contracted by Proposition \ref{proposition not contracted}.
  Else, by the same proposition, $D_{12}$ is not contracted. 
  In any event, $F=dC$ for some $d\in \NN$ and  $C\in\{v_1,D_{12}\}$ and the claim follows.

Next we suppose $\Gamma_Y$ is tamely degenerate and $|\Gamma_{Y}|=3$. Then  $\Gamma_{Y}$ has an exceptional vertex $v_1$ and is equal to the leg $(v_1,v_2)$ generated by $v_1$ together with a vertex $v_3$ meeting only $v_{2}$. Let us write $D_{12}+D_{13}$ for the double curve on $Y=Y_1$ and assume that $D_{12}.{v_1}=1$. 
As the leg of $v_1$ ends in a vertex (namely $v_2$) that meets $D_{13}$,  coming from $\gothY_2$ we must have contracted at least two $(-1)$-curves meeting birational transforms of $D_{13}$ and thus we have $D_{13}^2>0$. 
The curves with negative self-intersection are contained in the set $\{v_1, v_2, v_{3}, D_{12}\}$. By \cite[Corollary~3.7]{HL}, we have
\[
\NE(Y)=\langle C\in \{v_1, v_2, v_3, D_{12}\} : C^2<0 \rangle .
\]
Note that by Corollary~\ref{corollary degenerate} and \cite[Proposition 4.2]{HL}, exactly one of the components of $\shY_c$ has a non-degenerate curve structure (namely the one not contracted by $f$).

 Suppose $\Gamma_{Y_2}$ is non-degenerate. Then $|\Gamma_{Y_3}|=1$ by Proposition~\ref{proposition picard noncontracted improved} and hence $Y_3\cong \PP^2$ by \cite[Corollary~3.7]{HL}. As $D^2_{31}<0$, this is a contradiction.
 So $\Gamma_{Y_3}$ is non-degenerate. Then the elementary modification in $v_2$, call it $\sY_c\dashrightarrow Y_{v_2}$,  defines a projective surface $Y_{v_2}$ by \cite[Proposition~4.4]{HL}. 
 Indeed, one checks that $D^2_{31}<0$ and $D_{32}^2<0$ so that the assumptions of loc. cit. are satisfied. Hence, $v_2$ is not contracted by $\pi$ by Proposition~\ref{proposition not contracted}. 
 Also, $Y_2\cong \PP^2$. As $D_{12}^2+D_{21}^2=-2$, it follows that $D^2_{12}\leq -2$.
 
Suppose $\pi$ contracts $v_1$. Then $\pi$ factors as $Y_1\to Y' \to K$ where $Y_1\to Y'$ denotes the contraction of the  $(-1)$-curve $v_1$. Consequently, $Y'$ is a smooth surface of Picard number $\vrho(Y')=2$. We calculate ${v'}^2_2=0$ and ${D'}^2_{12}\leq -1$ for  the strict transforms $v'_2$  and $D'_{12}$ of $v_2$ and $D_{12}$. So 
\[
\NE(Y')=\langle v'_2, D'_{12}\rangle.
\] 
Note that as $v_2$ is not contracted by $\pi$, also $Y'\to K$ does not contract $v'_2$. Thus, the fiber of $Y'\to K$ is a multiple of $D'_{12}$, hence square negative, a contradiction.
\end{proof}

The above proposition helps to characterize the set of possible total spaces of cusp models with dual intersection complex of the central fiber given by $\scrP$. In the remainder of this section, we show that any of these models of the DNV family admits a cusp model.

\begin{lemma}\label{lemma qfact}
Let $f:X\to X'$ be a projective birational morphism between complex algebraic varieties with rational singularities, and suppose that $\Exz(f)=E_1 \cup \ldots \cup E_m$ for prime divisors $E_i$ where $m=\vrho(X/X')$ is the relative Picard rank. If $X$ is $\Q$-factorial, then so is $X'$.
\end{lemma}
\begin{proof}
Let $\pi:Y\to X$ be a resolution of singularities, denote $\pi':=f\circ \pi$, and let $Y \subset \ol Y$ be a smooth compactification. We denote by $F_1,\ldots, F_n$ the prime exceptional divisors of $\pi$. By \cite[(12.1.6)~Proposition]{KM92}, we have to show that
\begin{equation}\label{eq kollar mori qfact}
\img\left(H^2(\ol Y,\Q) \to H^0(X',R^2\pi'\Q)\right) = \img\left(\Q\left\langle E_i, F_j \mid i,j \right\rangle \to H^0(X',R^2\pi'\Q)\right)
\end{equation}
given that the corresponding equality holds for $X$ instead of $X'$ (and hence only the span of the $F_j$ on the right-hand side). The inclusion $\supseteq$ is clear, for the other direction one uses the semi-simplicity of the category of polarizable Hodge structures as in the proof of loc. cit. to reduce showing that the  right-hand side is equal to the image of $\Pic(\ol Y)_\Q \to H^0(X',R^2\pi'\Q)$ which in turn is equal to the image of $\Pic(Y)_\Q \to H^0(X',R^2\pi'\Q)$. Let us observe that
\[
\Pic(Y)_\Q = \pi^*\Pic(X)_\Q \oplus \Q\left\langle F_j \mid j \right\rangle = (\pi')^*\Pic(X')_\Q \oplus\Q\left\langle E_i, F_j \mid i,j \right\rangle
\]
where the first equality comes from $\Q$-factoriality and the second one holds by assumption on the relative Picard rank. Clearly, the pullback of $\Pic(X')$ maps to zero under the right-hand side of \eqref{eq kollar mori qfact} and so the claim follows.
\end{proof}

\begin{remark}
In the proof of the preceding lemma, we tacitly used that the classes of the $E_i$ are linearly independent in the (rational) Picard group. This can be seen by reducing to the surface case using hyperplane sections on $X'$. For surfaces, it is obvious e.g. from (the easy direction of) Grauert's criterion. 
\end{remark}

Using the usual algebraization arguments, we immediately deduce from Lemma~\ref{lemma qfact} the following

\begin{corollary}\label{corollary qfact}
Let $\shY$ be a model of the DNV family. Let $\contr_F\colon\shY\to \shY'$ be a contraction of an extremal face $F$ with $\dim F = 2$. Suppose $\Exz( \contr_F)=Y_1\cup Y_2$ where $Y_1$ and $Y_2$ are components of $\sY_c$.  Then $\shY'$ is $\QQ$-factorial. 	\qed
\end{corollary}

\subsection{$\Gamma_{Y_i}$ regular, $|\Gamma_{Y_i}|=2$ for $i=1,2$.}

Let $Y$ be a component of a central fiber $\sY_c$ of a model of the DNV family $\shY$ of degree $2$. Suppose the dual intersection complex of $\shY$ is $\scrP$. In particular, $Y$ is smooth. Let $D=D_0+D_1$ be the anticanonical divisor given by the restriction of $\Sing(\sY_c)$ to $Y$. 

\begin{lemma}\label{lemma tamely degenerate type p cardinality two}
Suppose $\Gamma_{Y}$ is regular  and $|\Gamma_Y|=2$.  Then $\Gamma_Y$ consists of two vertices $v_1, v_0$ such that $D_0.v_1=n+2$, $D_1.v_1=0$, $D_1.v_0=D_0.v_0=1$, $v_0^2=0$, $v_1^2=n$, $D_0^2=4+n$, and $D_1^2=-n$ for some $n\geq -2$. 
\end{lemma}
\begin{proof} 
Given that $\Gamma_Y$ is regular, it is uniquely determined by $D_0^2$ and $D_1^2$ by \cite[Proposition~3.18]{HL}. As $|\Gamma_Y|=2$, we must have $D_1^2+D_0^2=4$. Thus one easily classifies the unique such curve structure with $D_0^2=4+n$, and $D_1^2=-n$. As $v_1$ is a smooth rational curve, the adjunction formula gives the intersection number with $D_0$. The result is depicted in Figure~\ref{l2curvestruc} for $n\neq -2$. For $n=-2$ one obtains the same graph except that $D_0.v_1=0$. From $v_0^2=0$ and \cite[Proposition~3.2]{HL} one deduces that $D_1.v_0=D_0.v_0=1$.
\end{proof}

\begin{figure}\centering
\begin{tikzpicture}

\begin{scope}[shift= {(5,0)}]      
      
          \node[draw,shape=circle][fill=black] (D1) at (0,0){};      
  \node[draw,shape=circle] (1) at (1,1){};  
  \node(n1) at (1.5,1){{\small $n$}};  
\node[draw,shape=circle] (2) at (1,0){}; 
\node(n2) at (1,-0.5){{\small $0$}}; 
 \node[draw,shape=circle][fill=black] (D2) at (2,0){} ;     
\node(n3) at (2,-0.5){{$-n$}}; 
\node(n3) at (0,-0.5){{$4+n$}}; 
\draw[thick]  (1)--(2)
(1)--(D1)
   (D1)--(2)--(D2);

         \end{scope}        
\end{tikzpicture}
\caption{The augmented regular curve structures $\Gamma_Y$  with $|\Gamma_Y|=2$ for $n > -2$. For $n=-2$ the vertex of self intersection $-2$ does not meet the black nodes.}
\label{l2curvestruc}
\end{figure}
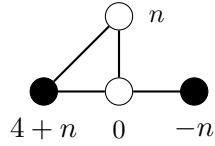

By \cite[Lemma 3]{FrieApp}, $v_0$ is smooth rational of genus $0$, thus isomorphic to $\PP^1$. It is a consequence of \cite[Proposition 3.5]{HL} that $Y$ is a Mori Dream space, so the nef bundle $\sO(v_0)$ is semiample, inducing a fibration $Y\to K$ with fiber $\PP^1$. 
Hence, $Y$ is the Hirzebruch surface $\mathbb{F}_n=\PP(\sO\oplus\sO(n))$.

We have the following lemma. 

\begin{lemma}\label{lemma model p two regular}
Let $\shY\to S$ be a model of the DNV family of degree $2$ with dual intersection complex $\mathscr{P}$ such that there are two components $Y_1, Y_2$ with $|\Gamma_{Y_i}|=2$ and $\Gamma_{Y_i}$ regular. Then there is a cusp model $f\colon\shY\to \sY'$ contracting $Y_1$ and $Y_2$. 
\end{lemma}

\begin{proof} 
Let us write $\Gamma_{Y_1}=\{v_0,v_1\}$ and $\Gamma_{Y_2}=\{w_0,w_1\}$ with $v^2_0=0$, $v_1^2=m$, $w_0^2=0$ and $w_1^2=n$ such that  $D_{12}.v_0=D_{13}.v_0=1=D_{21}.w_0=D_{23}.w_0$, see Lemma~\ref{lemma tamely degenerate type p cardinality two}. For a positive integer $p$ let $G_1$ be the divisor defined by $pv_0$ and $G_2$ the divisor defined by $pw_0$. If $p$ is suitably chosen, there is an ample divisor $G_3$ on $Y_3$ such that there is a $G\in \Pic(\shY)$ restricting to $G_i$ on $Y_i$ by \cite[Proposition 3.20]{HL}. The divisors $G_1$ and $G_2$ are nef: if $C_1$ is a curve on $Y_1$, then $G_1.C_1=p(v_0.C_1)\geq 0$.  So replacing $G$ by some multiple, we obtain a projective morphism $\phi:=\phi_{|G|}\colon \shY\to \shZ$ by abundance. 
  
If $C_1$ is a curve on $Y_1$ contracted by $\phi$, then $G_1.C_1=v_0.C_1=0$.  Write $cC_1=av_0+bv_1$ with $a,b\in \ZZ, c\in\NN$ using \cite[Proposition 3.12]{HL}. Then $b=0$ and thus $C_1\in \RR_+[v_0]$. Similarly, one shows that for a curve $C_2\subset Y_2$ contracted by $\phi$ we have $C_2\in \RR_+[w_0]$. 

The morphism $\phi$  is also an isomorphism on $\shY_\eta$:  Suppose there is a  curve $E\subset \shY_{\eta}$ contracted by $\phi_{|\shY}$.  Take the closure of $E$. This yields an effective Cartier divisor 
 $\bar{E}$ on $\shY$.  The divisor $\bar{E}$ is $2$ dimensional and irreducible, so it does not contain any component of the central fiber. 
  Let ${E}_i$ denote the restriction of $\bar{E}$ to $Y_i$. 
  Then $G_i.{E}_i=0$ as ${E}$ is contracted and hence ${E}_1=qv_0$, ${E}_2=rw_0$. But $G_3$ is ample so ${E}_3=0$ and hence from the gluing conditions for divisors \cite[Lemma 1.10]{HL},  one has $q=r=0$. 
	Thus, $E=0$ as well and the curves contracted by $\phi$ are precisely the curves in the cone $F$ generated by $C_{w_0}$ and $C_{v_0}$. It follows now from Corollary~\ref{corollary qfact} that $\phi$ is the sought-for cusp model.
\end{proof}

\begin{proposition}\label{proposition cuspidal cones model p two regular}
In the situation of Lemma~\ref{lemma model p two regular} the following holds:
\begin{enumerate}
	\item If neither $Y_1$ nor $Y_2$ is $\PP^1\times \PP^1$, the cone $C(f)$ is the only cuspidal cone contained in $\Nef(\sY)$.
	\item If $Y_1\cong \IP^1\times\IP^1$, then there are two distinct cusp models $f_j\colon\shY\to \shY'_j$, $j=1,2$ contracting precisely $Y_1 \cup Y_2$. The cones $C(f_1)$ and $C(f_2)$ are the only cuspidal cones contained in $\Nef(\sY)$.
\end{enumerate}
\end{proposition}
\begin{proof}
In the first case, the existence of the cusp model follows from Lemma~\ref{lemma model p two regular} and Corollary~\ref{corollary qfact}. Let the curve structure of $Y_1$ be given by $v_1,v_0$, with $v_0^2=0$ and $D_{12}.v_0=D_{13}.v_0=1$. If $Y_1$ is not isomorphic to $\PP^1\times \PP^1$, either  $v_1^2=-1$ or  $D^2_{13}<0$ or $D_{12}^2<0$.  Write $C$ for this curve. Then $\NE(Y_1)=\langle C, v_0 \rangle$. It follows that any cusp model  $f'\colon\shY\to \shY'$ contracts $v_0$. The same reasoning applies to $Y_2$. Hence, any such model contracts the same curves as the one constructed in Lemma~\ref{lemma model p two regular} and the uniqueness claim follows. 

In the second case, let us write $\Gamma_{Y_1}=\{v_0,v_1\}$ and $\Gamma_{Y_2}=\{w_0, w_1\}$ with indexing as above. For $f_1$, we can take the cusp model constructed in the proof of Lemma \ref{lemma model p two regular}. As we have seen, the corresponding extremal face is spanned by $R_1=\RR_+w_0$ and $R_2=\RR_+v_0$.

 As $Y_1\cong \PP^1\times \PP^1$, we have $v_1^2=v_0^2=0$. Suppose first that $D_{12}.v_1=2$, see Lemma~\ref{lemma tamely degenerate type p cardinality two}.  There is an ample divisor $A$ with $A.D_{31}=A.D_{32}=2q$ for an appropriate $q\in \NN$. Note that $D^2_{31}=-2$. So $(A+qD_{31}).D_{31}=0$ and $(A+qD_{31}).D_{32}=4q$. 
 Write $L_3:=A+qD_{31}$ and note that $L_3$ is nef and vanishes only on $D_{31}$. 
  Consider  the divisors $L_1:=2qv_1$ on $Y_1$ and $L_2:=4qw_0$ on $Y_2$. We have 
 \[
D_{12}.L_1=4q=D_{21}.L_2,\quad  D_{13}.L_1=0=D_{31}.L_3,\text{ and } \quad D_{23}.L_2=4q=D_{32}.L_3.
\]
Hence, there is a nef divisor $L$ on $\shY$ restricting to  $L_i$ on $Y_i$. A suitable multiple of $L$ defines a contraction $g_1=\pi_{|rL|}\colon \shY\to \shZ$.   
  
   Suppose now that $D_{12}.v_1=0$. Then $D^2_{32}=-8$. Let $A$ be an ample divisor on $Y_3$ and set $L_3=8A+(A.D_{32})D_{32}$. Then $L_3$ is nef and $L_3.C=0$ for a curve $C$ on $Y_3$ if and only if $C$ is a multiple of $D_{32}$. 
   Also, $L_3.D_{31}=q$ for some $q\in \NN$ which is divisible by $2$. Let $L_1$ be the divisor defined by $\frac{1}{2}qv_1$ and write $L_2=\sO_{Y_2}$. We have 
 \[ L_1.D_{12}=0=L_2.D_{21},\quad  L_1.D_{13}=q=L_3.D_{31},\text{ and } \quad L_2.D_{23}=0=L_3.D_{32}.\] Thus there is a divisor $L\in \Pic(\sY)$ restricting to $L_i$ on $Y_i$. Let $g_2=\pi_{|rL|}$ for $r\gg 0$.
 Then $g_2\colon \shY\to Z$ is a contraction. 
 
Now, let  $f_2=g_1$ or $g_2$, depending on the case.  By a similar argument as in the proof of Lemma~\ref{lemma model p two regular}, one sees that $\Exz(g_2)=Y_1\cup Y_2$ and that $f_2$ is the contraction of the extremal face $\NE(f_2)=\langle w_0,v_1\rangle$. Hence, it follows from Corollary~\ref{corollary qfact} that $f_2\colon \shY\to \shZ $ is a cusp model.

Note that in any case,  $\NE(Y_1)=\langle v_0,v_1\rangle$ and $\NE(Y_2)=\langle w_0,D_{21}\rangle$ and as above, one sees that $f_1$ and $f_2$ are the only two cusp models.
\end{proof}

\begin{remark}
It follows from the analysis of the curve structure, that in the situation of Proposition~\ref{proposition cuspidal cones model p two regular} that we cannot have $Y_i\isom \PP^1\times\PP^1$ for both $i=1,2$.
\end{remark}

\subsection{$\Gamma_{Y_1}$ regular, $\Gamma_{Y_2}$ very degenerate and  $|\Gamma_{Y_i}|=2$ for $i=1,2$.}

Let $Y$ be a component of a central fiber $\sY_c$ of a model of the DNV family $\shY$ of degree $2$. Suppose the dual intersection complex of $\shY$ is $\scrP$, so $Y$ is smooth. Let $D=D_0+D_1$ be the anticanonical divisor given by the restriction of $\Sing(\sY_c)$ to $Y$.

\begin{lemma}\label{lemma very degenerate type p cardinality two}
Suppose $\Gamma_{Y}$ is very degenerate and $|\Gamma_Y|=2$.  Then $\Gamma_Y$ consists of two vertices $v_1, v_0$ with $v_0^2=0$, $v_1^2=-1$ and such that $D_0.v_0=2$, $D_1.v_1=1$, $D_0.v_1=D_1.v_0=0$, $D_1^2=0$, and $D_0^2=4$, see Figure~\ref{l2curvestruc2}.
\end{lemma}
\begin{proof} 
As $\Gamma_Y$ is very degenerate, the self intersection numbers of the $D_i$ are uniquely determined by \cite[Proposition~3.18(iii)]{HL} and the requirement that $D_1^2+D_2^2 = 6-|\Gamma_Y|$. By \cite[Proposition~3.18(i)]{HL} there is a unique such $Y$. Constructing it from $\gothY_2$ by suitable blow-downs one sees that the curve structure has to be the one depicted in Figure~\ref{l2curvestruc2}. The claims about the intersection numbers of the vertices with the boundary now follow from  \cite[Proposition~3.2]{HL}.
\end{proof}

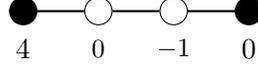
\begin{figure}\centering

\begin{tikzpicture}
\begin{scope}[shift= {(5,0)}]      
      
          \node[draw,shape=circle][fill=black] (D1) at (0,0){};      
   
\node[draw,shape=circle] (2) at (1,0){}; 
\node(n2) at (1,-0.5){{\small $0$}}; 
\node[draw,shape=circle] (3) at (2,0){}; 
\node(n2) at (2,-0.5){{\small $-1$}}; 

 \node[draw,shape=circle][fill=black] (D2) at (3,0){} ;     
\node(n3) at (3,-0.5){{$0$}}; 
\node(n3) at (0,-0.5){{$4$}}; 
\draw[thick]  
   (D1)--(2)-- (3)--(D2);

  \end{scope} 
\end{tikzpicture}
\caption{The augmented very degenerate curve structures $\Gamma_Y$  with $|\Gamma_Y|=2$.}
\label{l2curvestruc2}
\end{figure}

\begin{proposition}\label{proposition cuspidal cones model p one regular}
Let $\shY\to S$ be a model of the DNV family of degree $2$ with dual intersection complex $\mathscr{P}$ such that there are two components $Y_1, Y_2$ such that $|\Gamma_{Y_i}|=2$, $\Gamma_{Y_1}$ is very degenerate, and $\Gamma_{Y_2}$ is regular. Then there is a cusp model $f\colon\shY\to \sY'$ contracting precisely $Y_1$ and $Y_2$. The cone $C(f)$ is the only cuspidal cone contained in $\Nef(\sY)$.
\end{proposition}

\begin{proof}
Write $\Gamma_{Y_1}=\{v_0,v_1\}$ and $\Gamma_{Y_2}=\{w_0,w_1\}$. By Lemmas~\ref{lemma tamely degenerate type p cardinality two} and \ref{lemma very degenerate type p cardinality two} and the triple point formula, we necessarily have $D^2_{21}\in \{-2,-6\}$.
Assume first that $D^2_{21}=-2$. Let $A$ be an ample divisor on $Y_3$. Again by the triple point formula we have $D^2_{32}=-8$ so $L_3=8A+(A.D_{32})D_{32}$ defines a nef divisor on $Y_3$. We have $L_3.C=0$ if and only if $C=D_{32}$ and $L_3.D_{31}=2q$ for some $q\in \NN$ because $D_{31}.D_{32}=2$. Let $L_1$ be the divisor in $\Pic(Y_1)$ defined by $qv_0$ and set $L_2=\sO_{Y_2}$. Then  
\[ L_1.D_{12}=0=L_2.D_{21},\quad  L_1.D_{13}=2q=L_3.D_{31},\text{ and } \quad L_2.D_{23}=0=L_3.D_{32}.\]
 Thus there is a divisor $L\in \Pic(\sY)$ restricting to $L_i$ on $Y_i$. Let $f=\pi_{|rL|}$ for suitable  $r\gg 0$.
 Then $f\colon \shY\to Z$ is a contraction. Suppose $f_{Y_\eta}\colon \shY_\eta\to \shZ_\eta$ contracts a curve $C_\eta$. Taking the closure of $C_\eta$ we obtain a divisor $C$ on $\Pic(\sY)$ with none of the $Y_i$ in the support of $C$. Write $C_i$ for the restriction of $C$ to $Y_i$. As the $C_i$ are contracted, one has that $C_1=av_0$ and $C_3=a'D_{32}$ with $a,a'\geq 0$. From $C_1.D_{13}=C_3.D_{31}$ one deduces that in fact $a=a'$. 
So $C_3.D_{32}=-8a$. The effective cone of $Y_2$ is given by $\Eff(Y_2)=\left\langle w_0, D_{21}\right\rangle$. Writing $C_2=bD_{21}+cw_0$ with $b,c\geq 0$ we have $-8a=C_2.D_{23}=2b+c$. It follows that $C$ is trivial, a contradiction. So $f$ is an isomorphism on $\shY_{\eta}$. It is immediate to check that $\Exz(f)=Y_1\cup Y_2$ and the extremal cone of $f$ is  $\NE(f)=\langle v_0, w_0\rangle$. It now follows from Corollary~\ref{corollary qfact} that $f\colon \shY\to \shZ$  is a cusp model. 
 
 Now suppose $D^2_{21}=-6$. Pick an ample divisor $A$ on $Y_3$. As $D^2_{31}=-2$, $L_3=2A+(A.D_{31})D_{31}$ defines a nef divisor such that $L_3.C=0$ for a curve $C$ on $Y_3$ if and only if $C=D_{31}$. 
 Set $q=\frac{1}{2}(L_3.D_{32})\in \NN$. Let $L_1$ be the divisor defined by $qv_0$ and  $L_2$ the divisor defined by $2qw_0$. As before, one checks that there is a divisor $L\in \Pic(\sY)$ restricting to $L_i$ on $Y_i$, 
 inducing a contraction $f\colon \shY\to \shZ$. If $C$ is a curve contracted by $f$, then $C$ is in the span of $w_0$ and $v_0$. By the same arguments as before, we conclude that $f$ is a cusp model. 
 
In both cases, $v_0$ and $w_0$, interpreted as curves on $\shY$, generate extremal rays that are contracted by $f$. We have  $\NE(Y_1)=\langle v_0, v_1\rangle$ and $\NE(Y_2)=\langle w_0, C\rangle$ where $C$ is a curve with $C^2<0$. So if $g$ is a cusp model, it necessarily contracts $w_0$ and $v_0$, so $g$ factors through $f$ and thus $g=f$.
\end{proof}

\subsection{$|\Gamma_{Y_1}|=1$,   $|\Gamma_{Y_2}|=2$.}

In this case, we write $\Gamma_{Y_1}=\{v_0\}$ with $v_0^2=1$. We also have $D^2_{1j}=4$ and $D_{1i}^2=1$, with $i,j\in \{2,3\}$, $i\neq j$, see Figure \ref{p1comp}. In particular, $Y_1\cong\PP^2$. Note that in this case, $\Gamma_{Y_2}$ is automatically regular.

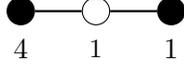
\begin{figure}\centering

\begin{tikzpicture}
\begin{scope}[shift= {(5,0)}]      
      
          \node[draw,shape=circle][fill=black] (D1) at (0,0){};      
   
\node[draw,shape=circle] (2) at (1,0){}; 
\node(n2) at (1,-0.5){{\small $1$}}; 
 \node[draw,shape=circle][fill=black] (D2) at (2,0){} ;     
\node(n3) at (2,-0.5){{$1$}}; 
\node(n3) at (0,-0.5){{$4$}}; 
\draw[thick]  
   (D1)--(2)--(D2);

  \end{scope} 
\end{tikzpicture}
\caption{The augmented  curve structure $\Gamma_Y$  with $|\Gamma_Y|=1$.}
\label{p1comp}
\end{figure}

\begin{proposition}\label{proposition p2}
Let $\shY\to S$ be a model of the DNV family of degree $2$ with dual intersection complex $\mathscr{P}$ such that there is a component $Y_1\cong \PP^2$ and a component $Y_2$ with $|\Gamma_{Y_2}|=2$ Then there is a cusp model $f\colon\shY\to \sY'$ contracting precisely $Y_1$ and $Y_2$. Moreover, $C(f)$ is the only cuspidal cone contained in $\Nef(\sY)$.
\end{proposition}
\begin{proof}
Either we have $D_{12}^2=1$ or $D_{12}^2=4$. In the first case we have $D_{31}^2=-6$ and $D_{32}^2=-9$, in the second case $D_{31}^2=-3$ and $D_{32}^2=-12$. Let $A$ be ample on $Y_3$, then $L_3=50 A+ (6(A.D_{32})+2 (A.D_{31})) D_{32} + (9(A.D_{31})+2 (A.D_{32})) D_{31}$ in the first case and $L_3=32 A+ (3(A.D_{32})+2(A.D_{31})) D_{32} + (12(A.D_{31})+2(A.D_{32})) D_{31}$ in the second case are nef on $Y_3$ and $L_3.C=0$ for an effective curve if and only if $C=aD_{31}+bD_{32}$ with $a,b \geq 0$. Setting $L_1=\sO_{Y_1}$ and $L_2=\sO_{Y_2}$ we obtain a unique semi-ample $\sL$ as before such that $\sL\vert_{Y_i}=L_i$. The rest of the proof is similar to the previous cases.
\end{proof}

\begin{remark}\label{rmk:geographcusp}
Note that in all cases considered so far any cusp model $f\colon\shY\to \sY'$ corresponds to a cone $C(f)$ such that $C(f)$ is not contained in a facet $\sigma$ of $\Nef(\shY)$ corresponding to a small contraction. 	
\end{remark}

\section{Models with dual intersection complex \texorpdfstring{$\mathscr{T}$}{T}}\label{section class t}

Let $\shY\to S$ be a model of the DNV family of class $\mathscr{T}$. There is a unique 
component $Y_\omega\subset \sY_c$ such that $\Gamma_{Y_\omega}$ is of type $d_4$. We denote by $\nu\colon Y_\omega^\nu \to Y_\omega$ the normalization and by $D_\omega \subset Y_\omega$ the unique component of $\sY_c^\sing \cap Y_\omega$  such that $\nu^{-1}(D_\omega)$ 
is reducible.  

We have the following result.  

\begin{lemma}\label{lemma special t not contracted}
Let $\shY$ be a model of the DNV family of class $\mathscr{T}$. Let $\shY\to \shY'$ be a cusp model contracting $Y_\omega$. Then $D_\omega$ is not contracted.
\end{lemma}
\begin{proof}
By projectivity of $\shY$, the preimage $\nu^{-1}(D_\omega)$ consists of two $(-1)$ curves by \cite[Proposition 4.1]{HL}. The elementary modification in $D_\omega$ defines a flop $\phi\colon\shY\dashrightarrow\shY^+$ such that $\shY^+$ has dual intersection complex $\scrP$.
If $D_\omega$ is contracted by $f$, then it follows from Construction~\ref{changemodel} that the there is a cusp model $f^+\colon \shY\to \sY'$ and a commutative diagram
\[
\xymatrix{ \shY\ar@{-->}[rr]^\phi\ar[dr]_f  & & \shY^+\ar[dl]^{f^+}\\
&\sY'.&
}
\]
It follows that the cusp model $\sY^+ \to \sY'$ factors through the flopping contraction associated to $\phi$, a contradiction, cf. Remark~\ref{rmk:geographcusp}.
\end{proof}

There is only one curve structure of type $d_1$ with $|\Gamma_Y|=3$, namely the one depicted in Figure~\ref{length4T}.

\begin{proposition}\label{proposition special t not contracted}
Let $\shY$ be a model of the DNV family of class $\mathscr{T}$. Let $f\colon\shY\to \sY'$ be a cusp model. Then  $Y_\omega$ is not contracted.	
\end{proposition}
\begin{proof} 
We will write $\sY_c=Y_1\cup Y_2 \cup Y_3$ for the irreducible components and assume that the special component is $Y_\omega=Y_1$. Then $Y_2$ and  $Y_3$ are the smooth components. Let us assume $Y_1$ and $Y_2$ are contracted by $f$. We will obtain a contradiction by a detailed analysis of the curve structure on $Y_1$.

Note that by the very nature of curve structures of type $d_4$, we always have $|\Gamma_{Y_1}|\geq 3$. By Proposition~\ref{proposition picard noncontracted}, we must have $\abs{\Gamma_{Y_1}}+\abs{\Gamma_{Y_2}} \leq 5$ and a lengthy but straightforward analysis of the curve structures shows that in this case $|\Gamma_{Y_1}|=4$, the curve structure is unique and depicted in Figure~\ref{figure d4 with four elements}. Moreover, we must have $D_{12}^2=-9$ and $D_{13}^2=9$. Let $C_1,\ldots, C_4$ be the curves corresponding to the vertices in $\Gamma_{Y_1}$ such that $C_1^2=0$ and $C_3^2=7$ and let $D_{\omega1}, D_{\omega2}$ be the preimages of the curve $D_\omega=Y_1^\sing$ under the normalization. From \cite[Corollary~3.7]{HL} we infer that $\NE(Y_1) = \left\langle D_{12}, D_{\omega1}, D_{\omega2}, C_2, C_4\right\rangle$. 

By Lemma~\ref{lemma special t not contracted}, the special component $Y_1$ is not contracted to a point. The pullback of a very ample line bundle from $\sY'$ restricts to a nef line bundle $L$ on $Y_1$ whose linear system gives the restriction $\vphi_L:Y_1 \to \PP^1$ of $f$. In particular, $L^2=0$. As $\Gamma_{Y_2}=1$ (and thus $Y_2\isom \PP^2$), the curve $D_{12}$ is contracted by $\vphi_L$. Hence, if we write $L=aC_1 + bC_2+cC_3+dC_4$ we must have $a=0$.  Testing nefness against $C_2, C_4, D_{\omega1}, D_{\omega2}$ gives the conditions 
\[
c\geq b,\quad c\geq d,\quad b\geq 0,\quad d\geq 0,
\]
so we can estimate
\[
0=L^2 = (2bc - b^2) + (2cd - d^2) + 7c^2  \geq b^2 + d^2 + 7c^2 > 0
\]
where the right-hand side is strictly positive because $L$ is not the trivial bundle. This is a contradiction.

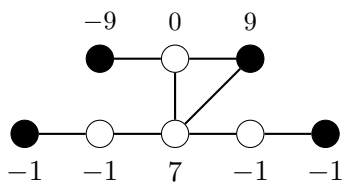
\begin{figure}\centering
\begin{tikzpicture}

\begin{scope}[shift= {(5,0)}]      
\node[draw,shape=circle][fill=black] (D1) at (0,0){};      
\node[draw,shape=circle][fill=black] (D2) at (4,0){} ;     
\node[draw,shape=circle][fill=black] (E1) at (1,1){};      
\node[draw,shape=circle][fill=black] (E2) at (3,1){} ;     

\node[draw,shape=circle] (1) at (2,1){};  
\node[draw,shape=circle] (2) at (1,0){}; 
\node[draw,shape=circle] (3) at (2,0){};  
\node[draw,shape=circle] (4) at (3,0){}; 

\node(n1) at (1,1.5){{\small $-9$}};  
\node(n2) at (2,1.5){{\small $0$}};  
\node(n3) at (3,1.5){{\small $9$}};  
\node(m1) at (0,-0.5){{$-1$}}; 
\node(m2) at (1,-0.5){{$-1$}}; 
\node(m3) at (2,-0.5){{$7$}}; 
\node(m4) at (3,-0.5){{$-1$}}; 
\node(m5) at (4,-0.5){{$-1$}}; 

\draw[thick]  (1)--(3)--(E2)
(E1)--(1)--(E2)
(D1)--(2)--(3)--(4)--(D2);

         \end{scope} 
       
\end{tikzpicture}
\caption{A curve structure of type $d_4$ of cardinality four. The black nodes labeled with $-1$ are the preimages of $D_\omega$ under the normalization.}
\label{figure d4 with four elements}
\end{figure}

\end{proof}

\begin{remark}\label{remark noflops}
One implication of Proposition~\ref{proposition special t not contracted} together with the observation in Remark~\ref{rmk:geographcusp} is that the situation described in Construction \ref{changemodel} does not occur. Cuspidal cones are not contained in interior facets. 
\end{remark}

\begin{proposition}\label{smoothcomp2}
Let $\shY\to S$ be a model of the DNV family of degree $2$ with dual intersection complex $\scrT$. Let $f\colon \shY\to \sY'$ be a cusp model. Let $Y_1$ be a component of $\shY_c$ contracted by $f$.  Then $|\Gamma_{Y_1}|\leq 2$.
\end{proposition}
\begin{proof}By the previous proposition, $Y_1$ is a smooth component.
We only need to show that $|\Gamma_{Y_1}|=3$ is impossible.  There is only one curve structure of type $d_1$ with $|\Gamma_Y|=3$, depicted in Figure \ref{length4T}. As the special component $Y_\omega$ is not contracted, it follows from Proposition~\ref{proposition minus two not contracted}  and Proposition~\ref{proposition not contracted} that none of the curves in $\Gamma_{Y_1}$ is contracted. As these generate $\NE(Y_1)$, this is impossible.

\begin{figure}\centering

\begin{tikzpicture}

  \begin{scope}[shift= {(5,0)}]      
      
          \node[draw,shape=circle](D1) at (0,1){};      
  \node[draw,shape=circle] (1) at (0,0){};  
  \node(n1) at (0,1.5){{\small $-1$}};  
  \node(n4) at (0,-0.5){{\small $-1$}}; 
\node[draw,shape=circle] (2) at (1,0){}; 
\node(n2) at (1,-0.5){{\small $-2$}}; 
 \node[draw,shape=circle][fill=black] (D2) at (-1,0){} ;     
\node(n3) at (-1,-0.5){{$7$}}; 
\draw[thick] (1)--(D2)--(D1)--(1)--(2);

         \end{scope}

\end{tikzpicture}
\caption{The augmented curve structure $\Gamma_Y$  with $|\Gamma_Y|=3$ of type $d_1$.}
\label{length4T}
\end{figure}
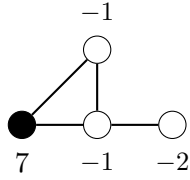

\end{proof}

We now determine which models of the DNV family with $\shY_c$ of dual intersection complex $\mathscr{T}$ admit a cusp model. Note that if $Y$ is a smooth component of $\shY_c$ with $|\Gamma_Y|=2$, then $\Gamma$ is one of the two graphs in Figure \ref{length2T}.

	\begin{figure}\centering

\begin{tikzpicture}
  \begin{scope}[shift= {(2.5,0)}]            
          \node[draw,shape=circle](D1) at (0,1){};      
  \node[draw,shape=circle] (1) at (0,0){};  
  \node(n1) at (0,1.5){{\small $0$}};  
  \node(n4) at (0,-0.5){{\small $-1$}}; 
 \node[draw,shape=circle][fill=black] (D2) at (-1,0){} ;     
\node(n3) at (-1,-0.5){{$8$}}; 
\draw[thick] (D2)--(1)--(D1)--(D2);

         \end{scope} 
         
   \begin{scope}[shift= {(-2.5,0)}]            
  \node[draw,shape=circle] (1) at (0,0){};  
  \node(n4) at (0,-0.5){{\small $0$}}; 
\node[draw,shape=circle] (2) at (1,0){}; 
\node(n2) at (1,-0.5){{\small $-2$}}; 
 \node[draw,shape=circle][fill=black] (D2) at (-1,0){} ;     
\node(n3) at (-1,-0.5){{$8$}}; 
\draw[thick] (2)--(1)--(D2);

         \end{scope}         
\end{tikzpicture}
\caption{The augmented curve structures $\Gamma_{Y}$  with $|\Gamma_{Y}|=2$ of type $d_1$.}
\label{length2T}
\end{figure}
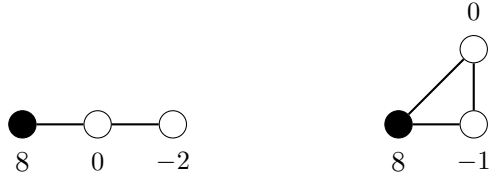

\begin{proposition}\label{modelT:unique}
Let $\shY\to S$ be a model of the DNV family of degree $2$ with dual intersection complex $\mathscr{T}$ such that for any smooth component $Y$ we have  $|\Gamma_{Y}|\leq 2$. Then there is a cusp model $f\colon\shY\to \sY'$ contracting precisely the smooth components. Moreover, $C(f)$ is the only cuspidal cone contained in $\Nef(\sY)$.
\end{proposition}
\begin{proof}
	Write $\shY_c=Y_1\cup Y_2\cup Y_3$ and assume $Y_2$ is the special component. 
	So $Y_1$ and $Y_3$ are the smooth components.
	Consider first the case where $|\Gamma_{Y_1}|=|\Gamma_{Y_3}|=2$. Write  $\Gamma_{Y_1}=\{v_0,v_1\}$  and $\Gamma_{Y_3}=\{w_0,w_1\}$ with $w^2_0=v^2_0=0$.  Let $A$ be an ample divisor on $Y_2$ and denote by $A^\nu$ its pullback to the normalization $Y_2^\nu \to Y_2$. 
	For suitable $n,k,m \in \NN$, the triple of  divisors $L_1=mv_0$, $kA^\nu$, and $L_3=nw_0$, interpreted as a divisor on the normalization $\sY_c^\nu$ of $\sY_c$, descends to a divisor on $\sY_c$ which then by maximality induces a divisor on $\sY$. This divisor is semiample because this is true for its restriction to the central fiber. 
	
	Let $f\colon\shY\to \shZ$ be the morphism induced by $|lL|$ for $l\gg 0$. By construction, $Y_1\cup Y_2 \subset \Exz(f)$. Note that if $C$ is a curve on $Y_1$ (respectively $Y_3$) contracted by $f$, then $C$ is a multiple of $v_0$ (respectively $w_0$). As $f$ does not contract any curve on $Y_2$,  it follows that $f$ does not contract any curve on $\shY_\eta$. Hence, we have  $Y_1\cup Y_2 = \Exz(f)$. It follows from Corollary~\ref{corollary qfact} that $\shZ$ is $\QQ$-factorial. So $f$ is a cusp model and $C(f)$ a cuspidal cone. As before one shows that $C(f)$ is the unique cuspidal cone contained in $\Nef(\sY)$.
	
The next case  $\Gamma_{Y_1}=\{v_0\}$, $\Gamma_{Y_3}=\{w_0\}$ is similar: as 
	$D_{21}^2=D_{23}^2<0$ it is straight forward to construct a nef divisor $L_2$ on 
	$Y_2$ such that $L_2.C=0$ if and only if $C\in\{D_{21}, D_{23}\}$. Extending 
	$L_2$ by the trivial divisors on $Y_1$ and $Y_3$, as above one obtains a 
	morphism  $f\colon\shY\to \shZ$. By construction, $Y_1\cup Y_3 = \Exz(f)$ and one 
	concludes that $f$ is a cusp model and $C(f)$ a cuspidal cone. 
	A similar proof 
	with $L_2$ chosen to be  a nef divisor  such that 
	$L_2.C=0$ if and only if $C=D_{21}$ shows the 
	 remaining case, where $|\Gamma_{Y_1}|=1$, $|\Gamma_{Y_3}|=2$.
 \end{proof}

\section{Counting cones}\label{section counting cones}

So far, we have examined curve structures of models of the DNV family in degree $2$ that admit cusp models. Building on this, we will in the present section classify the models and use the classification to count cuspidal cones.

\begin{proposition}\label{proposition classification of models class p}
Let $\shY\to \sY'$ be a cusp model such that $\shY_c=Y_1\cup Y_2 \cup Y_3$ has dual intersection complex $\mathscr{P}$. Suppose $|\Gamma_{Y_i}|=2$ for $i=2,3$ (and hence that $Y_2$, $Y_3$ are contracted under $\sY\to\sY'$). Then $\sY$ is uniquely determined by the regularity type of the $Y_i$ and $n:=D^2_{12}$. More precisely:
\begin{enumerate}
\item If $\Gamma_{Y_i}$ is regular for all $i=1,2,3$, then $\sY$ is uniquely determined by $n$ and there is one such model precisely for every $n\in [-7,1]$. 
\item If $\Gamma_{Y_i}$ is very degenerate for some $i$, then $\sY$ is uniquely determined by $n$ and there is one such model precisely for every $n\in \{-2,-6\}$.
\end{enumerate}
\end{proposition}
\begin{definition}
In the situation of the proposition, we write $\shY_{\operatorname{R}}(n):=\sY$ if all components $Y_i$ are regular and $\sY_\VD(n):=\sY$ if there is a very degenerate $Y_i$.
\end{definition}
\begin{proof}[Proof of Proposition~\ref{proposition classification of models class p}.]
As $|\Gamma_{Y_i}|=2$ for $i=2,3$, the components $Y_2, Y_3$ are degenerate by Corollary~\ref{corollary degenerate}. From projectivity of $\sY_c$ and \cite[Proposition~4.2]{HL} we deduce that $\Gamma_{Y_1}$ is non-degenerate. Write $n_1=D_{12}^2$ and $n_2=D_{13}^2$. Using non-degeneracy, we deduce $n_1\leq 1$ from \cite[Proposition~3.18]{HL}. Let $v_i$, $i=1,2$ be the 
exceptional vertices, with $v_1.D_{12}=1$. Write $l_i=|L(v_i)|-4$ where $L(v_i)$ is the leg of $v_i$. From the condition 
$|\Gamma_{Y_2}|=|\Gamma_{Y_3}|=2$ we get $|\Gamma_{Y_1}|=20$. By the definition of curve 
structure of type $d_2$, we have $|L(v_1)|+|L(v_2)|=|\Gamma_{Y_1}|-2$. So $l_1+l_2=12$. 
As a leg of any exceptional vertex has at least $2$ vertices, we have $l_i \in 
[-2,14]$.  
 Also, we have $n_i=-1-l_i$. So we have $n_2=-14-n_1$, and $n_1\in [-15,1]$. 
 Note that $n_1$ determines all self-intersection numbers $D_{ij}$. 
 Hence, as soon as we 
 know the regularity of $\Gamma_{Y_2}$ and $\Gamma_{Y_3}$, we have a complete 
 description of $\shY$ thanks to \cite[Proposition~3.18]{HL}. 
 Assuming all curve structures are regular, after possibly 
 changing the indexing of $Y_2$ and $Y_3$ of $\shY_c$ we find $
 \shY\cong\shY_{\operatorname{R}}(n)$ for $n\in [-7,1]$, where we have taken into account that we count unordered pairs $\{n_1,n_2\}$ with $n_1+n_2=-14$. 

By the general yoga of curve structures one shows that if $n_1\notin\{-2,-6,-8,-12\}$, all $\Gamma_{Y_i}$ are regular. If $n_1\in\{-2,-6,-8,-12\}$, there can be very degenerate curve structures. Note that by changing the indexing, we can reduce to the case $n_1\in \{-2,-6\}$. So assume $ \Gamma_{Y_2}$ is very degenerate. Then, $D^2_{23}\in \{4,-6\}$ so  $\Gamma_{Y_3}$
is regular.

The existence is straightforward to show by constructing a suitable sequence of elementary modifications of type I starting with $\YP$. 
\end{proof}

Let us now turn to models with dual intersection complex $\scrT$. It is also straightforward to show that for  $(n_1,n_2)\in \{ (-8,-8), (-8,-9), (-9,-9)\}$ there is a unique model $\shY(n_1,n_2)$ of the DNV family of degree $2$ such that the central fiber $\shY(n_1,n_2)_c$ has dual intersection graph $\mathscr{T}$ and, if $Y_3$ is the special component, $D_{3i}^2=n_i$. 
 These are precisely the models such that $|\Gamma_{Y}|\leq 2$ for any smooth component. We state this as a proposition.

\begin{proposition}\label{proposition classification of models class t}
Let $\shY\to \sY'$ be a cusp model such that $\shY_c$ has dual intersection complex $\mathscr{T}$. 
Then $\shY\cong\shY(n_1,n_2)$ for  $(n_1,n_2)\in \{ (-8,-8), (-8,-9), (-9,-9)\}$.\qed
\end{proposition}

As a last ingredient for counting cones, we show that in degree $2$, every cuspidal cone is contained in a unique maximal cone, cf. Remark~\ref{remark cuspidal cones}.

\begin{lemma}\label{lemma unique}
Assume $2d=2$. Let $\pi:\shY\to \shY'$ be a cusp model. Then $\sY'$ determines the model $\sY$ of the DNV family uniquely.
\end{lemma} 
\begin{proof}
Suppose that $\sY \to \sY'$ and $\sX\to\sY'$ were two such models of the DNV family. 
Let $B$ be the pullback along $\sY \ratl \sX$ of an ample prime divisor on $\sX$. As $B$ is $\pi$-nef if and only if $B+\pi^*L$ is nef for a sufficiently ample divisor $L$ on $\sY'$ and clearly $B+\pi^*L \in f^*\Amp(\sX)$, we deduce that for $B$ to be $\pi$-nef, we must have that $\sY \ratl \sX$ is an isomorphism by Lemma~\ref{lemma cones determine contractions}. Let us assume this is not the case. Then we run a log MMP for the pair $(\sY,\veps B)$, where $\veps$ is small enough in order to make the pair klt, and obtain a sequence of flops over $\sY'$ connecting $\sY$ to $\sX$. Thus, we may reduce to the case where $\sY \ratl \sX$ is a flop. Now the lemma follows from Remark~\ref{remark noflops}.
\end{proof}

Thus it makes sense to speak of a cusp model $\sY'$ of class $\scrG\in \{\scrP,\scrT\}$, meaning that the unique model $\sY$ of the DNV family admitting a regular contraction $\sY \to \sY'$ is of class $\scrG$.

\begin{theorem}\label{thm:count}
Let $\shY$ be a model of the DNV family of degree $2$. There are $93$ cuspidal cones in $\Morifan(\sY)$. Moreover, $81$ of these correspond to cusp models of class $\scrP$ and $12$ to cusp models of class $\scrT$.
\end{theorem}
\begin{proof}
Note first by Lemma~\ref{lemma unique} each cuspidal cone is contained in exactly one maximal cone of $\Morifan(\sY)$ so we are reduced to counting marked models of the DNV family which admit a cusp model. For a model $\sX\to S$ of the DNV family, we denote by $\ell_\sX$ the orbit length of $\Nef(\sX)$ under the action of the group $\Bir(\sY)$ 
of birational automorphisms and by $c_\X$ the number of cuspidal cones in $\Nef(X)$. The number of cuspidal cones in $\Morifan(\sY)$ is then given by $\sum_\sX c_\sX \cdot \ell_\sX$ where the sum runs over all (non-marked) models $\sX$ of the DNV family. 

Let us first count the models of class $\scrP$. Among those, models with $|\Gamma_Y|=2$ for every contracted component $Y$ are precisely given by the models $\shY_{\operatorname{R}}(n)$ for $-7 \leq n \leq 1$ and $\shY_{\operatorname{VD}}(n)$ for $n=-2$ or $-6$ by Proposition~\ref{proposition classification of models class p}. By \cite[Proposition 5.43]{HL}, the orbit length of $\shY_{\operatorname{R}}(n)$ is $6$ if $n\neq -7$ and $3$ if $n=-7$ and it is $6$ for the models $\shY_\VD(n)$. Also, one reads off from the curve structures that $\shY_c$ has a component isomorphic to $\PP^1\times \PP^1$ if and only if $\shY\cong\shY_{\operatorname{R}}(n)$ for $n\in\{-6,-2\}$. So by Propositions~\ref{proposition cuspidal cones model p two regular} and~\ref{proposition cuspidal cones model p one regular}, the nef cones of these last models contain $2$ cuspidal cones while all the other models contain $1$ such cone.  There is also the unique model characterized by having a component isomorphic to $\PP^2$, as in Proposition~\ref{proposition p2}. It has orbit length $6$ and a unique cuspidal cone inside its nef cone. 
Similarly, the two models $\shY_\VD(n)$, $n\in \{-2,-6\}$ have orbits of length $6$ and their nef cones contain a unique cuspidal cone. This gives another $12$ cuspidal cones. So we have 
$$
6\cdot 6 + 6\cdot 2 \cdot 2+ 3 \cdot 1 + 6 \cdot 1 + 6\cdot 2 = 81
$$
cuspidal cones defined by cusp models $\shY\to \sY'$ with $\shY_c$ having dual intersection complex $\mathscr{P}$. 

By Proposition~\ref{proposition classification of models class t}, there are $3$ models  $\shY(n_1,n_2)$ of the DNV family with $\shY_c$ having dual intersection complex $\mathscr{T}$. Two of these are symmetric, i.e. two components of $\sY_c$ are isomorphic (see \cite[Definition~5.40]{HL}), while $\shY(-8,-9)$ is not. Each cone $\Nef(\shY(n_1,n_2))$ contains a unique cuspidal cone by Proposition~\ref{modelT:unique}. 
So there are $6\cdot 1 + 3\cdot 2=12 $ cuspidal cones defined by cusp models $\shY\to \sY'$ with $\shY_c$ having dual intersection complex $\mathscr{T}$. Summing up, there are  $93=81+12$ cuspidal cones.
\end{proof}

By the $S_3$-symmetry of the set of cuspidal cones, we deduce our main result, Theorem~\ref{theorem main}.

\begin{corollary}\label{corollary morifan cusp degree 2}
Let $\sY'\to S$ be a cusp model. Then $\Morifan(\sY')$ has $31$ maximal cones, $27$ of these correspond to cusp models of class $\scrP$ and $4$ to cusp models of class $\scrT$. In particular, the toric fan $\mghks[2]$ has $31$ maximal cones inside the fundamental domain of the Coxeter fan. 
\end{corollary}
\begin{proof}
For every irreducible component $Y_i \subset \sY_i$ we choose a  rational $Y_i$-cusp model $\sY \ratl \sY_i$ for $i=1,2,3$.
By Corollary~\ref{corollary cuspidal cones}, the Mori fans $\Morifan(\sY_i)$ are in bijective correspondence via the $S_3$ action. Moreover, thanks to Proposition~\ref{proposition picard noncontracted} a given model $\sX \to S$ of the DNV family in degree $2$ cannot admit cusp models for different components of the central fiber. Thus, cuspidal cones for different components do not intersect: $\Cusp_{Y_i}\cap \Cusp_{Y_j} =\emptyset$ for $i\neq j$. The claim now follows from Theorem~\ref{thm:count}.
\end{proof}

\begin{corollary}\label{corollary morifan cusp degree 2 orbits}
The $\Gammabar[2]$-action induces a residual $S_3$-action on 
the set of maximal cones of the toric fan $\mghks[2]$ inside the fundamental domain of the Coxeter fan with $17=14+3$ orbits of class $\scrP$ and $\scrT$, respectively.
\end{corollary}
\begin{proof}
This follows from the counting arguments given in the proof of Theorem \ref{thm:count}. Taking the orbit length into account, we obtain $6+2\cdot 2 + 1 +1 + 2 =14$ maximal cones of class $\scrP$ and $1+2=3$ maximal cones of class $\scrT$.
\end{proof}

\bibliographystyle{myamsalpha}
\bibliography{morifanglobal}
\end{document}